%% file: GrowingBGW.tex
\documentclass[12pt,a4paper]{article}

\title{Growing conditioned BGW trees with\\ log-concave offspring distributions}
\author{
		William FLEURAT%
		\thanks{Université Paris-Saclay, France \hfill  \texttt{\href{mailto:william.fleurat@universite-paris-saclay.fr}{william.fleurat@universite-paris-saclay.fr}}}
	}

\input{preamble/packages}
	\numberwithin{equation}{section}

\graphicspath{{figures/}}

\input{preamble/macros}

\input{preamble/environments}

\begin{document}

\maketitle


\begin{abstract}
We show that given a log-concave offspring distribution, the corresponding sequence of Bienaymé--Galton--Watson trees conditioned to have $n\geq 1$ vertices admits a realization as a Markov process $(\T_n)_{n\geq1}$ which adds a new ``right-leaning'' leaf at each step.
This applies for instance to offspring distributions which are Poisson, binomial, geometric, or any convolution of those.
By a negative result of Janson, the log-concavity condition is optimal in the restricted case of offspring distributions supported in $\{0,1,2\}$.
We then prove a generalization to the case of an offspring distribution supported on an arithmetic progression, if we assume log-concavity along that progression.

As an application, we deduce the existence of increasing couplings in an inhomogeneous model of random subtrees of the Ulam--Harris tree. This is equivalent to the statement that, in a corresponding inhomogeneous Bernouilli percolation model on a regular tree, the root cluster is stochastically increasing in its size.

These results generalize a construction of Luczak and Winkler which applies to uniformly sampled subtrees with $n$ vertices of the infinite complete $d$-ary trees.
Our proofs are elementary and we tried to make them as self-contained as possible.
\end{abstract}

\vspace{2.5em}
\begin{figure}[h!]
	\centering
	\includegraphics[scale=.75, page=4]{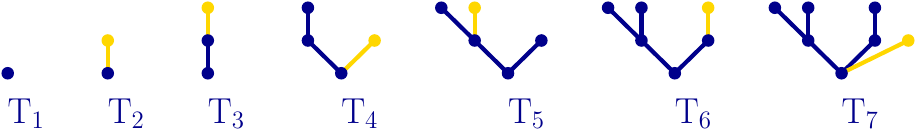}
	\caption{An increasing sequence of rooted plane trees as obtained in Theorem~\ref{thm:main-thm-BGW}.
	The root is at the bottom and the children of a vertex are ordered left-to-right.
At each step, a new \textit{right-leaning} leaf is added.}
	\label{fig:cover-1}
\end{figure}

\vfill\pagebreak
{%
	\hypersetup{linktocpage}
	\tableofcontents
}
\vfill\pagebreak

\section{Introduction}
\label{sec:intro}
\input{parts/intro}


\section{Growing random compositions}
\label{sec:coupling-random-compositions}
\input{parts/coupling-random-compositions}


\section{Growing conditioned BGW trees I: the non-arithmetic case}
\label{sec:application-to-trees}
\input{parts/application-to-trees}


\section{Growing conditioned BGW trees II: the arithmetic case}
\label{sec:arithmetic-case}

\input{parts/arithmetic-case}


\section{Application to a model of random subtrees}
\label{sec:application-random-subtrees}

\input{parts/application-random-subtrees}


\appendix
\section{Decoration-dependent shufflings of decorated plane trees}
\label{app:appendix-shuffling}

\input{parts/appendix-shuffling}


%

\printbibliography

\end{document}

%% file: preamble/packages.tex
\usepackage[utf8]{inputenc}
\usepackage[english]{babel}
\usepackage{csquotes}
\usepackage[margin=3.2cm]{geometry}
\usepackage{xspace}
\usepackage{amsmath, amssymb,
	amstext,
	amsthm,
	mathtools, 
	mathrsfs,
	stmaryrd
}
\usepackage{dsfont}	
	
\usepackage{xifthen}
\usepackage{textgreek}
\usepackage{enumerate}
\usepackage{enumitem}
\usepackage{indentfirst}
\usepackage{graphicx}
\usepackage{xcolor}
	\definecolor{MyRed}{HTML}{BC272D}
	\definecolor{MyBlue}{HTML}{0000FF}
\usepackage{bbm}
\usepackage[
colorlinks=true,
allcolors=MyRed,
bookmarksnumbered]{hyperref}	
\hypersetup{
	citecolor = MyBlue
}
\usepackage[toc,page]{appendix}
\usepackage{array}
\usepackage{threeparttable}

\usepackage[final]{microtype}

\usepackage{upgreek}

\usepackage{tikz-cd}

\usepackage{lipsum}

\usepackage[
	backend=biber,
	hyperref=true,
	style=ext-alphabetic,
	url=true,
	isbn=false,
	giveninits=true,
	maxnames=3,
	backref=false,
	date=year,
	abbreviate=true,
	eprint=true,
	articlein=false,
	maintitleaftertitle=true
	]{biblatex}
	\AtEveryBibitem{%
		\ifentrytype{article}{
			\clearfield{url}%
			\clearfield{urlyear}%
		}{}
		\ifentrytype{book}{
			\clearfield{url}%
			\clearfield{urlyear}%
			\clearfield{pages}%
			\clearfield{doi}%
		}{}
	}

%% file: preamble/macros.tex
\newcommand{\R}{{\mathbb R}}
\newcommand{\Z}{{\mathbb Z}}
\newcommand{\N}{{\mathbb N}}
\newcommand{\indic}[1]{{\mathbbm 1}_{#1}}

\newcommand{\supp}{\mathrm{supp}}
\newcommand\diff{\mathop{}\!\mathrm{d}}

\let\P\relax\DeclareMathOperator{\P}{\mathbf{P}}
\DeclareMathOperator{\E}{\mathrm{E}}
\newcommand{\Expect}[2][]{%
	\E_{#1}\mathopen{}\mathclose\bgroup\left[\,#2\,\aftergroup\egroup\right]}
\newcommand{\condExpect}[3][]{%
	\E_{#1}\mathopen{}\mathclose\bgroup\left[\,#2\;\middle\vert\;#3\,\aftergroup\egroup\right]}
\newcommand{\Prob}[2][]{\P_{#1}\mathopen{}\mathclose\bgroup\left(\,#2\,\aftergroup\egroup\right)}
\newcommand{\condProb}[3][]{%
	\P_{#1}\mathopen{}\mathclose\bgroup\left(\,#2\;\middle\vert\;#3\,\aftergroup\egroup\right)}
\newcommand{\anyProb}[2]{#1\mathopen{}\mathclose\bgroup\left(\,#2\,\aftergroup\egroup\right)}
\newcommand\cdf{{F}}
\newcommand{\Ecal}{\mathcal{E}}
\newcommand{\Fcal}{\mathcal{F}}
\newcommand\Xcal{\mathcal{X}}

\newcommand\mfrak{{\mathfrak{m}}}

\newcommand{\probStyle}[1]{\mathbf{#1}}
\newcommand{\Unif}{\probStyle{Unif}}
\newcommand{\Binom}{\probStyle{Binom}}
\newcommand{\Poisson}{\probStyle{Poisson}}
\newcommand{\Geom}{\probStyle{Geom}}
\newcommand{\PComp}[2]{\probStyle{P}_{#2}^{\,{#1}}}
\newcommand{\ProbComp}[3]{%
	\PComp{#1}{#2}\mathopen{}\mathclose\bgroup	\left(\,#3\,\aftergroup\egroup\right)}
\newcommand{\BGW}[1]{\probStyle{BGW}^{#1}}
\newcommand{\BGWcond}[2]{\BGW{#1}_{#2}}
\newcommand{\SimpGen}[2]{\probStyle{SG}^{#1}_{#2}}
\newcommand{\Sub}[2]{\probStyle{ST}^{#1}_{#2}}
\newcommand{\randSubset}[2]{\probStyle{B}^{#1}_{#2}}

\newcommand{\weight}[1]{\mathbf{#1}}
\newcommand{\wa}{{\weight{a}}}
\newcommand{\wb}{{\weight{b}}}
\newcommand{\w}{{\weight{w}}}
\newcommand{\W}{{\weight{W}}}
\newcommand\wtheta{{\boldsymbol{\uptheta}}}
\newcommand\wthetai{{\wtheta^{(i)}}}

\newcommand{\e}{{\weight{e}}}
\newcommand{\shiftNotation}[2]{{#2}^{+#1}}\long\def\gobbleone#1{}
\newcommand{\shift}[2]{\if\relax\detokenize\expandafter{\gobbleone#1}\relax\shiftNotation{#1}{#2}\else\shiftNotation{(#1)}{#2}\fi}

\newcommand{\PFComp}[2]{{Z}^{#1}_{#2}}
\newcommand{\PFtrees}[2]{b^{#1}_{#2}}
\newcommand{\PFforests}[2]{f^{#1}_{#2}}

\newcommand{\PFsubtrees}[2]{ST^{#1}_{#2}}

\newcommand{\U}{\mathbb{U}}
\newcommand{\Ud}[1]{\U^{(#1)}}
\renewcommand{\u}{\mathtt{u}}
\renewcommand{\v}{\mathtt{v}}
\newcommand{\trees}[1][]{\ifthenelse{\isempty{#1}}{\mathbb{T}}{\mathbb{T}^{[#1]}}}
\newcommand{\forests}[1][]{\ifthenelse{\isempty{#1}}{\mathbb{F}}{\mathbb{F}^{[#1]}}}
\newcommand{\subtrees}{\mathbbm{t}}
\newcommand{\dsubtrees}{\subtrees^{(d)}}
\newcommand{\dcomp}{\smash{\widehat{\mathbbm{T}}}^{(d)}}
\newcommand{\pt}{T}
\newcommand{\T}{\mathrm{T}}
\renewcommand{\t}{\tau}
\newcommand{\subT}{\mathcal{T}}
\newcommand{\dsubT}{\subT}

\newcommand{\pos}{{{C}}}
\newcommand\Pos{{\boldsymbol{C}}}

\newcommand\G{\mathbb{G}}
\newcommand\bbXseq{\bbX_{\rm seq}}
\newcommand\g{g}
\newcommand\boldg{\boldsymbol{g}}
\newcommand\boldh{\boldsymbol{h}}

\newcommand\boldp{\boldsymbol{p}}

\newcommand{\x}{{\boldsymbol{x}}}

\newcommand{\treesgr}{\trees^{\rm gr}}

\newcommand{\perm}{\mathfrak{S}}
\newcommand{\bbX}{\mathbb{X}}
\newcommand{\bbY}{\mathbb{Y}}
\newcommand{\boldsigma}{{\boldsymbol{\sigma}}}

\newcommand{\boldpi}{{\boldsymbol{\uppi}}}
\newcommand{\boldS}{{\boldsymbol{S}}}

\newcommand{\Parts}{\mathbb{S}}

\newcommand{\Compos}[2][]{\mathrm{Comp}_{#1}(#2)}
\newcommand{\letter}[1]{\underline{#1}}
\newcommand{\composes}[1][]{\:\models_{#1}\:}
\newcommand{\projParts}{\mathrm{proj}}
\newcommand{\precdot}{\prec\mathrel{\mkern-5mu}\mathrel{\cdot}}
\newcommand{\covered}[1][]{\precdot^{#1}}

\newcommand\pf{_*\,}
\newcommand{\treeOperations}[1]{\mathsf{#1}}
\newcommand{\push}{\treeOperations{push}}
\newcommand\epush{\mathfrak{P}}

\newcommand{\embed}{{\epush^{-1}}}
\newcommand{\comp}{\treeOperations{comp}^{(d)}}

\renewcommand{\emptyset}{\varnothing}
\renewcommand{\phi}{\varphi}
\renewcommand{\epsilon}{\varepsilon}
\newcommand{\where}{\qquad\text{where}\qquad}

%% file: preamble/environments.tex
	\theoremstyle{plain}
\newtheorem{thm}{Theorem}
\newtheorem*{thm*}{Theorem}
\newtheorem{prop}{Proposition}[section]

\newtheorem{lemma}[prop]{Lemma}

\newtheorem{pb}{Problem}
\newtheorem{cor}[prop]{Corollary}
	\theoremstyle{definition}
\newtheorem{defin}[prop]{Definition}

	\theoremstyle{remark}
\newtheorem{rem}[prop]{Remark}

\newcommand{\mypar}[1]{\paragraph{{{#1}}}}

%% file: parts/intro.tex
\begin{figure}
	\centering
	\includegraphics[scale=.75, page=7]{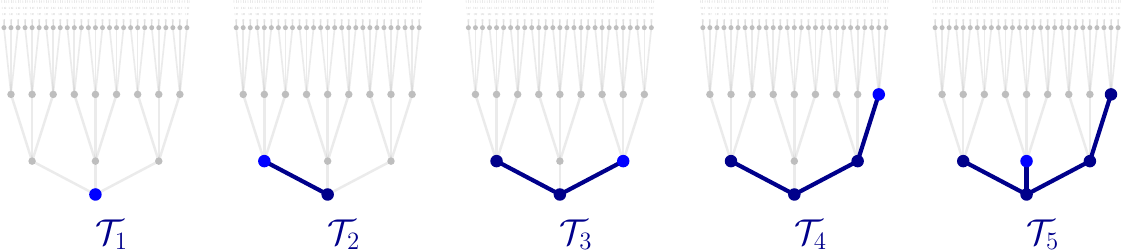}
	\caption{An increasing sequence of rooted subtrees of the infinite complete $d$-ary tree, as in Luczak and Winkler's coupling, here with $d=3$.}
	\label{fig:Luczak-Winkler}
\end{figure}%
Twenty years ago, Luczak and Winkler \cite{LuczakWinkler04}  showed that for every $d\geq 2$, the uniform distributions on rooted subtrees with $n$ vertices, $n\geq1$, of the infinite complete $d$-ary tree $\Ud d$ admit a coupling%
	\footnote{We recall that a \textit{coupling} of a collection of probability distributions $(\mathcal L_i)_i$ is a collection of random variables $(X_i)_i$ on a probability space $(\Omega,\Fcal,\P)$ such that for every $i$, the random variable $X_i$ has law $\mathcal L_i$.}
as an increasing process $(\dsubT_n)_{n\geq 1}$, in the sense that for this process we have the inclusions 
\begin{align*}
\dsubT_1\subset\dsubT_2\subset\dsubT_3\subset\dots\subset\Ud d.
\end{align*}

This says in a precise sense that these distributions of random trees form a sequence which is stochastically increasing.
This result has been applied to various problems relating to the number of spanning trees in the Erdös--Rényi random graphs \cite{LyonsPeledSchramm08}, stochastic ordering for infinite Bienaymé--Galton--Watson trees \cite{LyonsPeledSchramm08,Broman14,Broman16}, the parking problem on random trees \cite{GoldschmidtPrzykucki19}, growth procedures for random planar maps \cite{Addario-Berry14,CaraceniStauffer23}, or the edge flip chain on quadrangulations \cite{CaraceniStauffer20}.
In the case $d=2$, the construction has been made more explicit in \cite{CaraceniStauffer20}, and  a detailed analysis of its asymptotic behaviour has been performed in \cite{CaraceniCurienStephenson24}, uncovering multifractal properties of the random measure which describes where the tree $\subT_n$ grows new vertices for every $n$.

Among the most natural and well-studied models of random trees reside \textit{Bienaymé--Galton--Watson trees} conditioned on having $n$ vertices, $n\geq1$, with some fixed offspring distribution $\mu$.
It would be desirable that such sequences of random trees be stochastically increasing, too.
Unfortunately that is not the case in general, as was shown by Janson \cite{Janson06}. 
Still, some partial results in this direction can be deduced from Luczak and Winkler's couplings, as we shall recall.

In this paper we give a handy sufficient condition relying on the notion of \textit{log-concavity}, for the existence of increasing couplings for conditioned Bienaymé--Galton--Watson trees.
We then give an application to increasing couplings in a natural inhomogeneous model of random subtrees.

\subsection{Context}
\label{sec:results-LW}

All our trees will be represented as subsets of the Ulam--Harris tree $\U$,
\begin{align*}
\U=\bigcup_{h\geq0}\{1,2,\dots\}^h.
\end{align*}
The empty word is denoted by $\emptyset$, which is the \textit{root vertex} of a natural \textit{rooted tree} structure underlying the set $\U$: we declare that a \textit{vertex} $\u\in\U$ is an ancestor of $\v\in\U$ if $\u$ is a prefix of $\v$.
Related genealogical notions such as \textit{parents}, \textit{siblings}, \textit{children} follow.
For $\u=(u_1,\dots,u_h)\in \U$, the integer $h$ is called the \textit{height} of $\u$, and its \textit{ancestral line} is the sequence of vertices $(\u_0,\u_1,\dots,\u_h)$, where $\u_0=\emptyset$ and $\u_\ell=(u_1,\dots,u_\ell)$, $1\leq \ell\leq h$.
For $\u,\v\in\U$, we denote by $\u\v$ the \textit{concatenation} of the words $\u$ and $\v$,
and we write $\u i$ instead of $\u(i)$ when $\v$ is the singleton $(i)$.
In particular, a \textit{child} of $\u\in\U$ is a vertex $\u i$ for $i\in\{1,2,\dots\}$.

\mypar{Plane trees and BGW measures}
In Neveu's formalism \cite{Neveu86}, a \textit{plane tree} is a non-empty finite subset $\pt\subset\U$ which is closed under ``taking parents'' and ``taking siblings to the left'', or in other words which is such that:
\begin{align*}
\forall\u\in\U,\,\forall i\in\{1,2,\dots\},	\qquad
\u i\in\pt
\quad\implies\quad
\begin{cases}
\u\in\pt,\\
\u j\in \pt	& 1\leq j\leq i.
\end{cases}
\end{align*}
The set of plane trees is denoted by $\trees$ and we let $\trees_n$ be the subset of plane trees with $n$ vertices, $n\geq1$.
For $\pt\in\trees$ and $\u\in \pt$, we let $k_\u(\pt)$ be the number of children of $\u$ which are in $\pt$.
A \textit{leaf} of $\pt$ is a vertex with no children in $\pt$.

Given some probability distribution  $\mu$ on%
	\footnote{
		In this work we let $\Z_+=\{0,1,2,\dots\}$ and $\N=\{1,2,3,\dots\}$.
	}
$\Z_+=\{0,1,2,\dots\}$, with $0<\mu(0)<1$, the corresponding Bienaymé--Galton--Watson measure, or the $\mu$-BGW measure for short, is defined by:
\begin{align}\label{eq:def-BGW}
\forall \pt\in \trees,\quad
\BGW{\mu}(\pt)=\prod_{\u\in \pt}\mu\bigl(k_\u(\pt)\bigr).
\end{align}
This is a probability measure on $\trees$ when $\mu$ has mean at most one, and a sub-probability measure otherwise.
For $n\geq1$, as long as $\trees_n$ gets non-zero weight under the $\mu$-BGW measure, we can condition on the event $\{\#T=n\}$ to define the \textit{probability} distribution $\BGWcond{\mu}{n}$ on $\trees_n$.

\mypar{Rooted subtrees}
More generally, a \textit{rooted subtree} of $\U$ is a non-empty finite subset $\t$ of $\U$ which is closed under ``taking parents'' only, that is:
\begin{align*}
\forall\u\in\U,\,\forall i\in\{1,2,\dots\},	\qquad
\u i\in\t
\qquad\implies\qquad
\u\in\t.
\end{align*}
The set of rooted subtrees of $\U$ is denoted by $\subtrees$ and we let $\subtrees_n$ be the subset of rooted subtrees with $n$ vertices for $n\geq1$.

\mypar{Luczak and Winkler's couplings}
For every $d\geq2$, the infinite complete $d$-ary tree $\Ud d$ and its rooted subtrees can be viewed as subsets of $\U$:
\begin{align*}
\Ud d=\bigcup_{h\geq0}\{1,2,\dots,d\}^h,
\qquad\text{and}\qquad
\dsubtrees_n=\{\t\in\subtrees_n\colon \t\subset \Ud d\},\quad n\geq1.
\end{align*}
The main result in \cite{LuczakWinkler04}, their Theorem~4.1, states that for $d\geq 2$, the distributions%
	\footnote{
		Given a finite set $S$, we denote by $\Unif(S)$ the uniform probability distribution on it.
	}
$\Unif(\dsubtrees_n)$, $n\geq1$, admit a coupling $(\dsubT_n)_{n\geq1}$ which is an increasing process with respect to the inclusion order $\subseteq$.
This also gives the existence of increasing couplings for some related models which we shall now review.
A summary is presented in Table~\ref{table:Luczak--Winkler}.

\begin{figure}
	\centering
	\includegraphics[scale=.75, page=2]{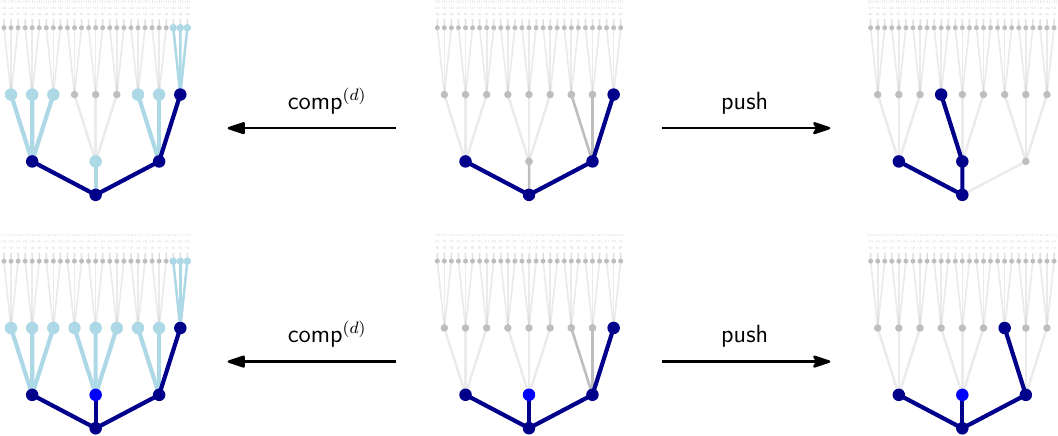}
	\caption{An illustration of the action of the mappings $\push$ and $\comp$, here with $d=3$.
	Notice that $\comp$ preserves the order $\subseteq$ while $\push$ does not.}
	\label{fig:Luczak-Winkler-2}
\end{figure}

A plane tree $\pt$ is said to be a \textit{complete $d$-ary tree} if every $\u\in\pt$ has either $d$ or $0$ children in $\pt$.
The set of complete $d$-ary trees with $nd+1$ vertices, $n\geq0$, is denoted by $\dcomp_{nd+1}$.
The Luczak--Winkler couplings imply that the distributions $\smash{\Unif(\dcomp_{nd+1})}$, $n\geq0$, also admit a coupling as an increasing process, with respect to the inclusion order $\subseteq$.
Indeed, given $\t$ a rooted subtree with $n$ vertices, $n\geq1$, of $\Ud d$, $d\geq2$, we may define its \textit{completion} $\pt=\comp(\t)$ by adding all the vertices $\u i$ for $\u\in\t$ and $i\in\{1,\dots,d\}$.
This yields a complete $d$-ary tree with $nd+1$ vertices, and $\t$ is recovered by removing all the leaves of $\pt$.
It is easily seen that for $\mu=(1-\frac1d)\boldsymbol{\delta}_0+\frac1d\boldsymbol{\delta}_d$ we have $\BGWcond{\mu}{nd+1}=\smash{\Unif(\dcomp_{nd+1})}$ for all $n\geq0$.

Rooted subtrees of $\U$ naturally come with an associated plane tree.%
	\footnote{
		Rooted plane trees are often defined as unlabeled vertex-rooted trees, in the sense of graph theory, with an ordering of the children of each vertex.
		The plane tree associated to an arbitrary subtree of $\U$ is then the one where we ``forget'' the names of the vertices, only remembering the relative ordering of the children of each vertex.
	}
In Neveu's formalism, this plane tree can be obtained by ``pushing to the left'' as much as possible the rooted subtree of $\U$.
This induces a mapping $\push\colon\subtrees\rightarrow\trees$, which we describe in detail in Section~\ref{sec:application-random-subtrees}.
Let $d\geq2$.
It is easily seen that the image of $\Unif(\dsubtrees_n)$, $n\geq1$, under this mapping is $\BGWcond{\mu}{n}$ with $\mu=\Binom(d,1/d)$, see for instance Proposition~\ref{prop:comparison-SG-and-ST}.
Hence the Luczak--Winkler coupling is mapped to a coupling of these distributions.
However, the image coupling is no longer increasing with respect to $\subseteq$, but only with respect to a weaker order $\preceq$ defined by $\pt\preceq \pt'\iff\exists \t\subset\t',\pt=\push(\t),\pt'=\push(\t')$.

It was observed by Lyons, Peled and Schramm in \cite{LyonsPeledSchramm08} that by taking $d\rightarrow\infty$ in the preceding coupling of the distributions $(\BGWcond{\mu}{n})_{n\geq1}$ with $\mu=\Binom(d,1/d)$, we get an $\preceq$-increasing%
	\footnote{
		The authors do not mention the order $\preceq$ or the operation $\push$, simply because they do not use Neveu's formalism.
	}
coupling in the $\mu=\Poisson(1)$ case.

In the case $d=2$, the above coupling for the distributions $(\Unif(\dcomp_{nd+1}))_{n\geq0}$ is an $\subseteq$-increasing coupling for uniformly sampled complete binary trees with $n$ interior vertices respectively.
By applying the so-called \textit{rotation correspondence} between complete binary trees and plane trees, we obtain a coupling for uniformly random plane trees with $n$ vertices, $n\geq1$, see also \cite[p.~427]{LuczakWinkler04}.
This coupling is easily seen to be increasing with respect to $\subseteq$.
Uniformly random plane trees with $n$ vertices are well-known to be distributed as $\BGWcond{\mu}{n}$ with $\mu=\Geom(1/2)$.

We recall that for every offspring distribution $\mu$, the laws $\BGWcond{\mu}{n}$, $n\geq1$, are unchanged by exponential tilting%
	\footnote{
		We recall that an \textit{exponential tilt} of a probability distribution $\nu$ on $\Z_+$ is a probability distribution $\widetilde\nu$ on $\Z_+$ of the form $\widetilde\nu(k)={x^k\nu(k)}/({\sum_j x^j\nu(j)})$, $k\geq0$ for some $x>0$ such that $\sum_j x^j\,\nu(j)<\infty$.
	}
of $\mu$. 
Hence the above couplings cover the cases $\Binom(d,p),\Poisson(\lambda),\Geom(p)$ for any $d\geq2$, $\lambda>0$, $p\in(0,1)$ instead of the specific values we gave.

\begin{table}
	\centering
	\begin{tabular}{l l l c }
		\hline
		\textsc{Parameters} & \textsc{Model}& \textsc{Offspring distr.}	& \textsc{Order rel.} \\
		\hline
		$d\geq2$ &  $(\Unif(\dsubtrees_n))_{n\geq1}$ &  N/A & $\subseteq$ \\ 
		$d\geq2$ &  $(\Unif(\dcomp_{nd+1}))_{n\geq0}$ &  N/A &	$\subseteq$\\
		$d\geq2$ & 	$(\BGWcond{\mu}{nd+1})_{n\geq0}$ &  $\mu=(1-\frac1d)\boldsymbol{\delta}_0+\frac1d\boldsymbol{\delta}_d$	& $\subseteq$\\
		$d\geq2$, $p\in(0,1)$ &	$(\BGWcond{\mu}{n})_{n\geq1}$ & $\mu=\Binom(d,p)$ & $\preceq$ \\
		$\lambda>0$ & $(\BGWcond{\mu}{n})_{n\geq1}$ & $\mu=\Poisson(\lambda)$ & $\preceq$ \\
		$p\in(0,1)$	& $(\BGWcond{\mu}{n})_{n\geq1}$ & $\mu=\Geom(p)$		& $\subseteq$\\
		\hline
	\end{tabular}
	\caption{Summary of the models of random trees for which the results of \cite{LuczakWinkler04} allow to build increasing couplings.
		The right-hand column specifies with respect to which order relation the resulting couplings are increasing.
		The second and third line describe the same model.}
	\label{table:Luczak--Winkler}
\end{table}

\mypar{Janson's negative result}
In \cite{Janson06}, Janson investigated several conditions related to the existence of increasing couplings for conditioned BGW trees.
In particular, it is justified in Section~3 therein that if for $\epsilon\in (0,1)$ one considers the offspring distribution $\mu$ supported on $\{0,1,2\}$ such that
\begin{align}\label{eq:Janson-counter-ex}
\mu(0)=\frac{1-\epsilon}{2},\qquad \mu(1)=\epsilon,\qquad \mu(2)=\frac{1-\epsilon}{2},
\end{align}
then when $\epsilon<1/3$, the expected number of children of the root vertex is strictly larger under $\BGWcond{\mu}{3}$ than under $\BGWcond{\mu}{4}$.
Hence, the distributions $(\BGWcond{\mu}{n})_{n\geq1}$ cannot be coupled increasingly with respect to $\subseteq$ nor with respect to $\preceq$.

\subsection{Main results}

Central to our results will be the notion of log-concavity, which
has a rich history in combinatorics, algebra and geometry, see the classical surveys \cite{Stanley89,Brenti89,Brenti94}, as well as the more recent ones \cite{Branden15,Huh18}.
A non-negative sequence $(x_i)_{i\geq0}$ is said to be \textit{log-concave} if the following conditions are satisfied.
	\begin{enumerate}
		\item The inequality $x_i^2\geq x_{i-1}x_{i+1}$ holds for all $i\geq1$.
		\item It has \textit{no internal zeros}.
		That is if $x_ix_j\neq0$, $i<j$, then $x_ix_{i+1}\cdots x_j\neq0$.
	\end{enumerate}
Some authors do not enforce the second condition when defining log-concavity, but we insist that it will be needed in our case.
Convolutions of log-concave sequences are log-concave, as well as pointwise limits of log-concave sequences.
Basic examples comprise binomial coefficients $(\binom d k)_{k\geq 0}$ with $d\geq2$, geometric sequences $(q^k)_{k\geq0}$ with $q>0$, or the inverse factorials $(1/k!)_{k\geq 0}$.
As another important example, if a polynomial with non-negative coefficients has only real roots then said coefficients are log-concave.
This can be derived using the preservation under convolution, or using Newton's inequalities, which are recalled in Lemma~\ref{lem:Newton-inequalities}.

\mypar{Main result on conditioned BGW trees}
Note that if an offspring distribution $\mu$ has no internal zeros, and $0<\mu(0)<1$, then $\mu(1)>0$.
Hence for $n\geq1$, the tree with $n$ vertices arranged in a line gets non-zero $\mu$-BGW weight and so does $\trees_n$.
The distributions $(\BGWcond{\mu}{n})_{n\geq 1}$ are therefore well-defined.
Our first main result goes as follows.

\begin{figure}
	\centering
	\includegraphics[scale=.75, page=8]{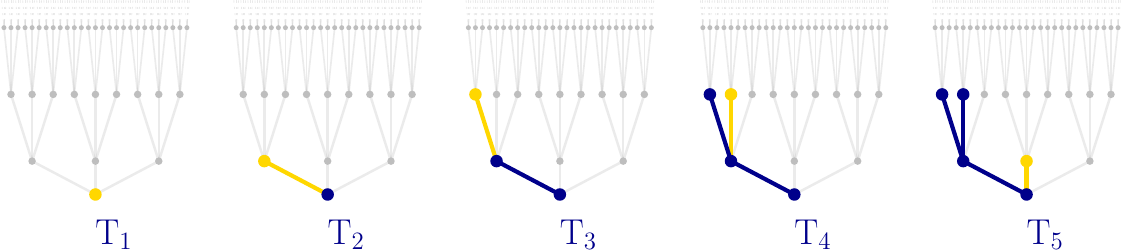}
	\caption{An example of an increasing sequence of rooted plane trees, as in Theorem~\ref{thm:main-thm-BGW}.
		Observe that the new leaves at each step (in yellow) must be \textit{right-leaning}.
	For convenience, we only represented a portion of the Ulam--Harris tree, namely the bottom part of its subset $\Ud 3$.}
	\label{fig:growth}
\end{figure}

\begin{thm}\label{thm:main-thm-BGW}
	Let $\mu$ be a probability measure on $\Z_+$ with $0<\mu(0)<1$.
	If the sequence $(\mu(k))_{k\geq 0}$ is log-concave, then the distributions $(\BGWcond{\mu}{n})_{n\geq 1}$ can be coupled as a Markov process $(\T_n)_{n\geq 1}$ such that $\T_1\subset\T_{2}\subset\T_{3}\subset\dots$.
\end{thm}

Let us make a few comments.
First, notice that in the sequence $(\T_n)_{n\geq 1}$, the new vertex $\u\in\T_{n+1} \setminus\T_n$ which is added at each step must be a \textit{leaf}, which must be \textit{right-leaning}, that is of the form $\u=\v i$ with $i=k_\v(\T_{n})+1$, for some $\v\in\T_{n}$.
This is illustrated in Figure~\ref{fig:growth}.
Since a plane tree is completely determined by its genealogy together with the relative order of the children of each vertex, it is sufficient on a drawing to represent this data, without representing the ``ambient'' Ulam--Harris tree $\U$.
This is what we did to illustrate Theorem~\ref{thm:main-thm-BGW} in Figure~\ref{fig:cover-1}.

The fact that in our coupling the new leaves are \textit{right-leaning} allows a lot of flexibility to make new couplings out of it.
For instance, using the symmetries of the BGW distributions, we can perform random independent permutations of the subtrees of descendants above each node, so as to obtain another coupling of the same conditioned BGW distributions, in which the new leaves are in a controlled position.
This strategy underpins the proof of Theorem~\ref{thm:growing-subtrees} below.

Theorem~\ref{thm:main-thm-BGW} covers notably the cases of \textit{binomial}, \textit{Poisson}, and \textit{geometric} offspring distributions listed in Table~\ref{table:Luczak--Winkler}.
Indeed the binomial coefficients, the inverse factorials, and the geometric sequences all form log-concave sequences.
But since log-concavity is preserved under convolution, Theorem~\ref{thm:main-thm-BGW} also applies to arbitrary convolutions of these distributions, with different parameters.
Note that Theorem~\ref{thm:main-thm-BGW} is already a substantial improvement in the binomial and Poisson cases, since the inclusion order $\subseteq$ is stronger than the order $\preceq$.

Let us mention another example, which is not\footnote{This is easily seen at the level of generating functions, say by looking at their complex roots.} a convolution of binomial, Poisson, or geometric distributions.
The measure $\mu=\Unif(\{0,1,2,\dots,d\})$, $d\geq2$, satisfies the log-concavity assumption, so that Theorem~\ref{thm:main-thm-BGW} applies to the corresponding sequence of conditioned BGW trees, which are uniformly distributed in the set of plane trees with $n$ vertices having at most $d$ children per node.

Lastly, let us insist on the fact that in our definition of log-concavity the condition of having \textit{no internal zeros} is essential.
In fact, the conclusion of Theorem~\ref{thm:main-thm-BGW} never holds when the sequence $(\mu(k))_{k\geq 0}$ possesses an internal zero.%
	\footnote{
		Indeed, let $(\T_n)_{n\geq1}$ be as in the conclusion of Theorem~\ref{thm:main-thm-BGW} and let $k\in\supp(\mu)$.
		Then $\T_{k+1}$ is with non-zero probability the ``$k$-star'' $\{\emptyset\}\cup\{(1),(2),\dots,(k)\}$.
		This forces $\T_{j+1}$ to be a $j$-star for $0\leq j\leq k$, so that $k_\emptyset(\T_{j+1})=j$ with non-zero probability and $j\in\supp(\mu)$.
	}

In the case where the support of $\mu$ is included in $\{0,d,2d,\dots\}$ for some $d\geq2$, it is trivial that the conclusions of Theorem~\ref{thm:main-thm-BGW} do not hold since the conditioned $\mu$-BGW measures are not well-defined when $n$ has residue different from $1$ modulo $d$, see below.
Still, we can adapt our approach and the statement of Theorem~\ref{thm:main-thm-BGW} in this setting, which leads us to our next result.

\mypar{Extension to arithmetic offspring distributions}
Let $\mu$ be a probability distribution on $\Z_+$ whose support is included in the arithmetic progression $d\Z_+=\{0,d,2d,\dots\}$ for some $d\geq2$, and with $0<\mu(0)<1$.
Then the $\mu$-BGW measure is still well-defined, but the conditional distribution $\BGWcond{\mu}{n}$ may only be possibly defined when $n$ has residue $1$ modulo $d$.
Indeed if for $\T\in\trees_n$, $n\geq1$, all vertices in $\T$ have a number of children which is a multiple of $d$, then the $n-1$ non-root vertices can be split in groups of $d$ vertices.

If the sequence $(\mu(kd))_{k\geq 0}$ has no internal zeros, then $\trees_{nd+1}$ does indeed get non-zero weight for all $n\geq 0$, so that the conditional distribution $(\BGWcond{\mu}{nd+1})_{n\geq0}$ are well-defined.
Our generalization of Theorem~\ref{thm:main-thm-BGW} in the arithmetic case goes as follows.

\begin{thm}\label{thm:main-thm-BGW-arithmetic}
	Let $\mu$ be a probability measure on $\Z_+$ with $0<\mu(0)<1$ and whose support is included in $d\Z_+$ for some $d\geq1$.
	If the sequence $(\mu(kd))_{k\geq 0}$ is log-concave, then the distributions $(\BGWcond{\mu}{nd+1})_{n\geq 1}$ can be coupled as a Markov process $(\T_{nd+1})_{n\geq 0}$ such that $\T_1\subset\T_{d+1}\subset\T_{2d+1}\subset\dots$.
\end{thm}

\begin{figure}
	\centering
	\includegraphics[scale=.75, page=9]{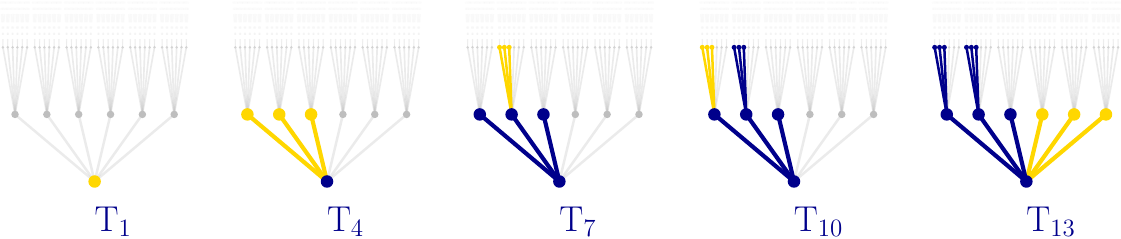}
	\caption{An example of an increasing sequence of rooted plane trees as obtained in Theorem~\ref{thm:main-thm-BGW-arithmetic} with $d=3$.
		Observe that at each step, a new \textit{right-leaning bouquet of $d$ leaves} is added (in yellow).
		For convenience, we only represented a portion of the Ulam--Harris tree, namely the bottom part of its subset $\Ud 6$.}
	\label{fig:growth-arithmetic}
\end{figure}

For a rooted plane tree $\pt$ and an integer $d\geq1$, we call \textit{right-leaning bouquet of $d$ leaves} in $\pt$ a set of $d$ leaves which are siblings and which additionally are the $d$ rightmost vertices among their siblings.
More formally, a right-leaning bouquet of $d$ leaves in $\pt$ is a family of leaves of $T$ which consists in the vertices $\u({k_\u(T)-(d-1)})$ and $\u({k_\u(T)-(d-2)})$ up to $\u({k_\u(T)})$, for some $\u\in T$.
Then, as for Theorem~\ref{thm:main-thm-BGW}, a sequence of plane trees which is increasing with respect to $\subseteq$ must grow ``on the right'', so that in the sequence $(\T_{nd+1})_{n\geq 1}$ of Theorem~\ref{thm:main-thm-BGW-arithmetic}, a new right-leaning bouquet of $d$ leaves is added at each step.
This is illustrated in Figure~\ref{fig:growth-arithmetic}.

Notice that the $d=1$ case of Theorem~\ref{thm:main-thm-BGW-arithmetic} corresponds precisely to Theorem~\ref{thm:main-thm-BGW}.
Theorem~\ref{thm:main-thm-BGW-arithmetic} applies notably when $\mu=\frac{d-1}{d}\delta_0+\frac{1}{d}\delta_d$ for some $d\geq 2$, in which case for all $n\geq 0$ the tree $\T_{nd+1}$ has the distribution of a uniformly sampled complete $d$-ary tree with $n$ interior vertices.
This recovers Luczak and Winkler's main result\footnote{We recall that their statement, which is about subtrees of $\Ud d$, is equivalent to a statement for complete $d$-ary trees using the bijection $\comp$ described above.} from \cite{LuczakWinkler04}, namely their Theorem 4.1.

Theorem~\ref{thm:main-thm-BGW-arithmetic} also applies for instance when $\mu=\mathrm{Unif}(\{0,d,2d,\dots,rd\})$ for some $r\geq1$.
In this case, for every $n\geq0$, the distribution $\BGWcond{\mu}{nd+1}$ is uniform in the set of plane trees with $nd+1$ vertices satisfying that the number of children of each node belongs to the arithmetic progression $\{0,d,2d,\dots,rd\}$.

\mypar{Application to a model of random subtrees}

As a testament to the flexibility we have to generate new couplings from those of Theorems~\ref{thm:main-thm-BGW} and~\ref{thm:main-thm-BGW-arithmetic}, we shall construct increasing couplings for a natural \textit{inhomogeneous} model of random subtrees of the Ulam--Harris tree $\U$.
For background on models of random subtrees of a graph, we refer the reader to the recent survey of Fredes and Marckert \cite{FredesMarckert23} and references therein.

Let $\wtheta=(\theta_1,\theta_2,\dots)$ be a non-negative sequence such that $0<\sum_i \theta_i<\infty$.
For a non-root%
	\footnote{We do not assign a type to the root vertex $\emptyset$.}
vertex $\u\in\U$, we say that it is a vertex of type $i\in\{1,2,\dots\}$ if it is the $i$-th child of its parent, that is $\u=\v i$ for some $\v\in\U$.
For $i\in\{1,2,\dots\}$ and $\t\in\subtrees$, we let $V_i(\t)$ denote the set of $\u\in\t$ of type $i$, and we write $N_i(\t)=\#\, V_i(\t)$.
For every $n\geq1$, we define a measure $\Sub{\wtheta}{n}$ on $\subtrees_n$ as follows:
\begin{align}\label{eq:def-subtree-model}
\forall \t\in\subtrees_n,\quad
\Sub{\wtheta}{n}(\t)=\frac{\prod_{i\geq 1}\theta_i^{N_i(\t)}}{\PFsubtrees{\wtheta}{n}},
\qquad \PFsubtrees{\wtheta}{n}=\sum_{\t\in\subtrees_n}\prod_{i\geq 1}\theta_i^{N_i(\t)}.
\end{align}
Notice that for $d\geq2$, if $\wtheta$ satisfies $\theta_i=0$ for all $i>d$, then the supports of the distributions $\Sub\wtheta n$, $n\geq1$, are included in $\dsubtrees_n$ respectively, so that this yields a model of random subtrees of the complete $d$-ary tree $\Ud d$.
We will deduce from Theorem~\ref{thm:main-thm-BGW} the following statement.

\begin{thm}\label{thm:growing-subtrees}
	Let $\wtheta=(\theta_1,\theta_2,\dots)$ be a non-negative sequence such that $0<\sum_i \theta_i<\infty$.
	The distributions $\Sub{\wtheta}n$, $n\geq 1$, can be coupled as a Markov process $(\subT_n)_{n\geq 1}$ such that $\subT_1\subset\subT_2\subset\subT_3\subset\dots\subset\U$.
\end{thm}
Note that in Theorem~\ref{thm:growing-subtrees}, the new vertices at each step must be leaves, but which are not necessarily right-leaning.
However, our construction will rely crucially on the fact that the couplings in Theorem~\ref{thm:main-thm-BGW} add leaves which are right-leaning.

In the case of the sequence $\wtheta=(1,1,\dots,1,0,0,\dots)$, where the first $d$ terms equal $1$ and the others are zero, notice that for every $n\geq1$, the distribution $\Sub\wtheta n$ is uniform on the set $\dsubtrees_n$ of random subtrees with $n$ vertices of the complete $d$-ary tree $\Ud d$.
In particular, in the case of sequences of the form $\wtheta=(1,1,\dots,1,0,0,\dots)$, we recover \cite[Theorem 4.1]{LuczakWinkler04}, that is the statement that uniformly random subtrees of $\Ud{d}$ can be coupled increasingly.

Roughly, we will observe that the application $\push$ maps the distributions $\Sub\wtheta n$, $n\geq1$, onto some conditioned BGW distributions satisfying the log-concavity assumption in Theorem~\ref{thm:main-thm-BGW}, which yields some coupling $(\T_n)_{n\geq1}$.
The core of the proof of Theorem~\ref{thm:growing-subtrees} then consists in constructing an embedding of the plane trees $(\T_n)_{n\geq1}$ into $\U$ such that their images form an increasing coupling of the distributions $\Sub\wtheta n$, $n\geq1$.

\begin{rem}[Percolation interpretation]
	We let the reader check that given $\wtheta=(\theta_1,\theta_2,\dots)$ a non-negative sequence such that $0<\sum_i \theta_i<\infty$, and $n\geq1$, the probability distribution $\Sub{\wtheta}n$ can also be described as the distribution of the tree forming the root cluster, conditioned on it having $n$ vertices, in the following site percolation model: each non-root vertex $\u\in\U$ of type $i$ is present in the percolation configuration with probability $p_i=\theta_i/(1+\theta_i)$, independently of other vertices.
	In particular, Theorem~\ref{thm:growing-subtrees} states that in such an \textit{inhomogeneous} Bernoulli percolation model with type-dependent probabilities, the root cluster conditioned on having $n$ vertices can be realized as an increasing process as $n$ varies.
\end{rem}

\subsection{Perspectives}

In light of our results, it is natural to ask whether a converse statement to our main theorem holds.
Let $\mu$ be a probability distibution on $\Z_+$ with $0<\mu(0)<1$.

\begin{pb}\label{pb:converse-main-thm}
	Suppose that there exists a coupling of the distributions $\BGWcond{\mu}{n}$, $n\geq1$, which is increasing with respect to $\subseteq$.
	Does this imply that the sequence $(\mu(k))_{k\geq0}$ is log-concave?
\end{pb}

Note that the answer to Problem~\ref{pb:converse-main-thm} is positive when $\mu$ has support included in $\{0,1,2\}$.
Indeed, by possibly performing an exponential tilt,%
	\footnote{
		This leaves unchanged the corresponding conditioned BGW measures, and log-concavity of the sequence $(\mu(k))_{k\geq0}$ is equivalent to log-concavity of its tilted version.
	}
we may assume that $\mu$ takes the form \eqref{eq:Janson-counter-ex} for some $\epsilon\in(0,1)$. Then the $\epsilon<1/3$ regime, where no increasing couplings can exist by \cite{Janson06}, is exactly the regime where the sequence $(\mu(k))_{k\geq0}$ fails to be log-concave.
Problem~\ref{pb:converse-main-thm} remains open otherwise.

In a different direction, it would be interesting to see if the statements of stochastic increase for infinite BGW trees in \cite{LyonsPeledSchramm08,Broman14,Broman16}, which were proved using the results of \cite{LuczakWinkler04}, hold beyond the geometric, binomial, and Poisson cases.

\subsection{Organization of the paper}

We consider in Section~\ref{sec:coupling-random-compositions} a general model of random compositions of integers and the problem of coupling them in a suitable increasing way.
Ultimately, we extract a set of inequalities whose satisfaction is sufficient to ensure that there exist such increasing couplings, see Corollary~\ref{cor:sufficient-inequalities}.

Sections~\ref{sec:application-to-trees} and~\ref{sec:arithmetic-case} are then dedicated to the proof of Theorems~\ref{thm:main-thm-BGW} and~\ref{thm:main-thm-BGW-arithmetic}, which are conveniently re-formulated in terms of \textit{simply generated trees}, see Theorems~\ref{thm:main-thm-SG} and~\ref{thm:main-thm-SG-arithmetic}.

In Section~\ref{sec:application-to-trees}, we relate the model of random compositions studied in Section~\ref{sec:coupling-random-compositions} to simply generated trees, and explain how to construct increasing couplings for them using the increasing couplings for random compositions.
To complete the proof of Theorem~\ref{thm:main-thm-BGW}, or equivalently Theorem~\ref{thm:main-thm-SG}, we show that the sufficient set of inequalities we alluded to above is satisfied in this setting, by proving an even stronger set of inequalities, see Proposition~\ref{prop:checking-the-inequalities}, which vastly generalizes Corollary 5.4 in \cite{LuczakWinkler04}.

Then, Section~\ref{sec:arithmetic-case} generalizes the results of Sections~\ref{sec:coupling-random-compositions} and~\ref{sec:application-to-trees} to the arithmetic case, in order to prove Theorem~\ref{thm:main-thm-BGW-arithmetic}.
The global approach is similar but the proof is more technical, and the set of inequalities we obtain in this setting has a more subtle structure, see Proposition~\ref{prop:checking-the-inequalities-new}.

Lastly, we apply our results by proving Theorem~\ref{thm:growing-subtrees} in Section~\ref{sec:application-random-subtrees}.
Our approach is to relate bijectively rooted subtrees of $\U$ to some decorated plane trees.
We can then ``grow'' the plane trees and their decorations separately.
The last step consists in making the two ``growths'' match.
This involves performing some decoration-dependent ``shuffling operations''.

We collect in Appendix~\ref{app:appendix-shuffling} some technical lemmas about these decoration-dependent ``shuffling operations'', which may be of independent interest.

\section*{Acknowledgments}
The author would like to express his deep gratitude to Grégory Miermont for many discussions and helpful advice, and for his careful reading of several versions of this article. The author is also thankful to Svante Janson for pointing out some inconsistencies in Section 4.5 of the first version of this article.

This work has been carried for the most part at the ENS de Lyon (UMPA), and it has been finished at the Université Paris-Saclay (LMO) with support from SuPerGRandMa (ERC Consolidator Grant no 101087572).


%% file: parts/coupling-random-compositions.tex
We start by introducing a model of random compositions which is closely related to BGW trees, and we study a notion of increasing couplings for this model, which will be instrumental to our constructions of increasing couplings for BGW trees.

As is usual in probability theory, by standard extension theorems we may assume that we have at our disposal a sufficiently big probability space $(\Omega,\mathcal F,\P)$, together with realizations on it of all the random variables we may need.
Unless we expressly mention the probability distribution we are working with, our random variables will be defined on $(\Omega,\mathcal F,\P)$.

\subsection{The model}

\mypar{Compositions of integers}
A composition of an integer $n\geq 0$ is a word $c=\letter{n_1} \cdots   \letter{n_r}$ with $r\geq0$ letters on the alphabet $\{1,2,\dots\}$ such that ${n}_1+\dots+{n}_r=n$.
The integers $({n}_i)$ are called the parts of $c$, and we underline them when writing a composition in order to avoid possible confusions.
We write $c\composes n$ to mean that $c$ is a composition of $n$.
By convention, the empty word $\emptyset$ is the only composition of the integer $0$.
The set of compositions of an integer $n\geq 0$ will be denoted by $\Compos{n}$.

\mypar{Covering and order relations}
We say that a composition $c$ is covered by a composition $c'$ and write $c\covered c'$ if we can obtain $c'$ from $c$ by either adding $1$ to one of its parts or by adding a new part with value $1$ on its right.
More precisley, if $c={\letter{n_1} \cdots   \letter{n_r}}$ then $c'$ covers $c$ if it is either one of the following compositions
\begin{align*}
\bigl(\letter{n_1}  \cdots  \letter{n_{i-1}}  \:\letter{\,n_i+1\,}\:  \letter{n_{i+1}} \cdots    \letter{n_r}\bigr),
\end{align*}
for some $i\in\{1,\dots,r\}$, or if it is the following composition
\begin{align*}
\letter{n_1}  \cdots   \letter{n_r}  \:\letter {\,1\,}.
\end{align*}
Notice that if $c$ is a composition of $n\geq 0$, then a composition $c'$ which covers $c$ is a composition of $n+1$.
The relation $\covered$ is extended by transitivity into a partial order $\preceq$ on the set of all compositions of integers.

{\begin{figure}[p]
	\begin{center}
		\includegraphics[page=1]{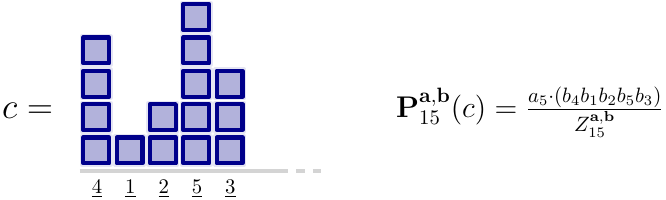}
	\end{center}
	\caption{A graphical representation of the composition $c=\letter 4 \,\letter 1\, \letter 2 \, \letter 5\, \letter 3$ as piles of square bricks.
	The parts of the composition are the numbers of bricks in each column.}
\end{figure}

\begin{figure}[p]
	\begin{center}
		\includegraphics[page=2,scale=.7]{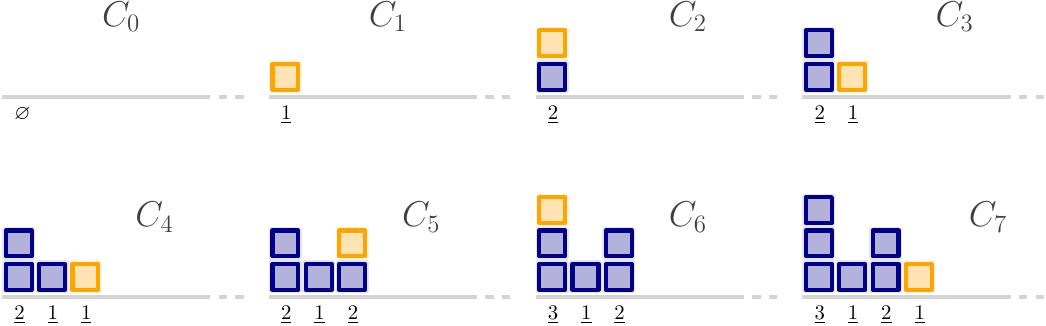}
	\end{center}
	\caption{An example of how of a sequence $C_0\covered C_1\covered C_2\covered\dots$ may begin.
		At each step, the newly added brick is colored in yellow.}
\end{figure}

\begin{figure}
	\begin{center}
		\includegraphics[page=3,scale=.8]{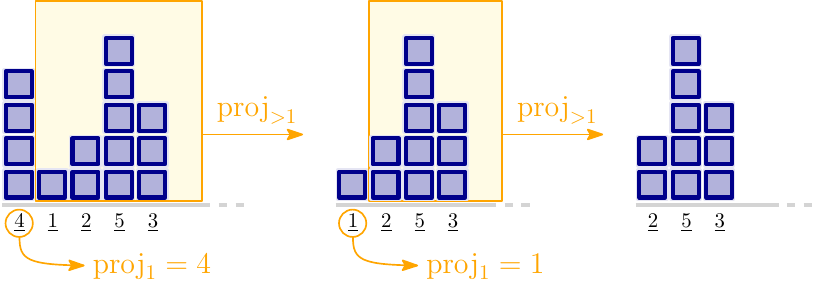}
	\end{center}
	\caption{An illustration of the action of ${\mathrm{proj}}_1$ and ${\mathrm{proj}}_{>1}$ on compositions.}
	\label{fig:proj-compositions}
\end{figure}}

\mypar{Weight pairs}
A pair $(\wa,\wb)$ where $\wa=(a_0,a_1,\dots)$ and $\wb=(b_1,b_2,\dots)$ are non-negative sequences will be called a \textit{weight pair}.
It is \textit{non-degenerate} if
\begin{enumerate}
	\item we have $a_0a_1>0$ and the sequence $\wa$ has no \textit{internal zeros}, or equivalently its support is $\{\ell\colon0\leq\ell\leq r\}$ for some $r=r(\wa)\in\{1,2,\dots\}\cup\{\infty\}$; and,
	\item the sequence $\wb$ is positive, or equivalently its support is $\{1,2,\dots\}$.
\end{enumerate}

\mypar{Random compositions}
Given a non-degenerate weight pair $(\wa,\wb)$ and for every $n\geq 0$, we define a probability distribution on $\Compos{n}$ by letting, for $\letter{n_1}  \dots   \letter{n_r}\in\Compos{n}$,
\begin{align}\label{eq:def-proba-compositions}
\ProbComp{\wa,\wb} n {\letter{n_1}  \cdots   \letter{n_r}}&=
\frac{a_r\cdot (b_{{n}_1}\cdots b_{{n}_r})}{\PFComp{\wa,\wb} n},
&&
\PFComp{\wa,\wb}n=
	\sum_{\substack{i\geq 0\\\letter{n_1}  \cdots   \letter{n_i}\composes n}}a_i\cdot(b_{{n}_1}\cdots b_{{n}_i}).
\intertext{
	Consistently with the usual convention that an empty product equals $1$, the case $n=0$ reads}
\label{eq:def-proba-compositions-2}
\ProbComp{\wa,\wb} 0 {\emptyset}&=
\frac{a_0}{\PFComp{\wa,\wb} 0}=1,
\qquad&&\text{where}\qquad
\PFComp{\wa,\wb}0=
a_0.
\end{align}
Obviously $a_0$ plays no role in the definition of $\PComp{\wa,\wb}{0}$, or of the $(\PComp{\wa,\wb}{n})_{n\geq1}$ for that matter, but having a consistently defined $n=0$ case will allow for cleaner statements later on when some compositions we deal with may be empty, such as in Proposition \ref{prop:role-of-shifted-a}.
The assumption that $(\wa,\wb)$ is non-degenerate ensures that $\PFComp{\wa,\wb}n$ is non-zero for every $n\geq 0$ so that the above probability distributions are well-defined.

\begin{rem}
	The model of random compositions defined by \eqref{eq:def-proba-compositions} can be seen as a variant on a model of random allocation called the \textit{balls-in-boxes} or \textit{balls-in-bins} model, but where the number of boxes is also allowed to be random.
	We refer the reader to Janson's survey~\cite[Sect.~11]{Janson12} for background on the model with a fixed number of boxes.
\end{rem}

\begin{rem}
	Another related model is Gnedin and Pitman's model of \textit{exchangeable Gibbs partitions} \cite{GnedinPitman05}.
	The authors associate to these random partitions some random compositions, whose distributions have a similar structure as the ones we defined in \eqref{eq:def-proba-compositions}, see \cite[Def.~1]{GnedinPitman05}.
\end{rem}

Our aim is to find \textit{admissible} weight pairs $(\wa,\wb)$, in the following sense.

\begin{defin}[Admissibility]
	A non-degenerate weight pair $(\wa,\wb)$ will be called \emph{admissible} if there exists a coupling of the random compositions $(C_n)_{n\geq 0}$  with distributions $(\PComp{\wa,\wb}{n})_{n\geq 1}$ respectively, in such a way that 
	\begin{align*}
	C_0\covered C_1\covered C_2\covered \dots\covered C_n\covered C_{n+1}\covered\dots.
	\end{align*}
\end{defin}

\begin{lemma}\label{lem:closure-admissibility}
	The set of admissible weight pairs is closed in the set of non-degenerate weight pairs equipped with the topology of pointwise convergence.
\end{lemma}

\begin{proof}
	Let $n\geq0$ and let $A_n$ be the subset of $\Compos{n}\times\Compos{n+1}$ consisting of those pairs $(C,C')$ for which $C\covered C'$.
	The mapping which maps a probability measure $\mfrak$ on $\Compos{n}\times\Compos{n+1}$ to the value $\mfrak(A_n)$ is clearly continuous with respect to convergence in distribution since we work with finite spaces.
	Hence if we take $(\wa_i,\wb_i)\rightarrow_{i\rightarrow\infty}(\wa,\wb)$ in a sequence $(C_n(\wa_i,\wb_i), C_{n+1}(\wa_i,\wb_i))_{i\geq1}$ of $A_n$-supported couplings of the distributions $\PComp{\wa_i,\wb_i}{n}$ and $\PComp{\wa_i,\wb_i}{n+1}$, then any subsequential limit we may obtain is also supported in $A_n$.
	Such a subsequential limit exists by compactness, and forms a coupling of $\PComp{\wa,\wb}{n}$ and $\PComp{\wa,\wb}{n+1}$ since the mappings $(\wa,\wb)\mapsto \PComp{\wa,\wb}{n}(c)$ and $(\wa,\wb)\mapsto\PComp{\wa,\wb}{n+1}(c)$ are continuous for every composition $c$ (as rational functions in the coordinates of $\wa$ and $\wb$, with non-vanishing%
	\footnote{
		To be precise, the denominators are non-vanishing on the set of non-degenerate weight pairs.
	}
	denominators).
	This holds for all $n\geq0$, so that taking a limit $(\wa_i,\wb_i)\rightarrow_{i\rightarrow\infty}(\wa,\wb)$ of admissible weight pairs inside the set of non-degenerate ones yields a weight pair $(\wa,\wb)$ which is admissible.
\end{proof}

\subsection{Key lemma}

It is well-known that two \textit{real} variables $X$ and $Y$ can be coupled in such a way that $X\leq Y$ if and only if their cumulative distribution functions $\cdf_X\colon t\mapsto \Prob{X\leq t}$ and $\cdf_Y\colon t\mapsto \Prob{Y\leq t}$ satisfy $\cdf_X\geq \cdf_Y$ on $\R$.
The latter is clearly necessary.
Conversely, denote by $F^{-1}$ the generalized inverse of a function $F\colon \R\rightarrow[0,1]$, that is
\begin{align*}
F^{-1}\colon u\in[0,1]\mapsto F^{-1}(u)=\inf\{t\in\R\colon F(t)\geq u\}.
\end{align*}
Then assuming that $\cdf_X\geq \cdf_Y$ on $\R$, the coupling $(X',Y')=(\cdf_X^{-1}(U),\cdf_Y^{-1}(U))$ satisfies $X'\leq Y'$, where $U$ is uniformly random in $[0,1]$.
If now $\cdf_X\geq \cdf_Y\geq \cdf_{X+1}$, then for $u\in[0,1]$,
\begin{align*}
\{t\in\R\colon \cdf_{X+1}(t)\geq u\}
\subseteq\{t\in\R\colon \cdf_Y(t)\geq u\}
\subseteq\{t\in\R\colon \cdf_X(t)\geq u\},
\end{align*}
and therefore by taking infima and observing that $\cdf_{X+1}$ is  $ t \mapsto \cdf_X(t-1)$, we get
\begin{align*}
X'\leq Y'\leq X'+1.
\end{align*}
The following lemma is equivalent to Lemma 3.1 in {\cite{LuczakWinkler04}}.

\begin{lemma}[Key lemma]\label{lem:key-lemma}
	Let $n\geq 1$ and let $\mu_n$ and $\mu_{n+1}$ be two probability measures on $\{1,2,\dots,n\}$ and $\{1,2,\dots,n+1\}$ respectively.
	Assume that for $m=1,\dots, n$ we have
	\begin{align}\label{eq:key-lemma-inequality}
	\mu_{n+1}(m)\leq \mu_n(m)\geq \mu_{n+1}(m+1).
	\end{align}
	Then, there exists a coupling $(X_n,X_{n+1})$ of the probability measures $\mu_n$ and $\mu_{n+1}$ such that $X_{n+1}\in\{X_n,X_n+1\}$ almost surely.
\end{lemma}

\begin{proof}
	Let $X_n$, resp~$X_{n+1}$, have distribution $\mu_n$, resp.~$\mu_{n+1}$.
	By summing the first inequality in \eqref{eq:key-lemma-inequality} for $m\leq k$ and the second one for $m>k$ we get for $1\leq k\leq n$,
	\begin{align*}
	\cdf_{X_{n+1}}(k)\leq \cdf_{X_n}(k)
	\qquad\text{and}\qquad
	1-\cdf_{X_n}(k)\geq 1-\cdf_{X_{n+1}}({k+1}),
	\end{align*}
	that is $\cdf_{X_{n+1}}(k)\leq \cdf_{X_n}(k)\leq \cdf_{X_{n+1}}({k+1})$.
	Hence by the discussion above, the variables $X_n$ and $X_{n+1}$ can be coupled in such a way that $X_n\leq X_{n+1}\leq X_n+1$, or equivalently $X_{n+1}\in\{X_n,X_n+1\}$ since $X_n$ and $X_{n+1}$ are integer-valued.
\end{proof}

\begin{cor}\label{cor:cor-key-lemma}
	Let $(\mu_n)_{n\geq 1}$ be a sequence of probability measures with support included in $\{1,\dots,n\}$ respectively.
	If the inequalities \eqref{eq:key-lemma-inequality} of Lemma \ref{lem:key-lemma} are satisfied for every $1\leq m\leq n$, then there exists a coupling $(X_n)_{n\geq 1}$ of the measures $(\mu_n)_{n\geq 1}$ such that $X_{n+1}\in\{X_n,X_n+1\}$ for every $n$ almost surely.
\end{cor}

\begin{proof}
	For every $n\geq1$, let $(X^{(n)}_n,X^{(n)}_{n+1})$  be a coupling of the distributions $\mu_n$ and $\mu_{n+1}$ given by Lemma \ref{lem:key-lemma}, hence satisfying $X^{(n)}_{n+1}\in\{X^{(n)}_n,X^{(n)}_n+1\}$ almost surely.
	Let us denote by $p_n(x_n,x_{n+1})$ the probability that $X^{(n)}_{n+1}=x_{n+1}$ given $X^{(n)}_n=x_n$.
	This suffices to define a time-inhomogeneous Markov chain $(X_n)_{n\geq 1}$ started at $X_1=1$ and with transition probabilities 
	\begin{align*}
	\condProb{X_{n+1}=x_{n+1}}{X_n=x_n}=p_n(x_n,x_{n+1}).
	\end{align*}
	Then $(X_n)_{n\geq 1}$ is a coupling of the measures $(\mu_n)_{n\geq 1}$ such that $X_{n+1}\in\{X_n,X_n+1\}$ for every $n$ almost surely.
\end{proof}

\subsection{Admissibility and shifted \texorpdfstring{$\wa$}{a}-weights}
\label{subsec:admissiblity-and-shift}

For a sequence $\wa=(a_0,a_1,\dots)$ of non-negative numbers, we denote by $\wa^+$ the same sequence shifted by one unit to the left, that is $(a^+_0,a^+_1,a^+_2,\dots)=(a_1,a_2,a_3,\dots)$.
For a composition $c=\letter{n_1}\cdots \letter{n_r}$ of $n\geq 1$,we let 
\begin{align*}
\projParts_1(c)={n}_1\in\{1,2,\dots\}
\qquad\text{and}\qquad
\projParts_{>1}(c)=\letter{n_2}\cdots \letter{n_r}\in\Compos{n-{n}_1}.
\end{align*}

\begin{prop}\label{prop:role-of-shifted-a}
	Let $(\wa,\wb)$ be a non-degenerate weight pair, with $r(\wa)\geq2$ so that $(\wa^+,\wb)$ is also non-degenerate.
	Let $C_n$ be a $\PComp{\wa,\wb}n$-distributed random composition  for some fixed $n\geq 1$.
	If we let $X_n=\projParts_1(C_n)$ and $C'_n=\projParts_{>1}(C_n)$, then $X_n$ has distribution given by
	\begin{align}\label{eq:law-X_n}
	\Prob{X_n=m}=\frac{b_m\cdot\PFComp{\wa^+,\wb}{n-m}}{\PFComp{\wa,\wb}{n}},
	\qquad
	m=1,\dots,n,
	\end{align}
	and conditionally on $X_n$, the composition $C'_n$ has distribution $\PComp{\wa^+,\wb}{n-X_n}$.
\end{prop}

\begin{proof}
	For $m=1,\dots,n$ and a composition $c=\letter{n_1}\cdots \letter{n_i}$ in $\Compos{n-m}$, we have
	\begin{align}\label{eq:law-X_n-and-C_n}
	\Prob{X_n=m,C'_n=c}
	=\Prob{C_n=\letter mc}
	=\frac{a_{i+1}\cdot(b_mb_{{n}_1}\cdots b_{{n}_i})}{\PFComp{\wa,\wb}{n}}.
	\end{align}
	In particular, by summing over $c\composes n-m$ we get
	\begin{align*}
	\Prob{X_n=m}
		=\sum_{\substack{i\geq 0\\\letter{n_1}\cdots \letter{n_i}\composes n-m}}\frac{a_{i+1}\cdot(b_mb_{{n}_1}\cdots b_{{n}_i})}{\PFComp{\wa,\wb}{n}}
		=\frac{b_m\cdot\PFComp{\wa^+,\wb}{n-m}}{\PFComp{\wa,\wb}{n}},
	\end{align*}
	which is \eqref{eq:law-X_n}.
	Dividing \eqref{eq:law-X_n-and-C_n} by the latter expression we get, for $m=1,\dots,n$ and $c=\letter{n_1}\cdots \letter{n_i}$ in $\Compos{n-m}$,
	\begin{align*}
	\condProb{C'_n=c}{X_n=m}
	=\frac{a_{i+1}\cdot(b_{{n}_1}\cdots b_{{n}_i})}{\PFComp{\wa^+,\wb}{n-m}}
	=\ProbComp{\wa^+,\wb}{n-m}{c},
	\end{align*}
	as claimed.
	Observe that the above expressions still hold when $n=m$, with the convention that an empty product evaluates to $1$.
\end{proof}

\begin{prop}\label{prop:admissibility-vs-shift}
	Let $(\wa,\wb)$ be a non-degenerate weight pair.
	\begin{enumerate}
		\item If $\wa=(a_0,a_1,0,0\dots)$ with $a_0,a_1>0$ then $(\wa,\wb)$ is admissible.
		\item If $(\wa^+,\wb)$ is admissible and if the inequalities
		\begin{align}\label{eq:ineq-shift}
		\frac{\PFComp{\wa^+,\wb}{n+1-m}}{\PFComp{\wa^+,\wb}{n-m}}
		\leq
		\frac{\PFComp{\wa,\wb}{n+1}}{\PFComp{\wa,\wb}{n}}
		\geq
		\frac{b_{m+1}}{b_m},
		\end{align}
		are satisfied for $m=1,\dots,n$, then $(\wa,\wb)$ is admissible.
	\end{enumerate}
\end{prop}

\begin{proof}
In the first case, the fact that $a_i=0$ for $i\geq 2$ imposes that a sample of $\PComp{\wa,\wb}{n}$ is a composition of $n$ into one part, for every $n\geq 1$.
There is only one such composition, namely the trivial composition $\letter n$.
In particular $(\wa,\wb)$ is trivially admissible in that case.

Let us now focus on the second case and assume that $(\wa^+,\wb)$ is admissible and that the inequalities \eqref{eq:ineq-shift} are satisfied.
We will leverage the basic observation that for two compositions $c$ and $c'$ we have the equivalence
\begin{align}\label{eq:equivalence-covering-relation}
c\preceq c' \qquad\iff\qquad\projParts_1(c)\leq \projParts_1(c')\quad\text{and}\quad \projParts_{>1}(c)\preceq\projParts_{>1}(c').
\end{align}
Since $(\wa^+,\wb)$ is admissible, there exists a coupling  $(C^+_n)_{n\geq0}$ of the probability distributions $(\PComp{\wa^+,\wb}{n})_{n\geq0}$ such that $C^+_0\covered C^+_1\covered\dots$.
For $n\geq1$, define the image measure $\mu_n=(\projParts_1)_*\bigl(\PComp{\wa,\wb}{n}\bigr)$.
Then by \eqref{eq:law-X_n} in Proposition \ref{prop:admissibility-vs-shift}, we have for $m=1,\dots,n$,
\begin{align*}
\mu_n(k)=\frac{b_m\cdot\PFComp{\wa^+,\wb}{n-m}}{\PFComp{\wa,\wb}{n}}.
\end{align*}
The inequalities \eqref{eq:ineq-shift}, which are assumed to hold, precisely tell that for all $1\leq m\leq n$,
\begin{align*}
\mu_{n+1}(m)\leq \mu_n(m)\geq \mu_{n+1}(m+1).
\end{align*}
Hence by Corollary \ref{cor:cor-key-lemma}, there exists a coupling $(X_n)_{n\geq 1}$ of the distributions $(\mu_n)_{n\geq 1}$ such that $X_{n+1}\in\{X_n,X_n+1\}$ for every $n$ almost surely, which may and will take independent from the sequence $(C^+_n)_{n\geq0}$.
The fact that $X_{n+1}\in\{X_n,X_n+1\}$ for every $n$ yields that $(n-X_n)_{n\geq 1}$ is non-decreasing so that by construction of the coupling $(C^+_n)_{n\geq0}$ we have $C^+_{n-X_{n}}\preceq C^+_{n+1-X_{n+1}}$ for every $n\geq 1$.
For $n\geq 1$ we define the composition $C_n=\letter{X_n}\,C^+_{n-X_n}$.
Then by \eqref{eq:equivalence-covering-relation} we have
\begin{align*}
C_1\preceq C_2\preceq \dots.
\end{align*}
By Proposition \ref{prop:role-of-shifted-a}, the sequence $(C_n)_{n\geq 1}$ is indeed a coupling of the distributions $(\PComp{\wa,\wb}{n})_{n\geq 1}$.
If we finally add the deterministic composition $C_0=\emptyset$, which trivially satisfies $C_0\preceq C_1$, we obtain a coupling $(C_n)_{n\geq 0}$ of $(\PComp{\wa,\wb}{n})_{n\geq 0}$ such that $C_0\preceq C_1\preceq C_2\preceq \dots$.
Now notice that for every $n\geq0$, if $c\in\Compos{n}$ and $c'\in\Compos{n+1}$ are such that $c\preceq c'$, then $c$ is actually \textit{covered} by $c'$.
Hence we actually have $C_0\covered C_1\covered C_2\covered\dots$, and therefore $(\wa,\wb)$ is admissible.
\end{proof}

For $\ell$ in $\{0,\dots,r-1\}$, we denote by $\shift \ell \wa$ the sequence obtained by shifting $\ell$ times to the left the sequence $\wa$, that is
$(\shift\ell a_0,\shift\ell a_1,\dots)=(a_\ell,a_{\ell+1},\dots)$.

\begin{cor}\label{cor:sufficient-inequalities}
	Let $(\wa,\wb)$ be a non-degenerate weight pair and assume that $r=r(\wa)$ is finite.
	Assume that the following inequalities hold for all $n\geq 0$ and all $\ell\in\{0,\dots,r-1\}$,
	\begin{equation}\label{eq:main-inequalities}
	\frac{b_{n+1}}{b_{n}}
	\leq
	\frac{
		\PFComp{\shift{\ell}\wa,\wb}{n+1}
	}{
		\PFComp{\shift{\ell}\wa,\wb}{n}}
	\leq
	\frac{b_{n+2}}{b_{n+1}},
	\end{equation}
	where we omit the ill-defined inequality on the left when $n=0$.
	Then $(\wa,\wb)$ is admissible, as well as the pairs $(\shift\ell\wa,\wb)$ for $\ell=0,\dots,r-1$.
\end{cor}

\begin{proof}
Since  $\shift{r-1}\wa=(a_{r-1},a_r,0,0,\dots)$ with $a_{r-1}a_r>0$, the first case of Proposition \ref{prop:admissibility-vs-shift} gives that $(\shift{r-1}\wa,\wb)$ is admissible.
If $r=1$ we are done, so let us consider the case $r\geq2$.
Notice that for $n\geq 1$, the inequalities \eqref{eq:main-inequalities} start with $\frac{b_{n+1}}{b_n}$ and end with $\frac{b_{n+2}}{b_{n+1}}$, so that they can be chained as $n$ increases to relate the inequalities corresponding to different values of $n$ together.
We read in the inequalities we obtain in this way that for $1\leq m\leq n$ and $0\leq\ell\leq r-2$,
\begin{align*}
\frac{\PFComp{\shift{\ell+1}\wa,\wb}{n+1-m}}{\PFComp{\shift{\ell+1}\wa,\wb}{n-m}}
\leq
\frac{b_{n+2-m}}{b_{n+1-m}}
\leq
\frac{b_{n+1}}{b_{n}}
\leq
\frac{\PFComp{\shift\ell\wa,\wb}{n+1}}{\PFComp{\shift\ell\wa,\wb}{n}}
\qquad\text{and}\qquad
\frac{\PFComp{\shift\ell\wa,\wb}{n+1}}{\PFComp{\shift\ell\wa,\wb}{n}}
\geq
\frac{b_{n+1}}{b_{n}}
\geq
\frac{b_{m+1}}{b_m}.
\end{align*}
Since $(\shift{r-1}\wa,\wb)$ is admissible, the last displayed inequalities imply through repeated applications of Proposition \ref{prop:admissibility-vs-shift} that $(\shift{r-2}\wa,\wb),\dots,(\shift0\wa,\wb)$ are also admissible, as claimed.
\end{proof}

%% file: parts/application-to-trees.tex
\subsection{Simply generated trees}\label{subsec:defs-trees}

For convenience we will use the framework of \textit{simply generated trees} in the sense of Meir and Moon \cite{MeirMoon78}, see also Janson's survey \cite{Janson12}.
For $\w=(w_0,w_1,\dots)$ a non-negative sequence, we denote by $\omega$ the following function on trees, which we interpret as a $\sigma$-finite measure,
\begin{align}\label{eq:def-omega}
\forall \pt\in \trees,\quad
\omega(\pt)=\prod_{\u\in\pt}w_{k_\u(\pt)},
\end{align}
where we recall that $k_\u(\pt)$ denotes the number of children of a vertex $\u$ in a plane tree $\pt$.
The dependence in $\w$ is implicit in the notation $\omega$ but will always be clear from context.
The distribution of the simply generated tree with $n$ vertices associated to the weight sequence $\w$ is defined for $n\geq1$ by
\begin{align}\label{eq:def-SG}
\forall \pt\in \trees,\quad
\SimpGen{\w}n(\pt)=\frac{\omega(\pt)}{\PFtrees{\w}{n}},
\where
\quad\text{and}\quad
\PFtrees{\w}{n}=\sum_{\pt\in\trees_n}\omega(\pt).
\end{align}
This definition only makes sense when $\PFtrees{\w}{n}\neq 0$, which is the case for every $n\geq1$ under the assumption that $w_0w_1>0$, since in this case the tree $\pt\in\trees_n$ consisting of a single ancestral line with $n$ vertices has $\omega(\pt)>0$.
When $\mu$ is a probability distribution, the distribution $\SimpGen{\w}n$ corresponding to $\w=(\mu(i))_{i\geq0}$ is precisely the conditioned Bienaymé--Galton--Watson distribution $\BGWcond{\mu}{n}$, for all $n\geq 1$ such that $\PFtrees{\w}{n}\neq 0$.
Conversely, given a non-negative sequence $\w$ such that the power series $\sum_i x^iw_i$ has non-zero radius of convergence, there exist $a,b>0$ such that the definition $\mu(i)=ab^iw_i$ for all $i\geq0$ yields a probability distribution $\mu$ which satisfies $\BGWcond{\mu}{n}=\SimpGen{\w}{n}$ for all $n\geq 1$ such that $\PFtrees{\w}{n}\neq 0$.
See \cite{Janson12} for details.

In this context, Theorem \ref{thm:main-thm-BGW} can be reformulated in the following way.

\begin{thm}\label{thm:main-thm-SG}
	Let $\w=(w_0,w_1,\dots)$ be a non-negative sequence with $w_0w_1>0$.
	If $\w$ is log-concave, then the random trees with respective distributions $(\SimpGen{\w}{n})_{n\geq 1}$ can be realized as a Markov process $(\T_n)_{n\geq 1}$ in which at each step a right-leaning leaf is added.
\end{thm}

This version may look stronger than Theorem \ref{thm:main-thm-BGW} since simply generated trees generalize Biemaymé--Galton--Watson trees, but both versions are actually strictly equivalent.
Indeed if $\w$ is log-concave, then $w_{n+1}\leq w_n w_1/w_0$ for all $n\geq 0$ so that the power series $\sum_i x^iw_i$ has non-zero radius of convergence.
As mentioned above, this entails that there exist $a,b>0$ such that setting $\mu(i)=ab^iw_i$ for all $i\geq0$ defines a probability distribution $\mu$ which satisfies $\BGWcond{\mu}{n}=\SimpGen{\w}{n}$.
But $\mu$ thus defined is log-concave and Theorem \ref{thm:main-thm-BGW} applies.
The remainder of Section \ref{sec:application-to-trees} is devoted to proving Theorem \ref{thm:main-thm-SG} and thus equivalently Theorem \ref{thm:main-thm-BGW}.

\subsection{Relationship with random compositions}\label{subsec:link-with-compositions}

For $n\geq 1$ and a tree $\pt\in\trees_n$ whose root has $i=k_\emptyset(\pt)\geq 0$ children, we write $\phi(\pt)=(\pt^{[1]},\dots,\pt^{[i]})$  for the possibly empty collection where, for $1\leq j\leq i$, we denote by $\pt^{[j]}$ the subtree of descendants of the $j$-th child of the root from left to right, that is, in Neveu's formalism:
\begin{align*}
\phi(\pt)=(\pt^{[1]},\dots,\pt^{[i]}),\quad\pt^{[j]}=\Bigl\{\u\in\U\colon j\u\in\pt \Bigr\},\quad  1\leq j\leq i,\quad i=k_\emptyset(\pt).
\end{align*}
Note that the mapping $\phi$ is a bijection whose inverse mapping associates to any finite sequence of plane trees $(T^1,\dots,T^i)$, $i\geq0$, the plane tree
\begin{align*}
\phi^{-1}(T^1,\dots,T^i)=\{\emptyset\}\cup\bigcup_{1\leq j\leq i}jT^j,
\end{align*}
where $j V=\{j\u\colon\u\in V\}$ for every $j\geq1$ and every $V\subset\U$.
The preceding allows to associate to any plane tree $\pt$ a composition 
\begin{align*}
c(\pt)=\letter{n}_1 \cdots  \letter{n}_i
\where
n_j=\#\,\pt^{[j]},\quad 1\leq j\leq i.
\end{align*}
Notice that $c(\pt)$ is a composition of ${n-1}$ since the sets $j\pt^{[j]}$, $1\leq j\leq i$, partition the non-root vertices of $\pt$.
Consistently, if $\pt$ is the tree with one vertex then $c(\pt)=\emptyset$ since in this case the collection $(\pt^{[1]},\dots,\pt^{[i]})$ is empty.

\begin{prop}\label{prop:link-with-compositions}
	Let $\w=(w_0,w_1,\dots)$ be a non-negative sequence which has no internal zeros and such that $w_0w_1>0$.
	Set $\wb=(\PFtrees{\w}{n})_{n\geq 1}$.
	Then $(\w,\wb)$ is a non-degenerate weight pair.
	For $n\geq 1$ we have $\PFtrees\w n=\PFComp{\w,\wb}{n-1}$, and for $\pt\in\trees_n$ if we write $c(\pt)=\letter{n}_1 \cdots  \letter{n}_i$ and $\phi(\pt)=(\pt^{[1]},\dots,\pt^{[i]})$, then
	\begin{align}\label{eq:description-laws-trees-compositions}
	\SimpGen{\w}{n}(\pt)=\PComp{\w,\wb}{n-1}\bigl(c(\pt)\bigr)\cdot\left(\SimpGen{\w}{n_1}\bigl(\pt^{[1]}\bigr)\cdots \SimpGen{\w}{n_i}\bigl(\pt^{[i]}\bigr)\right).
	\end{align}
\end{prop}

\begin{proof}	
	Since $w_0w_1>0$, the sequence $\wb$ is positive as mentioned after \eqref{eq:def-SG}.
	Since additionally $\w$ has no internal zeros, the weight pair $(\w,\wb)$ is indeed non-degenerate.
	Let $n\geq 1$ and let $\pt\in\trees_n$ and $\phi(\pt)=(\pt^{[1]},\dots,\pt^{[i]})$.
	The vertices of $\pt$ may be partitioned into the root vertex (which has $i$ children) on the one hand, and on the other hand the elements of the sets $j\pt^{[j]}$, $1\leq j\leq i$.
	Hence, by the product form \eqref{eq:def-omega} for $\omega(\pt)$, and the easy verification that for $1\leq j\leq i$ we have $k_{j\u}(\pt)=k_\u(\pt^{[j]})$, we get:
	\begin{align}\label{eq:proof-link-compos:product-form}
	\omega(\pt)=w_i\cdot \omega(\pt^{[1]})\cdots\omega(\pt^{[i]}).
	\end{align}
	In particular if $c(\pt)=\letter{n}_1 \cdots  \letter{n}_i$ then by the definition \eqref{eq:def-SG}, the latter rewrites as follows:
	\begin{align*}
	\SimpGen{\w}{n}(\pt)
		&= \frac{w_i\cdot \omega(\pt^{[1]})\cdots\omega(\pt^{[i]})}{\PFtrees{\w}{n}}\\
		&= \frac{w_i\cdot(\PFtrees{\w}{n_1}\cdots\PFtrees\w{n_i})}{\PFtrees{\w}{n}}\cdot
		\SimpGen{\w}{n_1}\bigl(\pt^{[1]}\bigr)\cdots \SimpGen{\w}{n_i}\bigl(\pt^{[i]}\bigr),
	\end{align*}
	and thus by using the definition of $\PComp{\w,\wb}{n-1}$ in \eqref{eq:def-proba-compositions},
	\begin{align}\label{eq:proof-link-compos:product-form-rewritten}
	\SimpGen{\w}{n}(\pt)
		= \frac{\PFComp{\w,\wb}{n-1}}{\PFtrees{\w}{n}}\cdot
		\ProbComp{\w,\wb}{n-1}{\letter{n}_1\cdots \letter{n}_i}\cdot\SimpGen{\w}{n_1}\bigl(\pt^{[1]}\bigr)\cdots \SimpGen{\w}{n_i}\bigl(\pt^{[i]}\bigr).
	\end{align}
	Since the measures $(\SimpGen{\w}{k})_{k\geq1}$ are probability measures, we get for all $c\composes n-1$, by summing over all $\pt\in\trees_n$ such that $c(\pt)=c$,
	\begin{align*}
		\SimpGen{\w}{n}(\{\pt\colon c(\pt)=c\})
	= \frac{\PFComp{\w,\wb}{n-1}}{\PFtrees{\w}{n}}\cdot
	\ProbComp{\w,\wb}{n-1}{c},
	\end{align*}
	and by summing again on all $c\composes n-1$ and using that $\SimpGen{\w}{n}$ and $\PComp{\w,\wb}{n-1}$ are also probability measures we find $\PFtrees\w n=\PFComp{\w,\wb}{n-1}$, and \eqref{eq:proof-link-compos:product-form-rewritten} becomes \eqref{eq:description-laws-trees-compositions}.
\end{proof}

The following proposition reduces the proof of Theorem \ref{thm:main-thm-SG} to verifying that $(\w,\wb)$ is an admissible weight pair.
This is the counterpart in our setting to \cite[Proposition 2.1]{LuczakWinkler04}.

\begin{prop}\label{prop:growing-comp-is-sufficient}
	Let $\w=(w_0,w_1,\dots)$ be a non-negative sequence such that $w_0w_1>0$ and let $\wb=(\PFtrees{\w}{n})_{n\geq 1}$.
	If $(\w,\wb)$ is an admissible weight pair then the random trees with respective distributions $(\SimpGen{\w}{n})_{n\geq 1}$ can be realized as a Markov process $(\T_n)_{n\geq 1}$ such that $\T_1\subset\T_2\subset\T_3\subset\dots$.
\end{prop}

\begin{proof}
Our Markov process is initialized at $\T_1$ the deterministic  tree with one vertex.
We may now construct recursively the transition probabilities $p_n(\pt,\pt')$ for $n\geq1$ and $(\pt,\pt')\in\trees_n\times\trees_{n+1}$.
It is actually sufficient to construct $p_n(\pt,\pt')$ when $\pt$ and $\pt'$ are respectively in the supports of $\SimpGen{\w}{n}$ and $\SimpGen{\w}{n+1}$.

Let $n\geq 1$ and assume that $p_1,\dots,p_{n-1}$ have been constructed, so that we obtain a Markov process $(\T_1,\dots,\T_n)$ which is a coupling of $(\SimpGen{\w}{m})_{1\leq m\leq n}$ such that $\T_1\subset\T_2\subset\dots\subset\T_n$.
We take independent copies $(\T_{m}^{(j)})_{1\leq m\leq n}$ indexed by $j\geq1$ of this coupling.
By the assumption that $(\w,\wb)$ is an admissible weight pair, there exists a coupling $(C_n, C_{n+1})$ of the distributions $\PComp{\w,\wb}{n}$ and $\PComp{\w,\wb}{n+1}$ such that $C_n\covered C_{n+1}$.
We choose such a coupling, independently of previously constructed random variables.
We write $C_n=\letter n_1\cdots \letter n_i$ and $C_{n+1}=\letter n'_1\cdots \letter n'_{i'}$ and we set $\T=\phi^{-1}(\T_{n_1}^{(1)},\dots,\T_{n_i}^{(i)})$ and $\T'=\phi^{-1}(\T_{n'_1}^{(1)},\dots,\T_{n'_{i'}}^{(i')})$, that is more explicitly:
\begin{align}\label{eq:recursive-construction-coupling}
	\T=\{\emptyset\}\cup \bigcup_{1\leq j\leq i}j\T_{n_j}^{(j)},
	&&
	\T'=\{\emptyset\}\cup \bigcup_{1\leq j\leq i'}j\T_{n'_j}^{(j)}.
\end{align}
Since $C_n\covered C_{n+1}$, we have $i\leq i'$, as well as $n_j\leq n'_j$ for every $1\leq j\leq i$.
This observation, combined with the fact that $(\T_{m}^{(j)})_{1\leq m\leq n}$ is increasing with respect to inclusion for every $j\geq1$, gives that $\T\subseteq\T'$ using the expressions~\eqref{eq:recursive-construction-coupling}.

Hence, we constructed two random plane trees $\T$ and $\T'$ satifying $\T\subset\T'$, and such that $\T$ and $\T'$ have respective distributions $\SimpGen{\w}{n}$ and $\SimpGen{\w}{n+1}$ by Proposition~\ref{prop:link-with-compositions}.
To complete the inductive construction of the transition probabilities, it suffices to set $p_n(\pt,\pt')=\condProb{\T'=\pt'}{\T=\pt}$ for all plane trees $\pt$ and $\pt'$ in the support of $\SimpGen{\w}{n}$ and $\SimpGen{\w}{n+1}$ respectively.
\end{proof}

\subsection{Log-concavity and the Toeplitz TP2 property}

We will make use of the notion of \textit{total positivity of order 2}.

\begin{defin}
	Let $A=(A_{i,j})_{i,j\geq0}$ be an array of nonnegative numbers. It is said to be \textit{totally positive of order 2} (TP2), if all its $2\times 2$ minors are non-negative, that is for all $i\leq i'$ and all $j\leq j'$, we have $A_{i,j}A_{i',j'}\geq A_{i,j'}A_{i',j}$.
\end{defin}

Proposition \ref{prop:Toeplitz-TP2} below states the well-known fact that the Toeplitz matrix of a log-concave sequence is TP2.

\begin{prop}\label{prop:Toeplitz-TP2}
	Let $\x=(x_i)_{i\geq0}$ be a non-negative sequence.
	If it is log-concave then, then its Toeplitz matrix $(x_{i-j})_{i,j\geq 0}$  is TP2, that is we have:
	\begin{align*}
	\forall {i}\leq{i'},\quad \forall {j}\leq{j'},\qquad
	x_{{i}-{j}} x_{{i'}-{j'}}\geq x_{{i}-{j'}} x_{{i'}-{j}},
	\end{align*}
	with the convention that $x_{-1}=x_{-2}=\dots=0$.
\end{prop}
\begin{proof}
	Let $0\leq i<i'$ and $0\leq {j}<{j'}$.
	If $x_{{i}-{j'}} x_{{i'}-{j}}=0$ then the inequality is trivial.
	Otherwise the support of $\x$ contains $\{i-j',i-j'+1,\dots,i'-j\}$, since $\x$ has no internal zeros, even after extending with $x_{-1}=x_{-2}=\dots=0$.
	For $\ell\in\{i,i+1,\dots,i'\}$ we have
	\begin{align*}
	\frac{x_{\ell-{j}}}{x_{\ell-{j'}}}=
	\left(\frac{x_{\ell-{j'}+1}}{x_{\ell-{j'}+0}}\right)
	\cdot\left(\frac{x_{\ell-{j'}+2}}{x_{\ell-{j'}+1}}\right)\cdots
	\left(\frac{x_{\ell-{j'}+({j'}-{j})}}{x_{\ell-{j'}+({j'}-{j}-1)}}\right).
	\end{align*}
	Notice that all indices in the last display belong to $\{i-j',i-j'+1,\dots,i'-j\}$, hence the denominators are non-zero.
	By log-concavity of $\x$, each factor in the right-hand side is non-increasing in $\ell\in\{i,i+1,\dots,i'\}$.
	Hence we have
	$
	\frac{x_{{i}-{j}}}{x_{{i}-{j'}}}
	\geq \frac{x_{{i'}-{j}}}{x_{{i'}-{j'}}}
	$
	since ${i}<{i'}$.
\end{proof}

\subsection{Checking the inequalities}
\label{subsec:checking-inequalities}

Consider $\w=(w_0,w_1,\dots)$ a non-negative sequence with $w_0w_1>0$.
For $n,k\geq 1$, we denote by $\forests_{n,k}$ the set of ordered forests $F=(\pt_1,\dots,\pt_k)$ made of $k$ plane trees $\pt_1,\dots,\pt_k$ having $n$ vertices \textit{in total}, in the sense that $\sum_j\#\,\pt_j=n$.
Given a forest $F=(\pt_1,\dots,\pt_k)\in\forests_{n,k}$, we write $\omega(F)=\omega(\pt_1)\omega(\pt_2)\cdots\omega(\pt_k)$.
Then, for $n,k\geq 1$, we may set
\begin{align}\label{eq:def-PFforests}
\PFforests{\w}{n,k}=\sum_{\letter n_1 \cdots \letter n_k \composes n}\PFtrees{\w}{ n_1}\cdots \PFtrees{\w}{ n_k}
	=\sum_{F\in\forests_{n,k}} \omega(F).
\end{align}
Note in particular that $\PFforests{\w}{n,k}=0$ when $k>n$.

\begin{lemma}\label{lem:recursion-fnk}
	The array $(\PFforests{\w}{n,k})_{n,k\geq 1}$ satisfies the recursion
	\begin{align}\label{eq:recursion}
	\PFforests{\w}{n,k}
		=\sum_{i\geq 0}w_i \,\PFforests{\w}{n-1,k+i-1},
	\qquad n\geq 2,\qquad k\geq 1,
	\end{align}
	where we make the convention that $\PFforests{\w}{n,0}=0$ for all $n\geq 1$.
\end{lemma}

\begin{proof}
	Let $n\geq 2$ and $k\geq1$, and consider a forest of $k$ trees $F=(\pt_1,\dots \pt_k)$ having $n$ vertices in total.
	Consider the forest $F'=(\pt_1^{[1]},\cdots,\pt_1^{[i]},\pt_2,\cdots,\pt_k)$, where $\phi(\pt_1)=(\pt_1^{[1]},\cdots,\pt_1^{[i]})$ using the notation of section \ref{subsec:link-with-compositions}.
	As in \eqref{eq:proof-link-compos:product-form}, we get:
	\begin{align*}
	\omega(F)
			&=\omega(\pt_1)\omega(\pt_2)\cdots\omega(\pt_k)\\
			&=w_i\cdot \omega(\pt^{[1]})\cdots\omega(\pt^{[i]})\omega(\pt_2)\cdots\omega(\pt_k)\\
			&=w_i\cdot\omega(F').
	\end{align*}
	The application $F\mapsto (i, F')$ maps $\forests_{n,k}$ onto $\bigsqcup_{i\geq 0}\{i\}\times\forests_{n-1,k+i-1}$ bijectively, since $(i,T_1,\dots,T_{k+i-1})\mapsto(\phi^{-1}(T_1,\dots,T_i),T_{i+1},\dots,T_{k+i-1})$ is its inverse mapping.
	Therefore,
	\begin{align*}
	\PFforests{\w}{n,k}
	=\sum_{F\in\forests_{n,k}}\omega(F)
	=\sum_{\substack{i\geq 0\\ F'\in\forests_{n-1,k+i-1}}}w_i\cdot\omega(F')
	=\sum_{i\geq 0}w_i \,\PFforests{\w}{n-1,k+i-1},
	\end{align*}
	which is the claimed recursion.	
\end{proof}

In the theory of total positivity and production matrices, the recursion \eqref{lem:recursion-fnk} implies the following corollary, which is not new since it is a direct consequence of a result by Pétréolle, Sokal and Zhu \cite[Theorem 9.15]{PetreolleSokalZhu23}.

\begin{cor}\label{cor:TP2-forests-partition-function}
	If $\w$ is log-concave, then the array $(\PFforests{\w}{n,k})_{n,k\geq1}$ is TP2, that is
	\begin{align}\label{eq:ineq-forests}
	\forall 1\leq {n}\leq{n'},\quad \forall 1\leq{k}\leq{k'},\qquad
	\PFforests{\w}{n,k}\,\PFforests{\w}{n',k'}
	\geq \PFforests{\w}{n,k'}\,\PFforests{\w}{n',k}.
	\end{align}
\end{cor}

This will be a direct consequence of the following statement, which is a specialization of \cite[Theorems 9.4]{PetreolleSokalZhu23}. We include a self-contained proof for the convenience of the reader.

\begin{prop}\label{prop:TP2-fnk}
	Let $A=(A_{i,k})_{i,k\geq0}$ and $F=(F_{n,k})_{n,k\geq0}$ be two arrays
	of non-negative numbers, such that $F_{n,k}=0$ when $k>n$ and such that the following recursion is satisfied:
	\begin{align}\label{eq:recursion-Fnk-TP2}
	F_{n,k}	=\sum_{i\geq 0}A_{i,k} F_{n-1,i-1},	
	\qquad n\geq 1,\quad k\geq 0,
	\end{align}
	where we make the convention $F_{n,-1}=0$ for all $n\geq 0$.
	If the array $A$ is TP2, then so is the array $F$.
\end{prop}

\begin{proof}
	We need to prove that for every $0\leq {n}\leq{n'}$ and $0\leq{k}\leq{k'}$, we have:
	\begin{align}\label{eq:ineq-Fnk-TP2}
		F_{n,k}F_{n',k'}
		\geq F_{n,k'}F_{n',k}.
	\end{align}
	First, in the case $n=0$, we have $F_{n,k'}= 0$ for all $k'\geq 1$, so that the right-hand side of \eqref{eq:ineq-Fnk-TP2} is zero unless $k'=0$, but then $k=k'$ and there is equality in \eqref{eq:ineq-Fnk-TP2}.
	Let us now prove by induction on $N\geq 0$ the statement $\mathcal H_N$ that the inequalities \eqref{eq:ineq-Fnk-TP2} hold for all $0\leq n\leq n'\leq N$ and all $0\leq k\leq k'$.
	The base case $N=0$ follows from the preceding discussion.	
	Consider now $N\geq 1$ and assume that $\mathcal H_{N-1}$ holds true.
	Let $1\leq n\leq n'\leq N$ and $0\leq k\leq k'$.
	Consider the following sum:
	\begin{align}\label{eq:proof:TP2-of-fnk:key-expression}
	\sum_{i,i'\geq 0}
	(	A_{i,k}  A_{i',k'}-A_{i,k'} A_{i',k})
	\left(
	F_{n-1,i-1} F_{n'-1,i'-1}
	- F_{n-1,i'-1} F_{n'-1,i-1}
	\right).
	\end{align}
	By the assumption that $A$ is TP2 and by the induction hypothesis $\mathcal H_{N-1}$, each term in the sum is a product of two factors which are both non-negative when $i\leq i'$ and both non-positive when $i\geq i'$.
	As a consequence, the sum in \eqref{eq:proof:TP2-of-fnk:key-expression} is non-negative, which means by expanding the products that
	\begin{multline*}
	2\biggl[\sum_{i}A_{i,k}F_{n-1,i-1}\biggr]
	\biggl[\sum_{i'}A_{i',k'}F_{n'-1,i'-1}\biggr]
	\\-
	2\biggl[\sum_{i}A_{i,k'}F_{n-1,i-1}\biggr]
	\biggl[\sum_{i'}A_{i',k}F_{n'-1,i'-1}\biggr]
	\geq 0,
	\end{multline*}
	where the sums run over $i\geq 0$ and $i'\geq 0$.
	By the recursion \eqref{eq:recursion-Fnk-TP2} this is equivalent to
	\begin{align*}
	F_{n,k}\,F_{n',k'} \geq  F_{n,k'}\,F_{n',k}.
	\end{align*}
	Hence inequality \eqref{eq:ineq-Fnk-TP2} holds when $1\leq n\leq n'\leq N$ and $0\leq k\leq k'$, and it also holds in the case $n=0$ by the discussion at the beginning of the proof.
	This concludes the inductive proof.
\end{proof}

By Lemma~\ref{lem:recursion-fnk} and Proposition~\ref{prop:Toeplitz-TP2}, we can apply Proposition~\ref{prop:TP2-fnk} to the arrays $F=(F_{n,k})_{n,k\geq0}$ and $A=(A_{i,k})_{i,k\geq0}$, given by $F_{n,k}=\PFforests{\w}{n+1,k+1}$ for ${n,k\geq0}$ and $A_{i,k}=x_{i-k}$ for ${i,k\geq0}$. This proves Corollary \ref{cor:TP2-forests-partition-function}.

\begin{rem}
	The results of Pétréolle, Sokal and Zhu in \cite{PetreolleSokalZhu23} are actually far more general than what we need here and deal with total positivity of arbitrary order,%
		\footnote{Total positivity or order $r$ means that all $\ell\times\ell$ minors are non-negative for all $\ell\leq r$.}
	for combinatorially defined arrays related to lattice paths with height-dependent weights.
	Their theorem  10.15 is obtained using the theory of production matrices.
	Starting from a ``production matrix'' $A$, one builds a so-called ``output matrix'' $F$ \textit{via} a recursion of the form~\eqref{eq:recursion-Fnk-TP2}, and the construction has the property that the output matrix is totally positive of order $r\geq2$, whenever the production matrix is.
	We refer the reader to \cite[pp.~132--134]{Karlin68}, \cite[Theorem 1.11]{Pinkus09} and \cite[Theorems 9.4]{PetreolleSokalZhu23}.
\end{rem}

We are now in a position to show that the compositions defined in Proposition \ref{prop:link-with-compositions} satisfy the hypotheses of Corollary \ref{cor:sufficient-inequalities}.
We assumed that $w_0w_1>0$ at the beginning of this section, hence if additionally $\w$ has no internal zeros then its support has the form $\{\ell\colon 0\leq\ell\leq r\}$ for some $r=r(\w)\in\{1,2,\dots\}\cup\{\infty\}$.
Write then $\wb=(\PFtrees{\w}{n})_{n\geq 1}$ and for $\ell=0,\dots,r-1$ let $\shift \ell \w=(w_{i+\ell})_{i\geq 0}$ be the $\ell$-th shift of $\w$ to the left.

\begin{prop}\label{prop:checking-the-inequalities}
	Assume that $\w$ is log-concave and that $r=r(\w)$ is finite.
	Then the following holds for all $n\geq 0$,
	\begin{equation}\label{eq:checking-main-inequalities}
	\frac{\PFtrees{\w}{n+1}}{\PFtrees{\w}{n}}
	=
	\frac{
		\PFComp{\shift{r-1}\w,\wb}{n+1}
	}{
		\PFComp{\shift{r-1}\w,\wb}{n}}
	\leq
	\frac{
		\PFComp{\shift{r-2}\w,\wb}{n+1}
	}{
		\PFComp{\shift{r-2}\w,\wb}{n}}
	\leq
	\dots
	\leq
	\frac{
		\PFComp{\shift{1}\w,\wb}{n+1}
	}{
		\PFComp{\shift{1}\w,\wb}{n}}
	\leq
	\frac{
		\PFComp{\shift{0}\w,\wb}{n+1}
	}{
		\PFComp{\shift{0}\w,\wb}{n}}
	=
	\frac{\PFtrees{\w}{n+2}}{\PFtrees{\w}{n+1}},
	\end{equation}
	where we omit the ill-defined equality on the left when $n=0$.
\end{prop}

\begin{proof}
	Since $\shift{r-1}{\w}=(w_{r-1},w_r,0,0,\dots)$, we have $\PFComp{\shift{r-1}\w,\wb}{n}=w_r\cdot\PFtrees{\w}{n}$ for $n\geq 1$, by the definition from \eqref{eq:def-proba-compositions}.
	Therefore the equality on the left of \eqref{eq:checking-main-inequalities} holds for all $n\geq 1$.
	Now by Proposition \ref{prop:link-with-compositions}, we also have that $\PFtrees{\w}{n+1}=\PFComp{\w,\wb}{n}$ for all $n\geq 0$, which gives that the equality on the right of \eqref{eq:checking-main-inequalities} also holds for all $n\geq 0$.
	Let us check the inequalities in \eqref{eq:checking-main-inequalities} in the case $n=0$.
	By definition from \eqref{eq:def-proba-compositions}, they read
	\begin{align*}
	\frac{w_r\cdot\PFtrees{\w}{1}}{w_{r-1}}
	\leq
	\frac{w_{r-1}\cdot\PFtrees{\w}{1}}{w_{r-2}}
	\leq
	\cdots
	\leq
	\frac{w_2\cdot\PFtrees{\w}{1}}{w_{1}}
	\leq
	\frac{w_1\cdot\PFtrees{\w}{1}}{w_{0}},
	\end{align*}
	or equivalently $(w_i)^2\geq w_{i-1}w_{i+1}$ for $1\leq i\leq r-1$, which holds by log-concavity of $\w$.
	We finally verify the inequalities in \eqref{eq:checking-main-inequalities} when $n\geq 1$. Let $n\geq 1$ and consider, for some $\ell=0,\dots r-2$, the sum
	\begin{align}\label{eq:proof-key-inequalities:key-expression}
	\sum_{i,j\geq 0}
		(w_{i+\ell+1}\,w_{j+\ell}-w_{i+\ell}\,w_{j+\ell+1})
		\left(
			\PFforests{\w}{n,i}\,\PFforests{\w}{n+1,j}
			-	\PFforests{\w}{n+1,i}\,\PFforests{\w}{n,j}
		\right).
	\end{align}
	By Proposition \ref{prop:Toeplitz-TP2} and by Corollary \ref{cor:TP2-forests-partition-function}, each term in the sum is a product of two factors which are both non-negative when $i\leq j$ and both non-positive when $i\geq j$, as in the proof of Proposition \ref{prop:TP2-fnk}.
	Thus the expression in \eqref{eq:proof-key-inequalities:key-expression} is non-negative, and by developing the products and using that by definition we have
	\begin{align*}
	\PFComp{\shift \ell \w,\wb}{n}
	= \sum_{i\geq 0}w_{i+\ell}\PFforests{\w}{n,i}
	\qquad n\geq 1, \quad\ell\in\{0,\dots,r-1\},
	\end{align*}
	one obtains for $n\geq1$ and $\ell=0,\dots,r-2$ the inequality $	\PFComp{\shift{\ell+1}\w,\wb}{n+1}\PFComp{\shift{\ell}\w,\wb}{n}
	\leq
	\PFComp{\shift{\ell}\w,\wb}{n+1}\PFComp{\shift{\ell+1}\w,\wb}{n}$, which once re-written as
	\begin{align*}
	\frac{
		\PFComp{\shift{\ell+1}\w,\wb}{n+1}
	}{
		\PFComp{\shift{\ell+1}\w,\wb}{n}}
	\leq
	\frac{
		\PFComp{\shift{\ell}\w,\wb}{n+1}
	}{
		\PFComp{\shift{\ell}\w,\wb}{n}},
	\end{align*}
	are precisely the inequalities we need to conclude.
\end{proof}

\subsection{Proof of Theorems \ref{thm:main-thm-BGW} and \ref{thm:main-thm-SG}}

We recall that Theorems~\ref{thm:main-thm-BGW} and~\ref{thm:main-thm-SG} are equivalent, so that we only need to prove the latter.
Let $\w=(w_0,w_1,\dots)$ be a non-negative log-concave sequence with $w_0w_1>0$.
By Proposition \ref{prop:growing-comp-is-sufficient}, it is sufficient to show that $(\w,\wb)$ is an admissible weight pair where $\wb=(\PFtrees{\w}{n})_{n\geq0}$.
If $\w$ has finite support, that is if $r(\w)<\infty$, then this follows from Proposition \ref{prop:checking-the-inequalities} and Corollary \ref{cor:sufficient-inequalities}.
Otherwise $\w$ is the pointwise limit $i\rightarrow\infty$ of the log-concave and finitely supported sequences $\w_i=(w_0,w_1,\dots w_i,0,0,\dots)$, $i\geq1$, and we conclude by Lemma~\ref{lem:closure-admissibility}.\qed
%

%% file: parts/arithmetic-case.tex
\subsection{Random compositions with arithmetic conditions}

\mypar{Arithmetic conditions on compositions}
Let $d\geq 1$ and $s\in\{0,\dots,d-1\}$.
For $n\geq 0$ and a composition $c=\letter n_1\dots\letter n_i\composes n$, we say that $c$ satisfies the $(d,s)$-arithmetic condition if its number $i$ of parts has residue $s$ modulo $d$, and if each of its parts $(n_j)_{1\leq j\leq i}$ has residue $1$ modulo $d$.
In this case we write $c\composes[d,s]n$ and we denote by $\Compos[d,s]{n}$ the set of such compositions.
In particular for $n\geq0$, the set $\Compos[d,s]{n}$ is non-empty if and only if $n$ has residue $s$ modulo $d$.

\mypar{Modified covering and order relations}
Let $d\geq 1$.
We say that a composition $c$ is $d$-covered by a composition $c'$ and write $c\covered[d] c'$ if we can obtain $c'$ from $c$ by either incrementing by $d$ one of its parts or by adding $d$ new parts with value $1$ on its right.
More precisely, if $c={\letter{n}_1 \cdots   \letter{n}_r}$ then $c'$ covers $c$ if it is either one of the following compositions
\begin{align*}
\bigl(\letter{n}_1  \cdots  \letter{n}_{i-1}  \:\letter{\,n_i+d\,}\:  \letter{n}_{i+1} \cdots    \letter{n}_r\bigr),
\end{align*}
for some $i\in\{1,\dots,r\}$, or if it is the following composition
\begin{align*}
\letter{n}_1  \cdots   \letter{n}_r  \:\underbrace{\letter {\,1\,}\dots \:\letter {\,1\,}}_{\text{$d$ times}}.
\end{align*}
Notice that if $c\in\Compos[d,s]{n}$ for $n\geq 0$ with residue $s$ modulo $d$, then a composition $c'$ which $d$-covers $c$ also satisfies the $(d,s)$-arithmetic condition and thus is in $\Compos[d,s]{n+d}$.
The relation $\covered[d]$ is extended by transitivity into a partial order $\preceq^d$ on the set of all compositions of integers.

\begin{figure}[p]
	\begin{center}
		\includegraphics[page=4, scale=.6]{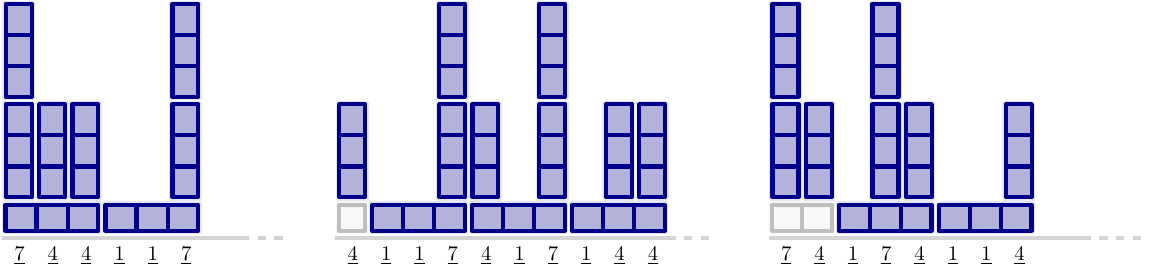}
	\end{center}
	\caption{A graphical representation of compositions satisfying respectively the $(3,0)$, $(3,1)$ and $(3,2)$ arithmetic conditions, as piles of square bricks.
		These square bricks have been grouped by $3$ except possibly the bottom left one, so that the arithmetic conditions are clearly visible.
		The parts of the composition are the numbers of \textit{squares} in each column.}
\end{figure}

\mypar{Arithmetic conditions on weight pairs}
Let $d\geq 1$ and let $s\in\{0,\dots,d-1\}$.
A weight pair $(\wa,\wb)$ will be called a $(d,s)$-\textit{non-degenerate} weight pair if the following conditions hold.
\begin{enumerate}
	\item 
	The sequence $\wa$ satisfies $a_s\neq0$ and its support is an interval in $\{s,d+s,2d+s,\dots\}$, not reduced to $\{0\}$ when $s=0$.
	More precisely its support is $\{\ell d+s\colon 0\leq\ell\leq r\}$ for some $r=r_{d,s}(\wa)\in\{0,1,2,\dots\}\cup\{+\infty\}$, with the requirement that $r_{d,s}(\wa)\geq 1$ when $s=0$.
	\item The sequence $\wb$ has support $\{1,d+1,2d+1,\dots\}$.
\end{enumerate}

\mypar{Corresponding random compositions}
Consider $d\geq1$ and $s\in\{0,1,\dots,d-1\}$, and let $(\wa,\wb)$ be a $(d,s)$-non-degenerate weight pair.
Consider as in the non-arithmetic case the partition function $\PFComp{\wa,\wb}n$ for every $n\geq0$,
\begin{align}\label{eq:def-Z-new}
\PFComp{\wa,\wb}n=
\sum_{\substack{i\geq 0\\\letter{n}_1  \cdots   \letter{n}_i\composes n}}a_i\cdot(b_{{n}_1}\cdots b_{{n}_i}).
\end{align}
The $(d,s)$-non-degeneracy of $(\wa,\wb)$ implies that compositions  which do not satisfy the $(d,s)$-arithmetic condition do not contribute to the sum.
Hence $n$ must have residue $s$ modulo $d$ for the last display to be non-zero.
Conversely if $n$ has residue $s$ modulo $d$, then $\PFComp{\wa,\wb}n\neq 0$ since when $s\neq 0$ the composition $\letter{n-(s-1)}\letter1\dots\letter1$ (with $s-1$ parts $\letter 1$ after the first part) contributes a non-zero term in the sum, while when $s=0$ the composition $\letter{n-(d-1)}\letter1\dots\letter1$ does.
All in all, we have for all $n\geq 0$ and all $s'\in\{0,\dots,d-1\}$,
\begin{align*}
\PFComp{\wa,\wb}{nd+s'}\neq0
\qquad \iff \qquad
s'=s.
\end{align*}
And in this case, for all $n\geq 0$ and all $\letter{n}_1  \cdots   \letter{n}_i\composes[d,s](nd+s)$, we can define
\begin{align}\label{eq:def-proba-compositions-new}
\ProbComp{\wa,\wb} {nd+s} {\letter{n}_1  \cdots   \letter{n}_i}&=
\frac{a_i\cdot (b_{{n}_1}\cdots b_{{n}_i})}{\PFComp{\wa,\wb} {dn+s}},
\quad
\PFComp{\wa,\wb}{nd+s}=
\sum_{\substack{i\geq 0\\\letter{n}_1  \cdots   \letter{n}_i\composes[d,s] n+s}}a_i\cdot(b_{{n}_1}\cdots b_{{n}_i}).
\end{align}
Notice that the last sum runs over compositions satisfying the $(d,s)$-arithmetic condition.
Observe, as in the non-arithmetic case that when $s=0$, the case $n=0$ is still well-defined if we make the convention that an empty product evaluates to $1$.

\begin{figure}[p]
	\begin{center}
		\includegraphics[page=5,scale=.6]{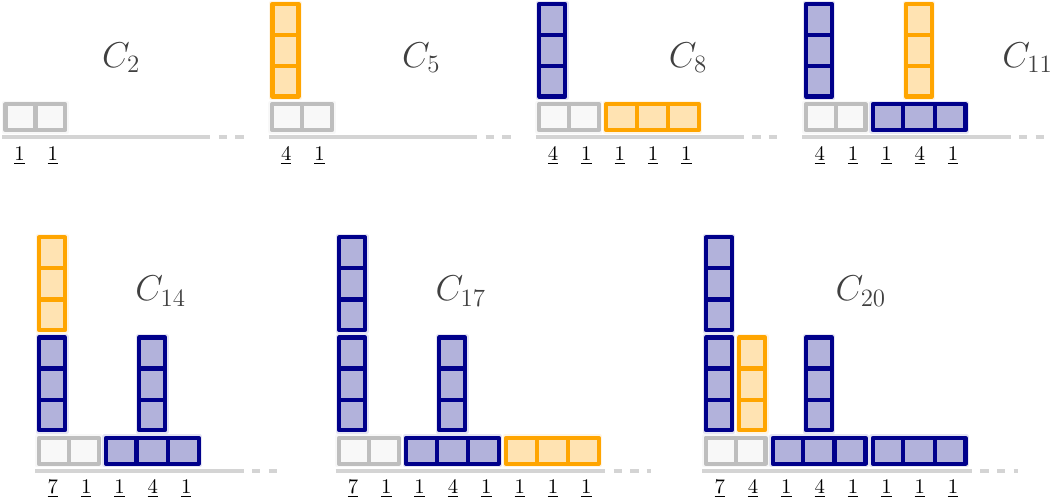}
	\end{center}
	\caption{An example of how of a sequence $C_2\covered[3] C_5\covered[3] C_8\covered[3]\dots$ may begin.
		At each step, the newly added group of $3$ bricks is colored in yellow.}
\end{figure}

\begin{figure}[p]
	\begin{center}
		\includegraphics[page=6,scale=.6]{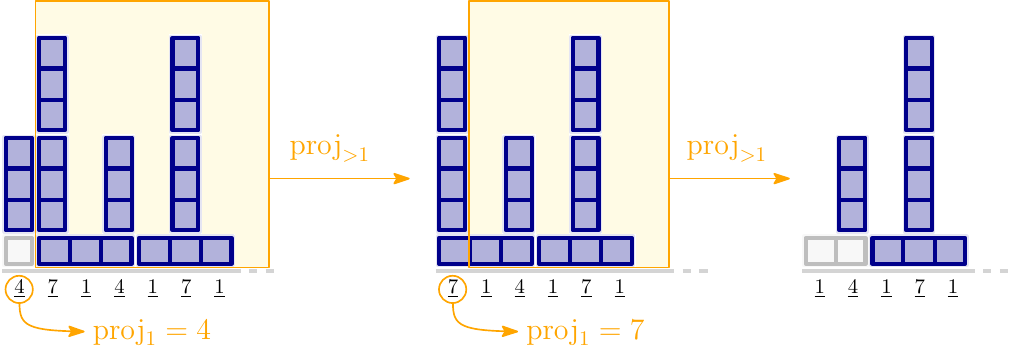}
	\end{center}
	\caption{An illustration of the action of ${\mathrm{proj}}_1$ and ${\mathrm{proj}}_{>1}$ on compositions satisfying the $(3,s)$-arithmetic condition with $s\in\{0,1,2\}$.
		Notice how the value of $s$ changes when applying ${\mathrm{proj}}_{>1}$.}
	\label{fig:proj-compositions-arithmetic}
\end{figure}

\begin{defin}[Modified admissibility]
	For $d\geq1$ and $s\in\{0,1,\dots,d-1\}$, we call $(d,s)$-\textit{admissible} the $(d,s)$-non-degenerate weight pairs $(\wa,\wb)$ for which the random compositions $(C_{nd+s})_{n\geq 0}$  with distributions $(\PComp{\wa,\wb}{nd+s})_{n\geq 0}$ can be coupled in such a way that
	\begin{align*}
	C_s\covered[d] C_{d+s}\covered[d] C_{2d+s}\covered[d] \dots\covered[d] C_{nd+s}\covered[d] C_{(n+1)d+s}\covered[d]\dots.
	\end{align*}
\end{defin}

A counterpart of Lemma~\ref{lem:closure-admissibility} then holds as follows.

\begin{prop}\label{prop:closure-admissibility-new}
	for $d\geq 1$ and $s\in\{0,\dots,d-1\}$, the set of $(d,s)$-admissible weight pairs is closed in the set of $(d,s)$-non-degenerate weight pairs equipped with the topology of pointwise convergence.
\end{prop}

\begin{proof}
	The proof of Lemma~\ref{lem:closure-admissibility} adapts easily.
\end{proof}

\subsection{Admissibility and shifted \texorpdfstring{$\wa$}{a}-weights, revisited}
We recycle from Section \ref{subsec:admissiblity-and-shift} the notation $\wa^+$ for the sequence $\wa$ shifted one unit to the left, $\projParts_1(c)$ for the first part of a composition $c$, and $\projParts_{>1}$ for the remaining parts, and we shall now provide adaptations of the results proved there.
First, note that for $d\geq 1$ and $s\in\{0,\dots,d-1\}$, the application $\projParts_1$ induces on each $\Compos[d,s]{nd+s}$ for $n\geq 0$ an application
\begin{align*}
\projParts_1&\colon
	\Compos[d,s]{nd+s}\longrightarrow\{1,d+1,2d+1,\dots\}.
\end{align*}
Also for $n\geq 1$ and a composition $\letter n_1\dots\letter n_i\composes[d,s](nd+s)$, if we write $n_1=n'_1 d+1$ then $\projParts_{>1}(\letter n_1\dots\letter n_i)=\letter n_2\dots\letter n_i$ is an element of
\begin{align*}
\begin{cases}
\Compos[d,s-1]{(n-n'_1)d+(s-1)}		&\qquad\text{if $s\neq0$},\\
\Compos[d,d-1]{(n-n'_1-1)d+(d-1)}	&\qquad\text{if $s=0$},
\end{cases}
\end{align*}
as illustrated in Figure \ref{fig:proj-compositions-arithmetic}.
Clearly, if $(\wa,\wb)$ is a $(d,s)$-non-degenerate weight pair with $r_{d,s}(\wa)\geq1$ if $s\neq 0$ or $r_{d,s}(\wa)\geq2$ if $s=0$, then $(\wa^+,\wb)$ is $(d,\overline{s-1})$-non-degenerate, where we use for all $m\in\Z$ the notation $\overline m$ to designate the residue of $m$ modulo $d$, which is an element of $\{0,\dots,d-1\}$.
Proposition \ref{prop:role-of-shifted-a} becomes the following.

\begin{prop}\label{prop:role-of-shifted-a_new}
	Let $d\geq1$ and $s\in\{0,\dots,d-1\}$.
	Consider $(\wa,\wb)$ a $(d,s)$-non-degenerate weight pair for which $(\wa^+,\wb)$ is $(d,\overline{s-1})$-non-degenerate.
	Let $C_{nd+s}$ be a $\PComp{\wa,\wb}{nd+s}$-distributed random composition  for some fixed $n\geq 0$.
	If we let $X_{nd+s}=\projParts_1(C_{nd+s})$ and $C'_{nd+s}=\projParts_{>1}(C_{nd+s})$, then when $s\neq0$ the variable $X_{nd+s}$ has distribution given by
	\begin{align}\label{eq:law-X_n_new}
	\Prob{X_{nd+s}=md+1}&=\frac{b_{md+1}\cdot\PFComp{\wa^+,\wb}{(n-m)d+(s-1)}}{\PFComp{\wa,\wb}{nd+s}},
	\qquad
	&&m=1,\dots,n,
	\intertext{while when $s=0$ the variable $X_{dn}$ has distribution given by}
	\Prob{X_{dn}=md+1}&=\frac{b_{md+1}\cdot\PFComp{\wa^+,\wb}{(n-m-1)d+(d-1)}}{\PFComp{\wa,\wb}{nd+s}},
	\qquad
	&&m=1,\dots,n-1,
	\end{align}
	Conditionally on $X_{nd+s}$, the composition $C'_{nd+s}$ has distribution $\PComp{\wa^+,\wb}{nd+s-X_n}$.
\end{prop}

\begin{proof}
	The proof adapts \textit{mutatis mutandi} from  that of Proposition \ref{prop:role-of-shifted-a}.
\end{proof}

Then, Proposition \ref{prop:admissibility-vs-shift} is turned into the following.

\begin{prop}\label{prop:admissibility-vs-shift-new}
	Let $d\geq1$ and $s\in\{0,1,\dots,d-1\}$.
	Let $(\wa,\wb)$ be a $(d,s)$-non-degenerate weight pair.
	Then there holds the following.
	\begin{enumerate}
		\item 
		\begin{enumerate}
			\item If $d=1$, $s=0$ and $r_{1,0}(\wa)=1$, that is if $\wa=(a_0,a_1,0,0,\dots)$, then $(\wa,\wb)$ is $(1,0)$ admissible.
			\item If $d\geq 2$, $s=1$ and $r_{d,s}(\wa)=0$, that is if $\wa=(0,a_1,0,0,\dots)$, then $(\wa,\wb)$  is $(d,1)$-admissible.
		\end{enumerate}
		\item
		\begin{enumerate}
			\item If $s\neq 0$ and the weight pair $(\wa^+,\wb)$ is $(d,s-1)$-admissible and if the inequalities
			\begin{align}\label{eq:ineq-shift-new}
			\frac{\PFComp{\wa^+,\wb}{(n+1-m)d+(s-1)}}{\PFComp{\wa^+,\wb}{(n-m)d+(s-1)}}
			\leq
			\frac{\PFComp{\wa,\wb}{(n+1)d+s}}{\PFComp{\wa,\wb}{nd+s}}
			\geq
			\frac{b_{md+1}}{b_{(m-1)d+1}},
			\end{align}
			are satisfied for $m=1,\dots,n$, then $(\wa,\wb)$ is $(d,s)$-admissible.
			\item If $s= 0$ and the weight pair $(\wa^+,\wb)$ is $(d,d-1)$-admissible and if the inequalities
			\begin{align}\label{eq:ineq-shift-new-2}
			\frac{\PFComp{\wa^+,\wb}{(n+1-m-1)d+(d-1)}}{\PFComp{\wa^+,\wb}{(n-m-1)d+(d-1)}}
			\leq
			\frac{\PFComp{\wa,\wb}{(n+1)d}}{\PFComp{\wa,\wb}{dn}}
			\geq
			\frac{b_{md+1}}{b_{(m-1)d+1}},
			\end{align}
			are satisfied for $m=1,\dots,n-1$, then $(\wa,\wb)$ is $(d,0)$-admissible.
		\end{enumerate}
	\end{enumerate}
\end{prop}

\begin{proof}
In cases 1.(a) and 1.(b), a sample of $\PComp{\wa,\wb}{nd+s}$ is a composition of $nd+s$ into one part for every $n\geq 0$, that is the trivial composition $\letter{nd+s}$.
Thus $(\wa,\wb)$ is trivially admissible in that case, since at each step the first and only part of the composition gets an increment of $d$.
Now for two compositions $c$ and $c'$ satisfying the $(d,s)$-arithmetic condition, we have the equivalence
\begin{align}\label{eq:equivalence-covering-relation-new}
c\preceq^d c' \qquad\iff\qquad\projParts_1(c)\leq \projParts_1(c')\quad\text{and}\quad \projParts_{>1}(c)\preceq^d \projParts_{>1}(c').
\end{align}
This generalizes \eqref{eq:equivalence-covering-relation}.
The proof of the second case in Proposition \ref{prop:admissibility-vs-shift} then adapts seamlessly to prove 2.(a) and 2.(b), using Proposition \ref{prop:admissibility-vs-shift-new} instead of Proposition \ref{prop:admissibility-vs-shift}.
\end{proof}

If $(\wa,\wb)$ is a $(d,0)$-non-degenerate weight pair with $d\geq 1$ and if we write $r=r_{d,0}(\wa)$, then we can consider for every $\ell\in\{0\dots,rd-1\}$ the $\ell$-th shift of $\wa$ to the left, which we notate $\shift{\ell} \wa$.
For every $q$ in $\{0,\dots,r-1\}$ and $s\in\{0,\dots,d-1\}$, the weight pair $(\shift{(qd+s)}\wa,\wb)$ is then $(d,\overline{d-s})$-non-degenerate.
Corollary \ref{cor:sufficient-inequalities} may be adapted as follows.

\begin{cor}\label{cor:sufficient-inequalities-new}
	Let $(\wa,\wb)$ be a $(d,0)$-non-degenerate weight pair with $d\geq 1$ and assume that $r=r_{d,0}(\wa)$ is finite.
	Assume that the following inequalities hold for all $n\geq 0$, all $q\leq q'$ in $\{0,\dots,r-1\}$, and all $s \leq s'$ in $\{0,\dots,d-1\}$,
	\begin{equation}\label{eq:main-inequalities-new}
	\frac{b_{nd+1}}{b_{(n-1)d+1}}
	\leq
	\frac{
		\PFComp{\shift{q' d+s'}\wa,\wb}{nd+(d-s')}
	}{
		\PFComp{\shift{q' d+s'}\wa,\wb}{(n-1)d+(d-s')}}
	\leq
	\frac{
		\PFComp{\shift{q d+s}\wa,\wb}{nd+(d-s)}
	}{
		\PFComp{\shift{q d+s}\wa,\wb}{(n-1)d+(d-s)}}
	\leq
	\frac{b_{(n+1)d+1}}{b_{nd+1}},
	\end{equation}
	where we omit the ill-defined inequality on the left when $n=0$.
	Then $(\wa,\wb)$ is $(d,0)$-admissible, and $(\shift{qd+s}\wa,\wb)$ is $(d,\overline{d-s})$-admissible for every $s\in\{0,\dots,d-1\}$ and $q\in\{0,\dots,r-1\}$.
\end{cor}

\begin{proof}
The proof is almost the same as that of Corollary \ref{cor:sufficient-inequalities}.
That is, the first case of Proposition \ref{prop:admissibility-vs-shift-new} proves that $(\shift{(r-1)d+(d-1)}\wa,\wb)$ is $(d,1)$-admissible, and then one propagates admissibility to the weight pairs $(\shift{\ell}\wa,\wb)$ in descending order of $\ell=qd+s$ using the second case of Proposition \ref{prop:admissibility-vs-shift-new} when $s\neq 0$ and the third case when $s=0$.
We let the reader check that by chaining the inequalities \eqref{eq:main-inequalities-new} as $n$ varies, just as in the proof of Corollary \ref{cor:sufficient-inequalities}, we indeed have all necessary inequalities to apply Proposition \ref{prop:admissibility-vs-shift-new} to the weight pairs $(\shift{\ell}\wa,\wb)$ for $\ell<(r-1)d+(d-1)$.
\end{proof}

\subsection{Simply generated trees in the arithmetic case}\label{subsec:defs-trees-new}

For $d\geq 1$ and $n\geq 0$, we let $\trees[d]_{nd+1}$ denote the set of rooted plane trees with $nd+1$ vertices in which all vertices have a number of children which is divisible by $d$, and $\trees[d]=\bigcup_{n\geq0}\trees[d]_{nd+1}$.
Such trees must indeed have a number of vertices whose residue modulo $d$ is $1$ since one may partition their vertex sets into first the root vertex and then the set of all vertices which are children of some vertex, the latter set having cardinality a multiple of $d$.
Conversely, any complete $d$-ary tree with $n\geq0$ interior vertices is an element of $\trees[d]_{nd+1}$.

\begin{defin}
	For $d\geq1$, a non-negative sequence $\w$ whose support takes the form $\{\ell d\colon 0\leq \ell\leq r\}$ for some $r=r_d(\w)\in\{1,2,\dots\}\cup\{+\infty\}$ will be called a \textit{$d$-arithmetic} sequence.
\end{defin}

The corresponding simply generated tree distributions are then supported in $\trees[d]$ and given for $n\geq 0$ and $\pt\in\trees[d]_{nd+1}$ by
\begin{align*}
\SimpGen{\w}{nd+1}(\pt)=\frac{\omega(\pt)}{\PFtrees{\w}{dn+1}},
\qquad\text{where}\qquad
\PFtrees{\w}{nd+1}
=\sum_{\pt\in\trees[d]_{nd+1}}\omega(\pt),
\end{align*}
where $\omega(\pt)=\prod_{\u\in\pt}w_{k\u(\pt)}$ for every $\pt\in\trees$.
Notice that the sum in the last display runs over $\trees[d]_{nd+1}$, since $\omega(\pt)$ may be non-zero only for $\pt\in \trees[d]_{nd+1}$, by $d$-arithmeticity of $\w$.
The other terms of the sequence $(\PFtrees{\w}{n})_{n\geq 0}$ are zero, that is
\begin{align*}
\PFtrees{\w}{nd+s}=0\qquad\text{for}\qquad n\geq0, \quad s\in\{0,\dots,d-1\}\setminus\{1\}.
\end{align*}
In Section \ref{subsec:defs-trees} we explained how Theorems \ref{thm:main-thm-BGW} and \ref{thm:main-thm-SG} are equivalent.
Similarly, Theorem \ref{thm:main-thm-BGW-arithmetic} is equivalent to the following statement in terms of simply generated trees.
\begin{thm}\label{thm:main-thm-SG-arithmetic}
	Let $d\geq 1$ and let $\w$ be a non-negative $d$-arithmetic sequence with $w_0w_d>0$.
	$(w_0,w_d,w_{2d},\dots)$ is log-concave, then the random trees with respective distributions $(\SimpGen{\w}{nd+1})_{n\geq 0}$ can be realized as a Markov process $(\T_{nd+1})_{n\geq 0}$ such that $\T_1\subset\T_{d+1}\subset\T_{2d+1}\subset\dots$.
\end{thm}
The remainder of Section \ref{sec:arithmetic-case} is dedicated to proving Theorem \ref{thm:main-thm-SG-arithmetic}, by providing suitable adaptations to the proof of Theorem \ref{thm:main-thm-SG}.

\subsection{Relationship with compositions satisfying arithmetic conditions}

We recall from Section \ref{subsec:link-with-compositions} some notation.
For $n\geq 1$ and $\pt\in\trees_n$ whose root has $i\geq 0$ children, $\phi(\pt)=(\pt^{(1)},\dots,\pt^{(i)})$ is the possibly empty collection where $\pt^{(j)}=\{\u\in\U\colon j\u\in\pt\}$ is the subtree of descendants of the $j$-th child of the root, $1\leq j\leq i$.
Recall also the definition of the composition $c(\pt)=\letter{n}_1 \cdots  \letter{n}_i$ where $n_j=\#{\pt^{(j)}}$ for $1\leq j\leq i$.

\begin{prop}\label{prop:link-with-compositions-new-1}
	Let $\w$ be a $d$-arithmetic weight sequence, $d\geq 1$, and let $\wb=(\PFtrees{\w}{n})_{n\geq 0}$.
	Then $(\w,\wb)$ is a $(d,0)$-non-degenerate weight pair.
\end{prop}

\begin{proof}
	The sequence $(\PFtrees{\w}{n})_{n\geq 0}$ has support included in $\{1,d+1,2d+1,\dots\}$ from Section \ref{subsec:defs-trees-new}.
	Conversely, any complete $d$-ary tree with $n\geq0$ interior vertices, and thus $nd+1$ vertices in total, gets non-zero $\omega$-weight since $w_0w_d>0$.
	Thus the support of $\wb$ is precisely $\{1,d+1,2d+1,\dots\}$.
	Since additionally $\w$ has support $\{\ell d\colon 0\leq \ell\leq r\}$ for some $r\in\{1,2,\dots\}\cup\{+\infty\}$, we have that $(\w,\wb)$ is indeed a $(d,0)$-non-degenerate weight pair.
\end{proof}

\begin{prop}\label{prop:link-with-compositions-new-2}
	Under the assumptions of Proposition \ref{prop:link-with-compositions-new-1}, we have for $n\geq 0$ the identity $\PFtrees\w {nd+1}=\PFComp{\w,\wb}{nd}$, and for $\pt\in\trees[d]_{nd+1}$ if we write $c(\pt)=\letter{n}_1 \cdots  \letter{n}_i$ and $\phi(\pt)=(\pt^{[1]},\dots,\pt^{[i]})$, then
	\begin{align}\label{eq:description-laws-trees-compositions-new}
	\SimpGen{\w}{nd+1}(\pt)=\PComp{\w,\wb}{n-1}\bigl(c(\pt)\bigr)\cdot\left(\SimpGen{\w}{n_1}\bigl(\pt^{[1]}\bigr)\cdots \SimpGen{\w}{n_i}\bigl(\pt^{[i]}\bigr)\right).
	\end{align}
\end{prop}

\begin{proof}
	The proof of Proposition \ref{prop:link-with-compositions} still works.
\end{proof}

\begin{prop}\label{prop:growing-comp-is-sufficient-new}
	Let $\w$ be a non-negative $d$-arithmetic sequence, $d\geq 1$, such that $w_0w_d>0$ and let $\wb=(\PFtrees{\w}{n})_{n\geq 1}$.
	If the weight pair $(\w,\wb)$ is $(d,0)$-admissible the random trees with respective distributions $(\SimpGen{\w}{nd+1})_{n\geq 0}$ can be realized as a Markov process $(\T_{nd+1})_{n\geq 0}$ in which at each step a right-leaning bouquet of $d$ leaves is added.
\end{prop}

\begin{proof}
	The proof of Proposition \ref{prop:growing-comp-is-sufficient} adapts with appropriate adjustments.
	In particular, for $n\geq 0$ the compositions $C_n\covered C_{n+1}$ in the proof are now replaced by $C_{nd}\covered[d]C_{(n+1)d}$ obtained by $(d,0)$-admissibility of $(\w,\wb)$.
\end{proof}

\subsection{Checking the inequalities}

Let $d\geq 1$ and fix $\w=(w_i)_{i\geq0}$ a non-negative $d$-arithmetic sequence.
The notation from Section \ref{subsec:checking-inequalities} is imported and we additionally denote by $\forests[d]_{nd+s,kd+s}$ the set of forests containing $kd+s$ trees form $\trees[d]$, with $nd+s$ vertices in total.
We have now for $s,\in\{0,\dots,d-1\}$ and for all $n,k\geq 0$,
\begin{align}\label{eq:def-PFforests-new}
\PFforests{\w}{nd+s,kd+s}=\sum_{\letter n_1 \cdots \letter n_{kd+s} \composes[d,s] nd+s}\PFtrees{\w}{ n_1}\cdots \PFtrees{\w}{ n_{kd+s}}
=\sum_{F\in\forests[d]_{nd+s,kd+s}} \omega(F),
\end{align}
and we also have $\PFforests{\w}{nd+s,kd+s'}=0$ for $s\neq s'$ in $\{0,\dots,d-1\}$.
In order to lighten notation, we let
\begin{align*}
F^s_{n,k}=\PFforests{\w}{nd+s,kd+s},\quad\text{for}\quad
	\begin{cases}
		n,k\geq0,\qquad s\in\{1,\dots,d-1\};\text{ or,}\\
		n, k\geq1,\qquad s=0.
	\end{cases}
\end{align*}
We let also $\W=(W_0,W_1,W_2,\dots)=(w_0,w_d,w_{2d},\dots)$ be the sequence defined by $W_j=w_{jd}$ for all $j\geq 0$.

\begin{lemma}\label{lem:recursion-fnk-new}
	The arrays $(F^s_{n,k})_{n,k\geq0}$ corresponding to adjacent values of $s\in\{0,\dots,d-1\}$ are linked by the following relations:
	\begin{align}
	F^s_{n,k}
	&=\sum_{j\geq 0}W_{j} \,F^{s-1}_{n,k+j},
	\qquad &&s\in\{1,\dots,d-1\},\label{eq:relation-fnk-different-s-new}\\
	F^0_{n,k}
	&=\sum_{j\geq 0}W_{j} \,F^{d-1}_{n-1,k+j-1},
	\qquad &&s=0,\label{eq:relation-fnk-different-s-new-2}
	\end{align}
	for all $s,n,k$ such that $nd+s\geq 2$ and $kd+s\geq1$.
	To give sense to \eqref{eq:relation-fnk-different-s-new} when $s=1$ and $k=0$, we make the convention that $F^0_{n,0}=0$ for all $n\geq1$.
\end{lemma}

\begin{proof}
	These identities may be written
	\begin{align*}
	\PFforests{\w}{nd+s,kd+s}
	&=\sum_{j\geq 0}w_{jd} \,\PFforests{\w}{nd+(s-1),(k+j)d+(s-1)},
	&&s\in\{1,\dots,d-1\},\\
	\PFforests{\w}{nd,kd}
	&=\sum_{j\geq 0}w_{jd} \,\PFforests{\w}{(n-1)d+(d-1),(k+j-1)d+(d-1)},
	&&s=0,
	\end{align*}
	where $s,n,k$ run over the specified ranges of indices.
	Those correspond to \eqref{eq:recursion} after taking into account that $\w$ is $d$-arithmetic.
	The proof of \eqref{eq:recursion} adapts without change.
\end{proof}

For $\ell\geq 1$ we denote by $\W^{*\ell}=(W^{*\ell}_j)_{j\geq0}$ the $\ell$-fold self-convolution of the sequence $\W$.

\begin{cor}\label{cor:TP2-fnk-new}
	For $s\in\{0,\dots,d-1\}$, the array $(F^s_{n,k})_{n,k\geq 0}$ satisfies the following recursion:
	\begin{align}\label{eq:recursion-fnk-new}
	F^s_{n,k}
	=\sum_{j\geq 0}W^{*d}_{j} \,F^s_{n-1,k+j-1},
	\quad
	\begin{cases}
		n\geq 1,\quad k\geq1,\quad s\in\{1,\dots,d-1\};\text{ or,}\\
		n\geq 2,\quad k\geq1,\quad s=0,\\
	\end{cases}
	\end{align}
	where we make the convention that $F^0_{n,0}=0$ for all $n\geq1$.
	When $s\in\{1,\dots,d-1\}$ and $k=0$, we have the modified recursion:
	\begin{multline}\label{eq:recursion-fnk-new-2}
		F^s_{n,0}
		=\sum_{j\geq 0}\Bigl(W^{*d}_{j}-W^{*s}_0\cdot W^{*(d-s)}_j\Bigr) \,F^s_{n-1,j-1},
		\\
			n\geq 1,\quad k=0,\quad s\in\{1,\dots,d-1\},
	\end{multline}
	where we make the convention that $F^s_{n,-1}=0$ for all $s\in\{1,\dots,d-1\}$ and $n\geq0$.
\end{cor}

\begin{proof}
	Let $s,n,k$ be as specified in~\eqref{eq:recursion-fnk-new}.
	Whenever this is legitimate, if we (i) apply $s$ times the relation \eqref{eq:relation-fnk-different-s-new}, (ii) apply one time the relation \eqref{eq:relation-fnk-different-s-new-2}, and (iii) apply the relation \eqref{eq:relation-fnk-different-s-new} another $d-s-1$ times; then we we get:
	\begin{align*}
	F^s_{n,k}\overset{(i)}{=}\sum_{j\geq 0}W^{*s}_{j} \,F^0_{n,k+j}
	\overset{(ii)}{=}\sum_{j\geq 0}W^{*(s+1)}_{j} \,F^{d-1}_{n-1,k+j-1}
	\overset{(iii)}{=}\sum_{j\geq 0}W^{*d}_{j} \,F^{d-1-(d-s-1)}_{n-1,k+j-1},
	\end{align*}
	which is precisely \eqref{eq:recursion-fnk-new}.
	These are indeed legitimate applications of Lemma~\ref{lem:recursion-fnk-new} provided that:
	\begin{align*}
		(i)\colon &\left\{
		\begin{matrix*}[l]
			nd+s' \geq 2\\
			(k+j)d+s' \geq 1
		\end{matrix*}
		\right.
		&&
			s'=1,\dots,s,&
			j\geq 0;
		\\
		(ii)\colon  &\left\{
		\begin{matrix*}[l]
			nd \geq 2\\
			(k+j)d \geq 1
		\end{matrix*}
		\right.
		&&
			&
			j\geq 0;
		\\
		(iii)\colon  &\left\{
			\begin{matrix*}[l]
				(n-1)d+s' \geq 2\\
				(k+j-1)d+s' \geq 1
			\end{matrix*}
			\right.
			&&
				s'=s+1,\dots,d-1,&
				j\geq 0.
	\end{align*}
	One then easily checks that, for the $s,n,k$ specified in~\eqref{eq:recursion-fnk-new}, the preceding inequalities are all satisfied.%
		\footnote{Notice in particular that $n\geq2$ when $d=1$ since the case $s\neq0$ is now non-existent. In particular, we have $nd\geq2$ even in this case, as needed.}
		Observe however that the second inequality in (ii) fails when $k=j=0$, hinting that the case $k=0$ may need to be treated separately.
	Let us now observe that when $n\geq1$, $k=0$ and $s\neq 0$; we get the claimed modified recursion upon making the convention that $F^s_{\cdot,-1}$ is identically zero. The steps are the same as above except that we use%
		\footnote{Dealing with the case $j_1=0$ separately is in fact necessary, since the identity $F^0_{n,j_1}=\sum_{j_2 \geq 0 }W_{j_2} \,F^{d-1}_{n-1,j_1+j_2-1}$ is \textit{not} valid when $j_1=0$ because the left-hand side is zero whereas the right-hand side is not.}
	the convention that $F^0_{n,0}=0$ between the first and second step. This gives:
	\begin{align*}
		F^s_{n,0}\overset{(i)}{=}\sum_{j_1\geq 0}W^{*s}_{j_1} \,F^0_{n,j_1}
			&= \sum_{j_1>0 } W^{*s}_{j_1}\,F^0_{n,j_1}\\
			&\overset{(ii)}{=} \sum_{j_1>0 } W^{*s}_{j_1}\,\sum_{j_2 \geq 0 }W_{j_2} \,F^{d-1}_{n-1,j_1+j_2-1}\\
			&\overset{(iii)}{=} \sum_{j_1>0 } W^{*s}_{j_1}\,\sum_{j_2,j_3 \geq 0 } W_{j_2} W^{*(d-s-1)}_{j_3} \,F^s_{n-1,j_1+j_2+j_3-1}\\
			&= \sum_{j> 0}\Bigl(W^{*d}_{j}-W^{*s}_0\cdot W^{*(d-s)}_j\Bigr) \,F^s_{n-1,j-1}.
	\end{align*}
	Making the convention that $F^s_{\cdot,-1}$ is identically zero, we can include $j=0$ in the last display, as claimed.
\end{proof}

We may now adapt Corollary \ref{cor:TP2-forests-partition-function}, in the form of the following proposition.

\begin{prop}\label{prop:TP2-forests-partition-function-new}
	Assume that $\w$ is $d$-arithmetic and that the sequence $\W=(w_0,w_d,w_{2d,\dots})$ is log-concave.
	Then, the arrays $(F^s, s\in\{0,\dots,d-1\})$, are TP2, that is we have:
	\begin{align}\label{eq:ineq-forests-new}
	F^s_{n,k}\,F^s_{n',k'}
	\geq F^s_{n,k'}\,F^s_{n',k}.
	&&\begin{cases}
		0\leq {n}<{n'},\quad  0\leq{k}<{k'}, & s\in\{1,\dots,d-1\}\\
		1\leq {n}<{n'},\quad  1\leq{k}<{k'}, & s=1.
	\end{cases}
	\end{align}
\end{prop}

\begin{proof}
	Let $A^0=(W^{*d}_{i-j})_{i,j\geq 0}$ be the Toeplitz matrix of $\W^{*d}$.
	Since convolution preserves log-concavity, the sequence $\W^{*d}$ is log-concave and thus $A^0$ is TP2 using Proposition~\ref{prop:Toeplitz-TP2}.
	Hence, in the case $s=0$, we obtain the result by applying Proposition \ref{prop:TP2-fnk} with the arrays $A=A^0$ and $F=(F^0_{n+1,k+1})_{n,k\geq0}$, and the recursion given in Corollary \ref{cor:TP2-fnk-new}.

	Now let $s\in\{1,\dots,d-1\}$. The situation is slightly more involved in this case since the recursion in Corollary~\ref{cor:TP2-fnk-new} behaves differently on the coordinate $k=0$. This recursion can be written as:
	\begin{align*}
		F^s_{n,k}=\sum_{j\geq0} A^s_{j,k} F^s_{n-1,k-1}, && n\geq1, k\geq0,
	\end{align*}
	where we write:
	\begin{align*}
		A^s_{j,k}=W^{*d}_{j-k}-\indic{\{k=0\}}\cdot W^{*s}_0 \cdot W^{(d-s)}_j && j,k\geq0.
	\end{align*}
	By Proposition \ref{prop:TP2-fnk}, in order to get the result it is sufficient to prove that $A^s$ is TP2. Outside the $k=0$ column, the array $A^s$ coincides with the Toeplitz matrix of the log-concave sequence $\W^{*d}$, so that the corresponding $2\times 2$ minors are non-negative by Proposition~\ref{prop:Toeplitz-TP2}. Let us now consider a generic $2\times2$ minor \textit{involving} the column $k=0$, that is we fix $k'>0$ and $0\leq j\leq j'$ and we shall prove that $A^s_{j,0}A^s_{j',k'}\geq A^s_{j',0} A^s_{j,k'}$. This inequality reads as follows:
	\begin{align*}
		\bigl(W^{*d}_{j}- W^{*s}_0 \cdot W^{(d-s)}_j\bigr)\cdot W^{*d}_{j'-k'}
			\geq \bigl(W^{*d}_{j'}- W^{*s}_0 \cdot W^{(d-s)}_{j'}\bigr)\cdot W^{*d}_{j-k'},
	\end{align*}
	where we recall the convention that $u_{-1}=u_{-2}=\dots=0$ for a sequence $(u_0,u_1,\dots)$ indexed by $\Z_+$.
	By exploiting that $\W^{*d}$ is the convolution of $\W^{*s}$ and $\W^{*(d-s)}$, the latter inequality can be re-written as:
	\begin{multline*}
		\Bigl[\sum_{i\geq0} W^{*(d-s)}_{j-i} \cdot \indic{\{i\neq0\}}W^{*s}_{i}\Bigr]
			\cdot
			\Bigl[\sum_{i\geq0} W^{*(d-s)}_{j'-i} \cdot W^{*s}_{i-k'}\Bigr]\\
		\geq
		\Bigl[\sum_{i\geq0} W^{*(d-s)}_{j'-i} \cdot \indic{\{i\neq0\}}W^{*s}_{i}\Bigr]
			\cdot
			\Bigl[\sum_{i\geq0} W^{*(d-s)}_{j-i} \cdot W^{*s}_{i-k'}\Bigr].
	\end{multline*}
	In order to prove this inequality, consider the following expression:
	\begin{multline*}
		\sum_{i_1,i_2\geq0}
			\Bigl(W^{*(d-s)}_{j-i_1}\cdot W^{*(d-s)}_{j'-i_2}-W^{*(d-s)}_{j-i_2}\cdot W^{*(d-s)}_{j'-i_1}\Bigr)\\
			\cdot
			\Bigl(\indic{\{i_1\neq0\}} W^{*s}_{i_1} \cdot W^{*s}_{i_2-k'}-\indic{\{i_2\neq0\}} W^{*s}_{i_2}\cdot W^{*s}_{i_1-k'}\Bigr).
	\end{multline*}
	Log-concavity being preserved by convolution, the sequence $\W^{*(d-s)}$ is log-concave so that by Proposition~\ref{prop:Toeplitz-TP2}, the expression between the first set of parentheses is non-negative when $i_1\leq i_2$, and non-positive when $i_1\geq i_2$. The same goes for the expression between the second set of parentheses: (i) with the same reasoning applied to $\W^{*s}$ when both indicators take the value 1, and (ii) by observing that the expression is trivially non-negative when $0=i_1<i_2$ and non-positive when $i_1>i_2=0$. Hence the expressions between both sets of parentheses have the same sign, so that the sum in the last display is non-negative. This gives the desired inequality by developing the product.
\end{proof}

Let us introduce a bit of notation for readability.
Assume that $\w$ is $d$-arithmetic and that $r=r_{d}(\w)$ is finite, and  write $\wb=(\PFtrees{\w}{n})_{n\geq 1}$.
For every $\ell\in\{0,\dots,rd-1\}$, we let $\shift\ell\w$ be the $\ell$-th shift of $\w$ to the left.
For $q$ in $\{0,\dots,r-1\}$ and every $s\in\{0,\dots,d-1\}$, the weight pair $(\shift{qd+s}\w,\wb)$ is then $(d,\overline{d-s})$-non-degenerate.
We set for $n\geq0$,
\begin{align*}
R_n(q,s)=
\frac{
	\PFComp{\shift{qd+s}\w,\wb}{nd+(d-s)}
}{
	\PFComp{\shift{qd+s}\w,\wb}{(n-1)d+(d-s)}},
\quad
\begin{cases}
q\in\{0,\dots,r-1\},\,s\in\{0,\dots,d-1\}	&\text{if $n\geq 1$,}\\
q\in\{0,\dots,r-1\},\,s=0	&\text{if $n=0$}.
\end{cases}
\end{align*}
We can now provide a generalization of Proposition \ref{prop:checking-the-inequalities}.

\begin{prop}\label{prop:checking-the-inequalities-new}
	Assume that $\w$ is $d$-arithmetic, that $\W=(w_0,w_d,w_{2d,\dots})$ is a log-concave sequence, and that $r=r_{d}(\w)$ is finite.
	Then we have
	\begin{align}\label{eq:ineq-partfunct-arithmetic-case}
	\tfrac{\PFtrees{\w}{d+1}}{\PFtrees{\w}{1}}
	= R_0(0,0)\geq R_0(1,0)\geq\dots\geq R_0\bigl(r-1,0\bigr),
	\end{align}
	and for all $n\geq1$, we have the following long chain of inequalities, where each line continues the preceding one,
	\begin{align*}
	\begin{matrix}
		&	 R_n(0,0)
			&\geq& 	R_n(1,0)
				&\geq&	\dots
					&\geq& 	R_n\bigl(r-1,0\bigr)
								\\
		\geq&  R_n(0,1)
			&\geq& 	R_n(1,1)
				&\geq&	\dots
					&\geq& 	R_n\bigl(r-1,1\bigr)
								\\
		&  \vdots
			&& 	\vdots
				&&	\vdots
					&& 	\vdots
								\\
		\geq&  R_n(0,d-2)
			&\geq& 	R_n(1,d-2)
				&\geq&	\dots
					&\geq& 	R_n\bigl(r-1,d-2\bigr)
								\\
		\geq&  R_n(0,d-1)
			&\geq& 	R_n(1,d-1)
				&\geq&	\dots
					&\geq& 	R_n\bigl(r-1,d-1\bigr).
	\end{matrix}
	\end{align*}
	The first and last term of this chain of inequalities read as follows:
	\begin{align*}
	 R_n(0,0) 
	 	= \frac{\PFtrees{\w}{(n+1)d+1}}{\PFtrees{\w}{nd+1}}
	 \qquad\text{and}\qquad
	 R_n\bigl(r-1,d-1\bigr) 
	 	= \frac{\PFtrees{\w}{nd+1}}{\PFtrees{\w}{(n-1)d+1}}.
	\end{align*}
\end{prop}

Before proceeding with the proof of Proposition \ref{prop:checking-the-inequalities-new}, we will need the two following lemmas.

\begin{lemma}\label{lem:expr-Z-arithmetic-case}
	For $n\geq 0$ we have the expressions
	\begin{align}
	\PFComp{\shift{qd+s}\w,\wb}{nd+(d-s)}
	&=\sum_{i\geq0}W_{i+q+1}F^{d-s}_{n,i},
	&&q\in\{0,\dots,r-1\},
	&&s\in\{1,\dots,d-1\},\label{eq:expression-Z-1}\\
	\PFComp{\shift{qd+0}\w,\wb}{nd+(d-0)}
	&=\sum_{i\geq0}W_{i+q}F^{0}_{n+1,i},
	&&q\in\{0,\dots,r-1\},
	&&s=0.\label{eq:expression-Z-2}
	\end{align}
\end{lemma}

\begin{lemma}\label{lem:expr-Z-arithmetic-case-2}
	For $n\geq 1$ we have the expressions
	\begin{align}
	\PFComp{\shift{(r-1)d+s}\w,\wb}{nd+(d-s)}
	&=W_r\sum_{i\geq0}W_{i}F^{d-s-1}_{n,i},
	&&s\in\{0,\dots,d-1\}\label{eq:expression-Z-3}.
	\end{align}
\end{lemma}

The proofs of Lemmas \ref{lem:expr-Z-arithmetic-case} and \ref{lem:expr-Z-arithmetic-case-2} are postponed to the end of this section.

\begin{proof}[Proof of Proposition \ref{prop:checking-the-inequalities-new}]
	By Proposition \ref{prop:link-with-compositions-new-2}, we have $\PFtrees{\w}{nd+1}=\PFComp{\w,\wb}{nd}$ for all $n\geq 0$.
	This yields the claimed identity for the first term of the sequence of inequalities,
	\begin{align*}
	R_n(0,0) = \frac{\PFtrees{\w}{(n+1)d+1}}{\PFtrees{\w}{nd+1}},\qquad n\geq0.
	\end{align*}
	Now notice that $\shift{(r-1)d+d-1}{\w}=\shift{rd-1}{\w}$ is of the form $(w_{rd-1},w_{rd},0,0,\dots)$ if $d=1$ or $(0,w_{rd},0,0,\dots)$ if $d\geq 2$.
	Thus we have $\PFComp{\shift{rd-1}\w,\wb}{(n-1)d+1}
	=w_{dr}\cdot\PFtrees{\w}{(n-1)d+1}$ for all $n\geq 1$ by definition in \eqref{eq:def-proba-compositions-new}.
	Hence the claimed identity for the last term of the sequence of inequalities,
	\begin{align*}
	R_n(r-1,d-1)=
	\frac{\PFtrees{\w}{nd+1}}{\PFtrees{\w}{(n-1)d+1}},\qquad n\geq1.
	\end{align*}
	Now the inequalities in \eqref{eq:ineq-partfunct-arithmetic-case} correspond by definition to the inequalities
	\begin{align*}
	{\frac{w_d\cdot\PFtrees{\w}{1}}{w_{0}}}
	\geq
	{\frac{w_{2d}\cdot\PFtrees{\w}{1}}{w_{d}}}
	\geq
	\cdots
	\geq
	{\frac{w_{(r-1)d}\cdot\PFtrees{\w}{1}}{w_{(r-2)d}}}
	\geq
	{\frac{w_{rd}\cdot\PFtrees{\w}{1}}{w_{(r-1)d}}},
	\end{align*}
	which are equivalent to the statement that $\W=(w_0,w_d,w_{2d,\dots})$ is log-concave.
	In order to conclude we are left with proving the following
	\begin{enumerate}
		\item
			The inequality $R_n(q,s)\geq R_n(q+1,s)$ for all $n\geq1$, $q\in\{0,\dots,r-2\}$, and $s\in\{0,\dots,d-1\}$.
			Using Lemma \ref{lem:expr-Z-arithmetic-case}, this reads as follows when $s\neq0$,
			\begin{align}\label{eq:long-chain-left-to-prove-1}
			\frac{
				\sum_{i\geq0}W_{i+q+1}F^{d-s}_{n,i}
			}{
				\sum_{i\geq0}W_{i+q+1}F^{d-s}_{n-1,i}}
			\geq
			\frac{
				\sum_{i\geq0}W_{i+q+2}F^{d-s}_{n,i}
			}{
				\sum_{i\geq0}W_{i+q+2}F^{d-s}_{n-1,i}},
			\end{align}
			while when $s=0$ it reads
			\begin{align}\label{eq:long-chain-left-to-prove-1bis}
			\frac{
				\sum_{i\geq0}W_{i+q}F^{0}_{n+1,i}
			}{
				\sum_{i\geq0}W_{i+q}F^{0}_{n,i}}
			\geq
			\frac{
				\sum_{i\geq0}W_{i+q+1}F^{0}_{n+1,i}
			}{
				\sum_{i\geq0}W_{i+q+1}F^{0}_{n,i}}.
			\end{align}
		\item
			The inequality $R_n(r-1,s)\geq R_n(0,s+1)$ for all $n\geq1$ and all $s\in\{0,\dots,d-2\}$.
			Using Lemma \ref{lem:expr-Z-arithmetic-case-2} to express $R_n(r-1,s)$ and Lemma \ref{lem:expr-Z-arithmetic-case} to express $R_n(0,s+1)$, this reads
			\begin{align}\label{eq:long-chain-left-to-prove-2}
			\frac{
				\sum_{i\geq 0}W_{i} \,F^{d-s-1}_{n,i}
			}{
				\sum_{i\geq 0}W_{i} \,F^{d-s-1}_{n-1,i}}
			\geq
			\frac{
				\sum_{i\geq0}W_{i+1}F^{d-s-1}_{n,i}
			}{
				\sum_{i\geq0}W_{i+1}F^{d-s-1}_{n-1,i}}.
			\end{align}
	\end{enumerate}
	The inequalities \eqref{eq:long-chain-left-to-prove-1}, \eqref{eq:long-chain-left-to-prove-1bis} and \eqref{eq:long-chain-left-to-prove-2} which we need to prove are all instances of the following inequalities
	\begin{align}\label{eq:final-inequ-to-prove}
	\frac{
		\sum_{i\geq0}W_{i+q}F^{s}_{n,i}
	}{
		\sum_{i\geq0}W_{i+q}F^{s}_{n-1,i}}
	\geq
	\frac{
		\sum_{i\geq0}W_{i+q+1}F^{s}_{n,i}
	}{
		\sum_{i\geq0}W_{i+q+1}F^{s}_{n-1,i}},
	&&n\geq1,
	&& 0\leq q<r,
	&& 0\leq s<d.
	\end{align}
	The proof of these inequalities is similar to the proof of \eqref{eq:checking-main-inequalities} in Proposition \ref{prop:checking-the-inequalities}.
	Namely we can consider the expression
	\begin{align*}
	\sum_{i,j\geq 0}
	(W_{i+q+1}W_{j+q}-W_{i+q}W_{j+q+1})
		\left(
		F^{s}_{n-1,i}\,F^{s}_{n,j}
		-	F^{s}_{n,i}\,F^{s}_{n-1,j}
		\right),
	\end{align*}
	and then use Proposition \ref{prop:Toeplitz-TP2} for the sequence $\W=(w_0,w_d,w_{2d},\dots)$---which is assumed to be log-concave---as well as Proposition \ref{prop:TP2-forests-partition-function-new} for the array $(F^{s}_{n,k})_{n,k\geq0}$ to deduce as in the proof of Proposition \ref{prop:checking-the-inequalities} that this expression is non-negative.
	By developing the products we obtain the desired inequalities \eqref{eq:final-inequ-to-prove}.
\end{proof}

It remains to prove Lemmas \ref{lem:expr-Z-arithmetic-case} and \ref{lem:expr-Z-arithmetic-case-2}.

\begin{proof}[Proof of Lemma \ref{lem:expr-Z-arithmetic-case}]
	Let  $n\geq 0$ and $\ell\in\{0,\dots,rd-1\}$.
	By comparing the expressions \eqref{eq:def-PFforests-new} and \eqref{eq:def-proba-compositions-new} we see that
	\begin{align*}
	\PFComp{\shift{\ell}\w,\wb}{nd+(d-s)}=
	\sum_{i\geq 0}w_{i+\ell}\PFforests{\w}{nd+(d-s),i}.
	\end{align*}
	Since $\w$ is $d$-arithmetic, the indices $i$ in the sum for which the corresponding term may be non-zero are those for which $i+\ell$ is a multiple of $d$.
	If $\ell=qd+s$ for some $q\in\{0,\dots,r-1\}$ and $s\in\{1,\dots,d-1\}$, then they have the form $i=i'd+{(d-s)}$ for some $i'\geq0$, and we have
	\begin{align*}
	\PFComp{\shift{qd+s}\w,\wb}{nd+(d-s)}=
	\sum_{i\geq 0}w_{i+(qd+s)}\,\PFforests{\w}{nd+(d-s),i}
	=\sum_{i'\geq0}w_{(i'+q+1)d}\,\PFforests{\w}{nd+(d-s),i'd+(d-s)},
	\end{align*}
	which is precisely \eqref{eq:expression-Z-1}.
	If now $\ell=qd$ for some $q\in\{0,\dots,r-1\}$, then the aforementioned indices have the form $i=i'd$ for some $i'\geq0$, and we have
	\begin{align*}
	\PFComp{\shift{qd}\w,\wb}{nd+d}=
	\sum_{i\geq 0}w_{i+qd}\,\PFforests{\w}{nd+d,i}
	=\sum_{i'\geq0}w_{(i'+q)d}\,\PFforests{\w}{(n+1)d,i'd},
	\end{align*}
	which gives \eqref{eq:expression-Z-1}.
\end{proof}

\begin{proof}[Proof of Lemma \ref{lem:expr-Z-arithmetic-case-2}]
	Let $n\geq0$.
	We first treat the case $s\in\{1,\dots,d-1\}$.
	By definition of $r=r_d(\w)$, the sequence $(W_{i+r})_{i\geq0}$ has the form $(W_r,0,0,\dots)$ so that \eqref{eq:expression-Z-1} becomes for $s\in\{1,\dots,d-1\}$,
	\begin{align*}
	\PFComp{\shift{(r-1)d+s}\w,\wb}{nd+(d-s)}
	=\sum_{i\geq0}W_{i+r}F^{d-s}_{n,i}
	&=W_r F^{d-s}_{n,0},
	\end{align*}
	and we obtain \eqref{eq:expression-Z-3} by applying \eqref{eq:relation-fnk-different-s-new} from Lemma \ref{lem:recursion-fnk-new}.
	Similarly the sequence $(W_{i+r-1})_{i\geq0}$ has the form $(W_{r-1},W_{r},0,0,\dots)$ so that \eqref{eq:expression-Z-2} becomes for $s=0$,
	\begin{align*}
	\PFComp{\shift{qd+0}\w,\wb}{nd+(d-0)}
	&=\sum_{i\geq0}W_{i+q}F^{0}_{n+1,i}
	=W_{r-1} F^{0}_{n+1,0} + W_{r} F^{0}_{n+1,1}
	=W_{r} F^{0}_{n+1,1},
	\end{align*}
	where the last equality comes from the fact that $F^{0}_{n+1,0}=0$.
	We therefore obtain the $s=0$ case of \eqref{eq:expression-Z-3} by applying \eqref{eq:relation-fnk-different-s-new-2} from Lemma \ref{lem:recursion-fnk-new}.
\end{proof}

\subsection{Proof of Theorems \ref{thm:main-thm-BGW-arithmetic} and \ref{thm:main-thm-SG-arithmetic}}

Theorems~\ref{thm:main-thm-BGW-arithmetic} and~\ref{thm:main-thm-SG-arithmetic} are equivalent, so let us prove the latter.
The proof is similar to that of Theorem \ref{thm:main-thm-SG}, but using our modified intermediary results
Let $d\geq 1$ and let $\w$ be a non-negative $d$-arithmetic sequence with $w_0w_d>0$ such that $\W=(w_0,w_d,w_{2d},0,0,\dots)$ is log-concave.
By Proposition \ref{prop:growing-comp-is-sufficient-new}, it is sufficient to show that $(\w,\wb)$ is a $(d,0)$-admissible weight pair where $\wb=(\PFtrees{\w}{n})_{n\geq0}$.
If $\w$ has finite support, that is if $r_d(\w)<\infty$, then this follows from Proposition \ref{prop:checking-the-inequalities-new} and Corollary \ref{cor:sufficient-inequalities-new}.
Otherwise $\w$ is the pointwise limit $i\rightarrow\infty$ of the log-concave, $d$-arithmetic, and finitely supported sequences $\w_i=(w_0,w_1,\dots w_i,0,0,\dots)$, $i\geq1$, and we conclude by Proposition \ref{prop:closure-admissibility-new}.\qed

%% file: parts/application-random-subtrees.tex
Let us recall some of the definitions from the introduction.
A \textit{rooted subtree} of $\U$ is a non-empty finite subset $\t$ of $\U$ such that:
\begin{align*}
\forall\u\in\U,\,\forall i\in\{1,2,\dots\},	\qquad
\u i\in\t
\qquad\implies\qquad
\u\in\t.
\end{align*}
We also recall that a \textit{plane tree} is a rooted subtree $T$ of $\U$ such that:
\begin{align*}
\forall\u\in\U,\,\forall i\in\{1,2,\dots\},	\qquad
\u i\in\pt
\qquad\implies\qquad
\u j\in\pt,\quad 1\leq j\leq i.
\end{align*}
The set of \textit{rooted subtrees} of $\U$, resp.~the set of those having $n$ vertices, is denoted by $\subtrees$, resp.~$\subtrees_n$, $n\geq1$, while the set of \textit{plane trees}, resp.~the set of those having $n$ vertices, is denoted by $\trees$, resp.~$\trees_n$, $n\geq1$.

Lastly, let us recall the probabilistic model we introduced on rooted subtrees of $\U$.
For $\wtheta=(\theta_1,\theta_2,\dots)$ a non-negative sequence such that $0<\sum_i \theta_i<\infty$, and for $n\geq1$, the measure $\Sub{\wtheta}{n}$ on $\subtrees_n$ is defined as follows:
\begin{align}
\forall \t\in\subtrees_n,\quad
\Sub{\wtheta}{n}(\t)=\frac{\prod_{i\geq 1}\theta_i^{\# V_i(\t)}}{\PFsubtrees{\wtheta}{n}},
\qquad \PFsubtrees{\wtheta}{n}=\sum_{\t\in\trees_n}\prod_{i\geq 1}\theta_i^{\# V_i(\t)},
\end{align}
where we recall that $N_i(\t)=\#\, V_i(\t)$, where $V_i(\t)$ is the set of vertices of type $i$ in $\t$, that is those vertices $\v\in\t$ which have the form $\v=u i$ for some $\u\in\U$.

\subsection{Moving subtrees around}
\label{subsec:moving-subtrees}

\begin{defin}
	For $\t\in\subtrees$ and a vertex $\u\in\t$, we define $\Pos(\t)=(\pos_\u(\t))_{\u\in\t}$ the collection of \textit{children-positions}, where:
	\begin{align*}
	\pos_\u(\t)=\bigl\{i\in\{1,2,\dots\}\colon\u i\in \t\bigr\},\qquad\u\in\t.
	\end{align*}
\end{defin}

\begin{rem}\label{rem:plane-trees-and-positions}
	Notice that $\t\in\subtrees$ is a plane tree if and only if for every $\u\in\t$, the set of children-positions $\pos_\u(\t)$ has the form $\{1,2,\dots,k\}$ for some $k\geq1$.
\end{rem}
Now, if for every $\u\in\t$ we are given an injective mapping $\g_\u\colon\pos_\u(\t)\to\{1,2,\dots\}$, then we can ``shuffle'' $\t$ using the collection $\boldg=(\g_\u)_{\u\in\t}$ by ``re-positioning'' the children of each vertex $\u$ using $\g_\u$.
Let us formalize this idea.

For a finite subset $S\subset\{1,2,\dots\}$, we denote by $\G_S$ the set of injective mappings $\g\colon S\to\{1,2,\dots\}$.
We will use the notation:
\begin{align*}
\G=\bigsqcup_{S\subset\{1,2,\dots\},\,|S|<\infty} \G_S.
\end{align*}
For $\t\in\subtrees$, we denote by $\G(\t)$ the set of collections $\boldg=(\g_\u)_{\u\in\t}$ such that for all $\u\in\t$ we have $\g_\u\in\G_{\pos_\u(\t)}$.
In this setup, for $\t\in\subtrees$ and $\boldg\in\G(\t)$, and for every $\u=(u_1,\dots,u_h)\in\t$, with ancestral line $\u_i=(u_1,\dots,u_i)$, ${0\leq i\leq h}$, we can define
\begin{align}\label{eq:def-action-g}
\boldg\cdot\u=\bigl(\g_{\u_0}(u_1),\g_{\u_1}(u_2),\dots,\g_{\u_{h-1}}(u_h)\bigr).
\end{align}
It is easily verified that $\boldg\cdot\t$ is a rooted subtree of $\U$, where for $V\subseteq \t$, we write $\boldg\cdot V=\{\boldg\cdot\u\colon\u\in V\}$.

For some finite $S\subset\{1,2,\dots\}$ and some $\g\in\G_S$, since $\g$ is injective, it admits an inverse on its image, which we denote by $\g^{-1}\colon\g(S)\to\{1,2,\dots\}$, thus forming an element of $\G_{g(S)}\subset\G$.
Our first observation, in the following lemma, is that the ``shuffling'' operation we have defined is bijective, with an explicit inverse.

\begin{lemma}\label{lem:inverse-g}
	For every $\t\in\subtrees$ and $\boldg\in\G(\t)$, there exists a unique $\boldg^{-1}\in\G(\boldg\cdot\t)$ such that $\boldg^{-1}\cdot(\boldg\cdot\u)=\u$ for all $\u\in \t$, and $\boldg\cdot(\boldg^{-1}\cdot\u')=\u'$ for all $\u'\in\boldg\cdot\t$.
\end{lemma}

\begin{proof}
	Let  $\t\in\subtrees$ and $\boldg=(\g_\u)_{\u\in\t}\in\G(\t)$.
	Let $\u=(u_1,\dots,u_h)\in\t$ with ancestral line $\u_i=(u_1,\dots,u_i)$, ${0\leq i\leq h}$, and form $\u'=\boldg\cdot\u$, that is $\u'=(u'_1,\dots,u'_h)$ with $u'_j=\g_{\u_{j-1}}(u_j)$ for $1\leq j\leq h$.
	Then $u_1=\g_{\emptyset}^{-1}(u'_1)$ and by an immediate induction we have
	\begin{align}\label{eq:inverting-g}
	u_j=\g_{\u_{j-1}}^{-1}(u'_j),\qquad 1\leq j\leq h.
	\end{align}
	Hence $\u$ is completely determined by $\boldg$ and $\u'$, and the last display tells that $\u=\boldg'\cdot\u'$ where $\boldg'=(\g'_{\v'})_{\v'\in\boldg\cdot\t}$ is obtained by setting for every $\v'\in\boldg\cdot\t$ the value $\g'_{\v'}=\g_{\v}^{-1}$, where $\v$ is the only preimage of $\v'$ by $\v\mapsto\boldg\cdot\v$.
	Then $\boldg'\cdot(\boldg\cdot\u)=\boldg'\cdot\u'=\u$ for all $\u\in \t$, and $\boldg\cdot(\boldg'\cdot\u')=\boldg\cdot\u=\u'$ for all $\u'\in\boldg\cdot\t$.
\end{proof}

\begin{rem}
	The algebraic structure we obtain is that of a \textit{groupoid}, which morally generalizes the notion of a group by lifting the restriction that all pairs of elements can be multiplied.
	In our case, the elements of the groupoid are the $\boldg\in\bigsqcup_{\t\in\subtrees}\G(\t)$, which are all invertible, but which can only be multiplied with elements satisfying the natural compatibility conditions.
\end{rem}

\paragraph{The push-forward operation}
For the remainder of this section, we fix some $\t\in\subtrees$.
For any set $\bbX$ and for all $\boldg\in\G(\t)$, Lemma~\ref{lem:inverse-g} allows to define a mapping $\bbX^\t\to\bbX^{\boldg\cdot\t}$ which to $\x=(x_\u)_{\u\in\t}$ associates its \textit{push-forward} $\boldg_*\x$, defined as follows:
\begin{align}\label{eq:def-push-forward}
\boldg_*\x=(x_{\boldg^{-1}\cdot\u'})_{\u'\in \boldg\cdot\t}.
\end{align}
This definition of the push-forward also makes sense when $\bbX=\G$, and this actually allows to give a concise description of $\boldg^{-1}$ for $\boldg\in\G(\t)$, as in Lemma~\ref{lem:other-expression-inverse-g} below.
Given $\boldg\in\G(\t)$, we let 
\begin{align*}
\overline\boldg=(g_\u^{-1})_{\u\in\t},
\end{align*}
which belongs%
	\footnote{Note that $\overline\boldg$ is not necessarily an element of $\G(\t)$.
	Indeed, for $\u\in\t$ we have $g_\u^{-1}\in\G_{g_\u(\pos_\u(\t))}$, which is different from $\G_{\pos_\u(\t)}$ unless $g_{\u}(\pos_\u(\t))=\pos_\u(\t)$.}
to the set $\G^\t$, for every $\boldg\in\G(\t)$.
We collect here a series of easy but useful lemmas on this push-forward operation.
The following lemma is a mere reformulation of the identity \eqref{eq:inverting-g}, while the next one follows from Lemma~\ref{lem:inverse-g}.

\begin{lemma}\label{lem:other-expression-inverse-g}
	For $\boldg\in\G(\t)$, we have $\boldg^{-1}=\boldg\pf\overline\boldg$.
\end{lemma}

\begin{lemma}\label{lem:inverse-pushforward-g}
	Given a set $\bbX$ and $\boldg\in\G(\t)$, we have $\boldg^{-1}\pf(\boldg\pf\x)=\x$ for every $\x\in\bbX^\t$ and $\boldg\pf(\boldg^{-1}\pf\x')=\x'$ for every $\x'\in\bbX^{\boldg\cdot \t}$.
\end{lemma}

\noindent
Next, we have a convenient description of the action of an element of $\G(\t)$ on the sets of children-positions.

\begin{lemma}\label{lem:children-after-shuffling}
	For every $\boldg\in\G(\t)$ and $\u\in\t$, we have $\pos_{\boldg\cdot\u}(\boldg\cdot\t)=g_\u(\pos_{\u}(\t))$ and in particular we have the identity $\#\,\pos_{\boldg\cdot\u}(\boldg\cdot\t)=\#\,\pos_{\u}(\t)$.
\end{lemma}

\begin{proof}
	It is easily verified from the definitions that given $\u,\v\in \t$ and $i\in\{1,2,\dots\}$, we have $\v=\u i$ if and only if $\boldg\cdot\v=(\boldg\cdot\u)g_\u(i)$.
	Hence the first identity.
	The second follows since $g_\u$ is injective for every $\u\in\t$.
\end{proof}
\noindent
Given two sets $\bbX$ and $\bbY$, and some mapping $f\colon\bbX\to\bbY$, we denote by $f(\x)$ the pointwise evaluation of $f$ on $\x=(x_\u)_{\u\in\t}\in\bbX^\t$, that is
\begin{align*}
f(\x)=(f(x_\u))_{\u\in\t},
\end{align*}
which is an element of $\bbY^\t$.

\begin{lemma}\label{lem:commutation-pushforward-g-pointwise-eval}
	Given two sets $\bbX$ and $\bbY$, and a mapping $f\colon\bbX\to\bbY$, we have the identity $\boldg_*\bigl( f(\x)\bigr)= f(\boldg_*\x)$ for every $\x\in\bbX^\t$ and every $\boldg\in\G(\t)$.
\end{lemma}
\noindent
This follows from the observation that both sides of the identity consist of the collection $(f(x_{\boldg^{-1}\cdot\u'}))_{\u'\in\boldg\cdot\pt}$, where $\x=(x_\u)_{\u\in\pt}$.
In particular, if we apply this with $f\colon\G\to\G$, $g\mapsto g^{-1}$, we obtain the following.
\begin{lemma}\label{lem:commutation-bar-pushforward-g}
	For every $\boldg,\boldh\in\G(\t)$, we have $\overline{(\boldg\pf\boldh)}=\boldg\pf\overline\boldh$.
\end{lemma}
\noindent
This concludes our series of lemmas on the push-forward operation.

\subsection{A bijection with some decorated plane trees}

Let $\pt\in\trees$ and let $\bbX$ be a set.
We call $\pt$-tuple of elements of $\bbX$ an element in the Cartesian product 
\begin{align*}
\bbX^\pt=\{(x_\v)_{\v\in \pt}\colon \forall\v\in \pt,\,x_\v\in\bbX\}.
\end{align*}

\begin{defin}
	An $\bbX$-\emph{decorated plane tree} is a pair $(\pt,\x)$ with $\pt\in\trees$ and $\x=(x_\v)_{\v\in \pt}\in\bbX^\pt$.
	The set of $\bbX$-decorated plane trees $(\pt,\x)$ such that $\pt\in\trees_n$, $n\geq1$, is denoted by $\trees_n[\bbX]$, and we set $\trees{[\bbX]}=\bigcup_{n\geq1}\trees_n[\bbX]$.
\end{defin}

In practice, we often need the decorations to satisfy some compatibility conditions with respect to the degree of the vertex they are assigned to, in the following sense.
We call \textit{graded set} a set $\bbX$ with a decomposition as a disjoint union $\bbX=\bigsqcup_{k\geq0}\bbX_k$, called the \textit{grading} of $\bbX$,  which will often be implicit.
Then a $\pt$-tuple of elements of $\bbX$ is said to be \textit{grading-compatible} if it belongs to
\begin{align*}
\bbX_\pt=\{(x_\v)_{\v\in \pt}\colon \forall\v\in \pt,\,x_\v\in\bbX_{k_\v(\pt)}\}\subset \bbX^\pt.
\end{align*}

\begin{defin}
	Let  $\bbX=\bigsqcup_k\bbX_k$ be a graded set.
	An $\bbX$-{decorated plane tree}  $(\pt,\x)$ is said to be \textit{grading-compatible} if $\x\in\bbX_\pt$.
	The set of grading-compatible $\bbX$-decorated plane trees is denoted by $\treesgr_n{[\bbX]}$, and we set $\treesgr[{\bbX}]=\bigcup_{n\geq1}\treesgr_n[\bbX]$.
\end{defin}

We will be particularly interested in plane trees decorated by finite subsets of $\{1,2,\dots\}$.
Hence, let us consider the graded set of finite subsets of $\{1,2,\dots\}$,
\begin{align*}
\Parts=\bigsqcup_{k\geq0}\Parts_k,
\end{align*}
where $\Parts_k$ is the set of $k$-element subsets of $\{1,2,\dots\}$, for every $k\geq0$.
For $S$ a finite subset of $\{1,2,\dots\}$, we let $p_S\in\G_S$ be the unique (injective) non-decreasing mapping $p_S\colon S\to\{1,2,\dots\}$ such that $p_S(S)=\{1,2,\dots,|S|\}$.
Morally, $p_S$ ``pushes'' the elements of $S$ as far to the left as possible, while the inverse mapping $p_S^{-1}\colon \{1,2,\dots,|S|\}\to\{1,2,\dots\}$ corresponds to listing the elements of $S$ from left to right.
Then for $\pt\in\trees$ and $\boldS=(S_\u)_{\u\in\pt}\in\Parts_\pt$, we set:
\begin{align}\label{eq:def-p_S}
p_\boldS&=\left(p_{S_\u}\right)_{\u\in\pt},
\end{align}
that is $p_\boldS$ is obtained from $\boldS$ by pointwise evaluation of the mapping $\Parts\to\G$ given by $S\mapsto p_S$.
For $\t\in\subtrees$, we will use the shorthand notation:
\begin{align}\label{eq:def-boldp}
\boldp_\t=p_{\Pos(\t)}.
\end{align}

\mypar{Enriching the mapping ``push''}
With the above notation, the mapping $\push$ which we alluded to in the introduction is simply $\push\colon\t\mapsto	\boldp_\t\cdot\t$.
It is indeed easily verified that $\push(\t)$ thus defined is a plane tree for every $\t\in\subtrees$.
In order to retrieve $\t$ from $\push(\t)$, we need to keep the information given by the collection $\Pos(\t)$, using the push-forward operation defined in~\eqref{eq:def-push-forward}.
Hence we consider an ``enriched'' version of the mapping $\push$, as follows:
\begin{align*}
\epush\colon
\t				\longmapsto			\Bigl(\boldp_\t\cdot\t,\,(\boldp_\t)\pf\Pos(\t)\Bigr).
\end{align*}

\begin{lemma}
	This defines a mapping $\epush\colon\subtrees\to\treesgr{[\Parts]}$.
\end{lemma}

\begin{proof}
	For $\t\in\subtrees$, if we write $\epush(\t)=(\pt,\boldS)$ with $\boldS=(S_\u)_{\u\in\pt}$, then we have $\pt=\push(\t)\in\trees$.
	We need to justify that $\boldS$ is grading-compatible.
	For $\u\in\pt$, we have $S_\u=\pos_{\boldp_\t^{-1}\cdot\u}(\t)$ by definition.
	Now by Lemma~\ref{lem:children-after-shuffling}, we have the identity $\#\,\pos_{\boldp_\t^{-1}\cdot\u}(\t)=\#\,\pos_{\u}(\boldp_\t\cdot\t)=k_{\u}(\pt)$.
	Wrapping up, for every $\u\in\pt$ we have $\#\, S_\u=k_\u(\pt)$, so that $S_\u\in\Parts_{k_\u(\pt)}$ and $\boldS$ is indeed grading-compatible.
\end{proof} 

There is a natural candidate to be an inverse of the mapping $\epush$, namely the mapping $(\pt,\boldS)\mapsto\overline{p_{\boldS}}\cdot\pt$.
This is verified by the following proposition.

\begin{prop}\label{prop:bijection-push}
	The mappings 
	\begin{align*}
	&\epush\colon
	\begin{pmatrix}
	\subtrees		&\longrightarrow			&\treesgr{[\Parts]}\\
	\t				&\longmapsto				&\bigl(\boldp_\t\cdot\t,\, (\boldp_\t)\pf\Pos(\t)\bigr)
	\end{pmatrix},\qquad
	&\epush^{-1}\colon
	\begin{pmatrix}
	\treesgr{[\Parts]}		&\longrightarrow			&\subtrees		\\
	(\pt,\boldS)			&\longmapsto				&\overline{p_{\boldS}}\cdot\pt
	\end{pmatrix},
	\end{align*}
	 are inverse of each other.
\end{prop}

\begin{proof}
	Let us temporarily write $\mathfrak Q\colon(\pt,\boldS)\mapsto \overline{p_{\boldS}}\cdot\pt$, and prove that it is indeed an inverse of $\epush$.
	For $\t\in\subtrees$, if we write $\epush(\t)=(\pt,\boldS)$, then Lemma~\ref{lem:commutation-pushforward-g-pointwise-eval} applied to the mapping $\Parts\to\G$, $S\mapsto p_S$ gives that:
	\begin{align*}
	p_{\boldS}
		=p_{(\boldp_\t)_*\Pos(\t)}
		=(\boldp_\t)\pf (p_{\Pos(\t)})
		\overset{\eqref{eq:def-boldp}}{=}(\boldp_\t)\pf \boldp_\t
	\end{align*}
	In particular, using Lemmas~\ref{lem:commutation-bar-pushforward-g} and~\ref{lem:inverse-g}, we have that $\overline{p_{\boldS}} = (\boldp_\t)\pf\overline{\boldp_\t}=\boldp_\t^{-1}$.
	Hence $\overline{p_{\boldS}}\cdot\pt=\boldp_\t^{-1}\cdot\pt=\t$, since $\pt=\overline{p_{\boldS}}\cdot\pt$.
	Hence we have proven:
	\begin{align}\label{eq:proof-epush-bij-1}
	\mathfrak Q\bigl(\epush(\t)\bigr)=\t,\qquad \t\in\subtrees.
	\end{align}
	
	On the other hand, for $(\pt,\boldS)\in\treesgr{[\Parts]}$ with $\boldS=(S_\u)_{\u\in\pt}$, if we set $\t=\mathfrak Q(\pt,\boldS)$, then for every $\u\in\pt$, we have by Lemma~\ref{lem:children-after-shuffling} applied with $\boldg=\overline{p_{\boldS}}=(p_{S_\u}^{-1})_{\u\in\pt}$,
	\begin{align*}
	\pos_{\overline{p_{\boldS}}\cdot\u}(\t)
		=\pos_{\overline{p_{\boldS}}\cdot\u}(\overline{p_{\boldS}}\cdot\pt)
		=p_{S_\u}^{-1}\bigl(\pos_\u(\pt)\bigr).
	\end{align*}
	Now since $\pt$ is a plane tree, we necessarily have that $\pos_\u(\pt)=\{1,2,\dots,k_\u(\pt)\}$.
	By definition, $p_{S_\u}^{-1}$ is such that $p_{S_\u}^{-1}(\{1,2,\dots,|S_\u|\})=S_\u$.
	Since $\boldS$ is grading-compatible we have $|S_\u|=k_\u(\pt)$.
	The conclusion of these observations is that $p_{S_\u}^{-1}\bigl(\pos_\u(\pt)\bigr)=S_\u$.
	Plugging this into the last display we have
	\begin{align*}
	\pos_{\overline{p_{\boldS}}\cdot\u}(\t)
		=S_\u,\qquad\u\in\pt.
	\end{align*}
	This expresses that $\boldS=(\overline{p_{\boldS}})^{-1}_*\Pos(\t)$, or equivalently $\Pos(\t)=(\overline{p_{\boldS}})_*\boldS$.
	Now, as above, we can apply Lemma~\ref{lem:commutation-pushforward-g-pointwise-eval} with the mapping $\Parts\to\G$, $S\mapsto p_S$ to obtain the identity $p_{\Pos(\t)}=(\overline{p_{\boldS}})^{-1}_*\,p_{\boldS}$, that is $\boldp_\t=(\overline{p_{\boldS}})^{-1}$ using the definition \eqref{eq:def-boldp} and Lemma~\ref{lem:inverse-g}.
	Hence, $\boldp_\t\cdot\t=(\overline{p_{\boldS}})^{-1}\cdot\t=\pt$, since $\t=\overline{p_{\boldS}}\cdot\pt$.
	Also, we have $(\boldp_\t)_*\Pos(\t)=(\overline{p_{\boldS}})^{-1}_*\Pos(\t)=\boldS$, since we justified that $\Pos(\t)=(\overline{p_{\boldS}})_*\boldS$.
	Wrapping up, we have $\boldp_\t\cdot\t=\pt$ and $(\boldp_\t)_*\Pos(\t)=\boldS$, so that we have proven:
	\begin{align}\label{eq:proof-epush-bij-2}
	\epush\bigl(\mathfrak Q(\pt,\boldS)\bigr)=(\pt,\boldS),\qquad (\pt,\boldS)\in\treesgr{[\Parts]}.
	\end{align}
	The combination of \eqref{eq:proof-epush-bij-1} and \eqref{eq:proof-epush-bij-2} gives the desired result.
\end{proof}

\begin{rem}\label{rem:bij-push-embed}
	The mappings $\epush\colon\subtrees\to\treesgr{[\Parts]}$ and $\epush^{-1}\colon\treesgr{[\Parts]}\to\subtrees$ actually induce mappings $\epush\colon\subtrees_n\to\treesgr_n{[\Parts]}$ and $\epush^{-1}\colon\treesgr_n{[\Parts]}\to\subtrees_n$ respectively, for every $n\geq1$.
	Indeed, if $\t=p_\boldS\cdot\pt$, then $\t$ and $\pt$ must have the same number of vertices since $\u\mapsto p_{\boldS}\cdot\u$ maps bijectively $\pt$ onto $\t$.
\end{rem}

\subsection{Probabilistic consequences}

\begin{defin}
	We call a non-negative sequence $\wtheta=(\theta_1,\theta_2,\dots)$ such that $0<\sum_i \theta_i<\infty$ a \textit{summable weight sequence}, and the number of elements in the support of a summable weight sequence $\wtheta$ is denoted by $N_\wtheta\in\N\cup\{\infty\}$.
\end{defin}

When $\wtheta$ is a summable weight sequence, in order to treat in a unified way the cases $N_\wtheta<\infty$ and $N_\wtheta=\infty$, we make the convention that `` $i\leq N_\wtheta$ '' means ``$i$ is \textit{finite} and at most equal to $N_\wtheta$''.

\mypar{A model of random subsets}
Let $\wtheta$ be a summable weight sequence.
We denote by $\e(\wtheta)=(e_0(\wtheta),e_1(\wtheta),e_2(\wtheta),\dots)$ the elementary symmetric functions evaluated at the coefficients of $\wtheta$, that is:
\begin{align*}
e_0(\wtheta)=1&&\text{and}&&e_k(\wtheta)&=\sum_{1\leq i_1<i_2<\dots<i_k} \theta_{i_1}\theta_{i_2}\cdots \theta_{i_k},
\quad k\geq1.
\end{align*}
Notice that for $k\geq1$ we have $e_k(\wtheta)\leq (\sum_i \theta_i)^k<\infty$, and also that $e_k(\wtheta)>0$ when $0\leq k\leq {N_\wtheta}$, while $e_k(\wtheta)=0$ for every $k>{N_\wtheta}$.
In particular, the sequence $\e(\wtheta)$ has no internal zeros.
There is a natural model of random subsets of $\{1,2,\dots\}$ associated to $\wtheta$, which we define as follows.
For\footnote{Recall our convention that when ${N_\wtheta}=\infty$, we only consider finite values of $k$.} $0\leq k\leq {N_\wtheta}$, the probability distribution $\randSubset{\wtheta}{k}$ on $\Parts_{k}$ is defined by
\begin{align}\label{eq:def-random-subsets}
\randSubset{\wtheta}{k}(S)=\frac{\prod_{i\in S}\theta_i}{e_k(\wtheta)},\qquad S\in\Parts_{k}.
\end{align}
In particular, $\randSubset{\wtheta}{k}$ is supported on finite subsets of the support of $\wtheta$.

\begin{rem}\label{rem:random-subset-product-measure}
	The use of the letter B in the notation $\randSubset{\wtheta}{k}$ is motivated by an alternative description of the corresponding probability distribution in terms of a  Bernouilli vector, as follows.
	Consider independent random variables $(B_i)_{i\geq 1}$ with respective distributions $(\mathrm{Ber}(\frac{\theta_i}{1+\theta_i}))_{i\geq1}$.
	The function $i\mapsto B_i$ is the indicator function of a random finite subset $S(B_1,B_2,\dots)\subset\{1,2,\dots\}$, and for $S\in\Parts$ we have $S(B_1,B_2,\dots)=S$ with probability $C\cdot\prod_{i\in S}\theta_i$, where $C$ is the constant $C=1/\prod_i (1+\theta_i)$.
	Hence the conditional distribution of $S(B_1,B_2,\dots)$ given that it has $k$ elements is $\randSubset{\wtheta}{k}(\diff S)$, for every $0\leq k\leq N_\wtheta$.
\end{rem}

\begin{prop}\label{prop:comparison-SG-and-ST}
	Let $n\geq1$ and let $\wtheta$ be a summable sequence.
	Fix some $\t\in\subtrees$ and $(\pt,\boldS)\in\treesgr_n[\Parts]$ such that $\epush(\t)=(\pt,\boldS)$.
	If we write $\boldS=(S_\u)_{\u\in\pt}$, then for every $n\geq1$ the following identity holds:
	\begin{align}\label{eq:comparison-SG-and-ST}
	\Sub{\wtheta}{n}(\t)=\SimpGen{\e(\wtheta)}{n}(\pt)\cdot\prod_{\u\in \pt}\randSubset{\wtheta}{k_\u(\pt)}(S_\u).
	\end{align}
\end{prop}

\begin{proof}
	By construction, for all $i\in\{1,2,\dots\}$ we have
	\begin{align*}
	\#V_i(\t)=\#\{\u\in\t\colon i\in\pos_\t(\u)\}
	= \#\{\u\in \pt\colon i\in S_\u\}=\sum_{\u\in \pt}\indic{i\in S_\u}.
	\end{align*}
	Hence we obtain
	\begin{align}\label{eq:proof-of-relation-SG-vs-subtrees}
	\prod_{i\geq1}\theta_i^{\#V_i(\t)}
	=\prod_{i\geq1}\prod_{\u\in \pt}\theta_i^{\indic{i\in S_\u}}
	=\prod_{\u\in \pt}\prod_{i\in S_\u}\theta_i,
	\end{align}
	Now by Prop~\ref{prop:bijection-push} and Remark~\ref{rem:bij-push-embed}, summing the left-hand side over $\t\in\subtrees_n$ and summing the right-hand side over $(\pt,\boldS)\in\treesgr_n[\Parts]$ gives the same result.
	Hence,
	\begin{align*}
	\sum_{\t\in\subtrees_n}\prod_{i\geq1}\theta_i^{\#V_i(\t)}
	=\sum_{(\pt,(S_\u)_{\u\in \pt})\in\treesgr_n[\Parts]}\prod_{\u\in \pt}\prod_{i\in S_\u}\theta_i.
	\end{align*}
	Using \eqref{eq:def-subtree-model} and the definition of $\treesgr_n[\Parts]$, we can re-write the last display as
	\begin{align*}
	\PFsubtrees{\wtheta}{n}
	=\sum_{\pt\in\trees_n}
	\sum_{\substack{(S_\u)_{\u\in \pt}\\ \forall \u,\, S_\u\in\Parts_{k_\u(\pt)}}}
	\prod_{\u\in \pt}
	\prod_{i\in S_\u}\theta_i
	=\sum_{\pt\in\trees_n}\prod_{\u\in \pt}\biggl(\sum_{S_\u\in\Parts_{k_\u(\pt)}}\prod_{i\in S_\u}\theta_i\biggr)
	=\sum_{\pt\in\trees_n}\prod_{\u\in \pt}e_{k_\u(\pt)}.
	\end{align*}
	But by the definition~\eqref{eq:def-SG}, the right-hand side is none other than $\PFtrees{\e(\wtheta)}{n}$.
	Hence we have proved that $\PFsubtrees{\wtheta}{n}=\PFtrees{\e(\wtheta)}{n}$.
	
	We now return to \eqref{eq:proof-of-relation-SG-vs-subtrees}, for some fixed $\t\in\subtrees_n$ and $(\pt,\boldS)\in\treesgr_n[\Parts]$.
	Suppose first that $k_\u(\pt)\leq {N_\wtheta}$ for all $\u\in \pt$.
	In this case, we can factor in the right-hand side of \eqref{eq:proof-of-relation-SG-vs-subtrees} by $e_{k_\u(\pt)}\neq0$ to get
	\begin{align*}
	\prod_{i\geq1}\theta_i^{\#V_i(\t)}
	=\prod_{\u\in \pt}\prod_{i\in S_\u}\theta_i
	= \prod_{\u\in \pt}
	e_{k_\u(\pt)}\cdot \randSubset{\wtheta}{k_\u(\pt)}(S_\u),
	\end{align*}
	where we used the definition \eqref{eq:def-random-subsets}.
	Since we proved that $\PFsubtrees{\wtheta}{n}=\PFtrees{\e(\wtheta)}{n}$, we can divide the left-hand side by $\PFsubtrees{\wtheta}{n}$ and the right-hand side by $\PFtrees{\e(\wtheta)}{n}$ to obtain the claimed identity \eqref{eq:comparison-SG-and-ST}, using the definitions \eqref{eq:def-subtree-model} and \eqref{eq:def-SG}.
	
	We supposed above that $k_\u(\pt)\leq {N_\wtheta}$ for all $\u\in \pt$.
	Let us justify that when this condition is not met, the result is trivial.
	Hence suppose that $k_{\v}(\pt)> {N_\wtheta}$ for some vertex $\v\in\pt$.
	Then we get that $\prod_{i\in S_{\v}}\theta_i=0$ since the inequality $k_{\v}(\pt)> {N_\wtheta}$ says that $\#\,S_{\v}>\#\,\supp(\wtheta)$.
	Therefore $\Sub{\wtheta}{n}(\t)=(\PFsubtrees{\wtheta}{n})^{-1}\prod_{\u\in \pt}\prod_{i\in S_\u}\theta_i=0$, and the left-hand side of \eqref{eq:comparison-SG-and-ST} is zero.
	On the other hand, remember that $e_k(\wtheta)=0$ whenever $k>N_\wtheta$, so that in particular $e_{k_\v(\pt)}(\wtheta)=0$.
	Hence $\SimpGen{\e(\wtheta)}{n}(\pt)=\smash{(\PFtrees{\e(\wtheta)}{n})^{-1}}\prod_{\u\in\pt}e_{k_\u(\pt)}(\wtheta)=0$, and the right-hand side of \eqref{eq:comparison-SG-and-ST} is also zero.
\end{proof}

\subsection{Growing the plane trees}

The fact that our subtree model admits a description in terms of simply generated trees with weight sequence $\e(\wtheta)$ is particularly nice for our purposes, thanks to the well-known Newton's inequalities.

\begin{lemma}[Newton's inequalities]\label{lem:Newton-inequalities}
	Let $(\alpha_1,\dots,\alpha_m)$ be a non-negative and non-zero sequence with finite length $m\geq1$, and let $(e_k)_{0\leq k\leq m}$ be the sequence of the elementary symmetric polynomials in the $(\alpha_1,\dots,\alpha_m)$, or equivalently let $e_k=e_k(\alpha_1,\dots,\alpha_m,0,0,\dots)$ for $0\leq k\leq m$.
	Then the sequence $(e_k)_{0\leq k\leq m}$ satisfies the \emph{ultra-log-concavity} inequality, that is
	\begin{align}\label{eq:ULC}
	\frac{e_k^2}{\binom{m}{k}^2}\geq 
		\frac{e_{k-1}}{\binom{m}{k-1}}\cdot
		\frac{e_{k+1}}{\binom{m}{k+1}},
		\qquad k\in\{1,\dots,m-1\}.
	\end{align}
	In particular, $\e(\alpha_1,\dots,\alpha_m,0,0,\dots)$ is log-concave.%
		\footnote{
		This is deduced from the inequality $\binom{m}{k}^2\geq\binom{m}{k-1}\binom{m}{k+1}$ for $0<k<m$, which holds by log-concavity of the binomial coefficients.
		Note that in this paper we require log-concave sequences to have no internal zeros, but we already justified that this is the case for sequences having the form $\e(\wtheta)$ with $\wtheta$ a summable weight sequence.
		}
\end{lemma}

Proofs of Lemma~\ref{lem:Newton-inequalities} may be found in many references.
See for instance Theorem~51 in \cite{HardyLittlewoodPolya52}, or Lemma~7.1.1 in the more recent reference \cite{Branden15}.

\begin{cor}\label{cor:LC-elem-sym-func}
	The sequence $\e(\wtheta)$ is log-concave for every a summable weight sequence $\wtheta$.
\end{cor}

\begin{proof}
	 For $\wtheta$ a summable weight sequence, the sequence $\e(\wtheta)$ is log-concave since it is the pointwise limit $\ell\rightarrow\infty$ of the sequences $\e(\theta_1,\dots,\theta_\ell,0,0,\dots)$, $\ell\geq1$, which are log-concave by Lemma~\ref{lem:Newton-inequalities}.
\end{proof}

\begin{cor}\label{cor:growing-underlying-plane-trees}
	The distributions $(\SimpGen{\e(\wtheta)}{n})_{n\geq1}$ can be coupled as a Markov process $(\T_n)_{n\geq1}$ such that $\T_1\subset\T_2\subset\T_3\subset\dots$.
\end{cor}

\begin{proof}
	Apply Theorem~\ref{thm:main-thm-SG} together with Corollary~\ref{cor:LC-elem-sym-func}.
\end{proof}

\subsection{Growing the decorations}

In order to construct increasing couplings for the distributions $(\Sub{\wtheta}{n})_{n\geq1}$ using Proposition~\ref{prop:comparison-SG-and-ST} and Corollary~\ref{cor:growing-underlying-plane-trees}, we will need a way to couple in an increasing way the distributions $(\randSubset{\wtheta}{k})_{0\leq k\leq {N_\wtheta}}$.
This is accomplished as follows.

\begin{prop}\label{prop:growing-subsets}
	Let $\wtheta$ be a summable weight sequence.
	There exists a random sequence  $X=(X_1,X_2,\dots,X_{N_\wtheta})$ of distinct elements of $\{1,2,\dots\}$ such that for every $0\leq k\leq N_{\wtheta}$, the random set $S_k=\{X_1,X_2,\dots,X_k\}$ has distribution $\randSubset{\wtheta}{k}$.
\end{prop}

By Remark~\ref{rem:random-subset-product-measure}, this is equivalent to the statement that if $\mathcal B=(B_i,i\geq1)$ is a sequence of independent Bernoulli random variables whose sum is finite almost surely, then the distributions of $\mathcal B$ conditionally on $\sum_i B_i=k$, $0\leq k\leq N$, admit an increasing coupling, \textit{i.e.} one which is coordinate-wise non-decreasing.
Here $N$ is the number of $i\in\{1,2,\dots\}$ such that $B_i$ is not almost surely zero.

In the case $N<\infty$, this result is already known, with several proofs.
Jonasson and Nerman leverage in \cite[Proposition~6.2]{JonassonNerman96} a Markov chain whose stationary distributions are the ones we need to couple.
Broman, van~de~Brug and Kager exhibit in \cite[Lemma~2.2]{BromanVDBrugKager12} an explicit Markov kernel which builds the coupling.
More generally, Pemantle proves in \cite[Lemma~3.3]{Pemantle00} a similar statement for a class of negatively dependent measures which contains products of Bernoulli measures.
Lastly, Borcea, Br\"and\'en and Liggett generalize the latter, in \cite[Thm.~4.19]{BorceaBrandenLiggett09}, to a large class of negatively dependent measures.

Still, we feel it is useful to give a self-contained proof of Proposition~\ref{prop:growing-subsets}.
Our proof has close similarities with that of \cite[Lemma~3.3]{Pemantle00}.
For $i\in\{1,2,\dots\}$, we consider the event: 
\begin{align*}
\Ecal_i=\{S\in\Parts\colon i\in S\}.
\end{align*}
We will need the following lemma, which resembles Lemma~3.2.(i) in the latter reference.

\begin{lemma}\label{lem:preparation-incr-coupling-subsets}
	Let $\wtheta$ be a summable weight sequence and let $0\leq k<N_\wtheta$.
	Then,
	\begin{align}\label{eq:inrease-prob-contains-i}
	\anyProb{\randSubset{\wtheta}{k}}{\Ecal_i}
		\leq \anyProb{\randSubset{\wtheta}{k+1}}{\Ecal_i},
		\qquad i\in\{1,2,\dots\}.
	\end{align}
\end{lemma}

\begin{proof}
	Let $i\in\{1,2,\dots\}$, which we assume to be in $\supp(\wtheta)$ to avoid trivialities.
	The $k=0$ case is trivial since $\anyProb{\randSubset{\wtheta}{0}}{\Ecal_i}=0$, so let $0< k<N_\wtheta$.
	Observe that if we write ${\wthetai}=(\theta_1,\dots \theta_{i-1},0,\theta_{i+1},\dots)$, then the inequality \eqref{eq:inrease-prob-contains-i} we need to prove may be written as follows:
	\begin{align*}
	\frac{
		\theta_i\cdot e_{k-1}({\wthetai})
	}{
		e_{k}(\wtheta)
	}
	\leq
	\frac{
		\theta_i\cdot e_{k}({\wthetai})
	}{
		e_{k+1}(\wtheta)
	}.
	\end{align*}
	Note that $e_\ell(\wtheta)=e_\ell({\wthetai})+\theta_i\cdot e_{\ell-1}({\wthetai})$ for $\ell\geq 1$.
	Hence the last display reads
	\begin{align*}
	\frac{
		\theta_i\cdot e_{k-1}({\wthetai})
	}{
		e_k({\wthetai})+\theta_i\cdot e_{k-1}({\wthetai})
	}
	\leq
	\frac{
		\theta_i\cdot e_{k}({\wthetai})
	}{
		e_{k+1}({\wthetai})+\theta_i\cdot e_{k}({\wthetai})
	},
	\end{align*}
	or $ e_k({\wthetai})^2\geq e_{k+1}({\wthetai})e_{k-1}({\wthetai})$ after re-writing.
	This inequality follows from Corollary~\ref{cor:LC-elem-sym-func}.
\end{proof}

\begin{proof}[Proof of Proposition~\ref{prop:growing-subsets}]
	It is sufficient to treat the case when $\wtheta$ is finitely supported since we can  deduce the general case by a limiting argument.%
		\footnote{
			Say, using a suitable adaptation of Lemma~\ref{lem:closure-admissibility} in our context.
		}
	We do so by induction on $1\leq N_\wtheta<\infty$.	
	The base case $N_\wtheta=1$ is trivial, so assume that $N_\wtheta>1$.
	Let $i\in\supp(\wtheta)$ and ${\wthetai}=(\theta_1,\dots \theta_{i-1},0,\theta_{i+1},\dots)$.
	Since $N_{\wthetai}=N_{\wtheta}-1$, we get by induction an increasing coupling $(S^+_k)_{0\leq k\leq N_\wtheta-1}$ of the distributions $(\randSubset{{\wthetai}}{k})_{0\leq k\leq {N_\wtheta-1}}$.
	We set $p_k= \anyProb{\randSubset{\wtheta}{k}}{\Ecal_i}$, $0\leq k \leq N_\wtheta$.
	Lastly, we let $U$ be uniformly random in $[0,1]$, independently from all other variables.
	We set%
		\footnote{
			When $k=N_\wtheta$, we have $p_k=1$ by definition.
			Hence we de not use $S^+_{N_\wtheta}$, which is undefined.
		}
	\begin{align}\label{eq:def-incr-coupling-subsets}
	S_0=\emptyset,\qquad\text{and}\qquad
	\forall 1\leq k\leq N_\wtheta,\quad
	S_k=
		\begin{cases}
			S^+_{k-1}\cup\{i\}		&\text{if $U\leq p_k$,}\\
			S^+_{k}				&\text{if $U> p_k$.}
		\end{cases}
	\end{align}
	By construction $S_0=\emptyset$ as needed, and if we let $1\leq k\leq N_\wtheta$ and $S\in\Parts_k$, then if $i\in S$, we have
	\begin{align*}
	\Prob{S_k=S}
		=\Prob{S^+_{k-1}=S\setminus\{i\},\, U\leq p_k}
		= p_k\cdot
		\anyProb{\randSubset{{\wthetai}}{k-1}}{S\setminus\{i\}},
	\end{align*}
	while if $i\notin S$, we have
	\begin{align*}
	\Prob{S_k=S}
		=\Prob{S^+_{k}=S,\, U> p_k}
		= (1-p_k)\cdot\anyProb{\randSubset{{\wthetai}}{k}}{S}.
	\end{align*}
	Let $\theta_S=\prod_{j\in S}\theta_j$.
	The two last displays give that
	\begin{align*}
	\begin{matrix}
	\Prob{S_k=S}
	&=& \indic{i\in S} 
	\cdot p_k\cdot	\anyProb{\randSubset{{\wthetai}}{k-1}}{S\setminus\{i\}}
	&+&
	\indic{i\notin S}\cdot
	(1-p_k)\cdot\anyProb{\randSubset{{\wthetai}}{k}}{S}\\
	&=& \indic{i\in S}\cdot\frac{\theta_i \cdot e_{k-1}({\wthetai})}{e_k(\wtheta)}
	\cdot
	\frac{ \theta_S/\theta_i}{e_{k-1}({\wthetai})}
	&+&
	\indic{i\notin S}\cdot \frac{e_{k}({\wthetai})}{e_k(\wtheta)}
	\cdot
	\frac{\theta_S}{e_{k}({\wthetai})}.
	\end{matrix}
	\end{align*}
	This simplifies to $\Prob{S_k=S}={\theta_S}/{e_{k}({\wthetai})}=\anyProb{\randSubset{\wtheta}{k}}{S}$, as needed.
	
	Now let $1\leq k< N_\wtheta$.
	By Lemma~\ref{lem:preparation-incr-coupling-subsets}, we have $p_k\leq p_{k+1}$, so that in order to verify that $S_k\subset S_{k+1}$, we only need to verify it on the events $\{U\leq p_k\}$, $\{p_k< U\leq p_{k+1}\}$, and $\{U>p_{k+1}\}$.
	Since $(S^+_\ell)_{0\leq\ell\leq N_{\wtheta}}$ is increasing, we have $S_k\subset S_{k+1}$ on the events $\{U\leq p_k\}$ and $\{U>p_{k+1}\}$.
	On the remaining event $\{p_k< U\leq p_{k+1}\}$, we have $S_k=S^+_k\subset S^+_k\cup\{i\}=S_{k+1}$.
	Hence we have verified that $S_0\subset S_{1}\subset\dots\subset S_{N_\wtheta}$.
	Since $\#\,S_k=k$ for $0\leq k\leq N_\wtheta$, we can set $X_k$ to be the unique element in $S_k\setminus S_{k-1}$ for $1\leq k\leq N_\wtheta$, and the proof is complete.
\end{proof}

\subsection{Matching the growths by shuffling the decorated plane trees}
\label{subsec:coupling-using-shuffling}

Fix $\wtheta=(\theta_1,\theta_2,\dots)$ a non-negative sequence such that $0<\sum_i \theta_i<\infty$.
Let us first describe a naive attempt to combine the increasing couplings for $(\SimpGen{\e(\wtheta)}{n})_{n\geq1}$ and for $(\randSubset{\wtheta}{k})_{0\leq k \leq N_\wtheta}$.

\mypar{Naive coupling}
Consider mutually independent copies
\begin{align*}
X_\u=(X_{1,\u},\dots,X_{N_\wtheta,\u}),\qquad\u\in\U,
\end{align*}
of the random sequence $X=(X_{1},\dots,X_{N_\wtheta})$ given by Proposition~\ref{prop:growing-subsets}.
Independently, let $(\T_n)_{n\geq1}$ be the coupling given by Corollary~\ref{cor:growing-underlying-plane-trees}.
For $n\geq1$, we let $\boldS_n=(S_{k_\u(\T_n),\u})_{\u\in\T_n}$, where $S_{k,\u}=\{X_{1,\u},\dots,X_{k,\u}\}$ for every $\u\in\T_n$ and every $0\leq k\leq N_\wtheta$.
This yields an element $\boldS_n\in\Parts_{\T_n}$.

\begin{prop}[Naive coupling]\label{prop:naive-coupling}
	In the above setting, if one sets
	\begin{align*}
	\subT_n=\embed\left(\T_n,\boldS_n\right),
	\qquad n\geq1,
	\end{align*}
	then the sequence $(\subT_n)_{n\geq1}$ is a coupling of the distributions $(\Sub{\wtheta}{n})_{n\geq 1}$.
\end{prop}
\begin{proof}
	This is a direct consequence of Proposition~\ref{prop:comparison-SG-and-ST}.
\end{proof}

At first sight, the coupling in Proposition~\ref{prop:naive-coupling} may look increasing, since the ones in Corollary~\ref{cor:growing-underlying-plane-trees} and Proposition~\ref{prop:growing-subsets} are, but in fact it is not.
The problem is that $(\T_n)_{n\geq1}$ ``grows from the right'' while for $\u\in\U$ the sets $(S_{k,\u})_{0\leq k\leq {{N_\wtheta}}}$, can in principle ``grow at any place''.
At the cost of properly shuffling the plane trees $(\T_n)_{n\geq1}$ given the history of the couplings $(S_{k,\u})_{0\leq k\leq {{N_\wtheta}}}$, ${\u\in\U}$, we can ``align'' the two growths and obtain a new coupling $(\subT_n)_{n\geq1}$ which is actually increasing.

\mypar{Shuffling decorated plane trees}
We specialize to the context of decorated plane trees the shuffling operations we defined in Section~\ref{subsec:moving-subtrees} for arbitrary rooted subtrees of $\U$, as follows.
For $k\geq0$, we let $\perm_k$ be the set of permutations of $\{1,2,\dots,k\}$, where $\perm_0$ is a singleton containing the empty permutation of the empty set.
We denote by $\perm$ the graded set 
\begin{align*}
\perm=\bigsqcup_{k\geq0}\perm_k.
\end{align*}
An element $\boldsigma$ of the set $\perm_\pt$ of grading-compatible $\pt$-tuples of elements of $\perm$ is then just a collection $\boldsigma=(\sigma_\u)_{\u\in\pt}$ such that $\sigma_\u\in\perm_{k_\u(\pt)}$ for all $\u\in\pt$.
Note that for every $\pt\in\trees$, we have that $\perm_T\subset\G(\pt)$.
Indeed, since $\pt$ is a plane tree, we have $\pos_\u(\pt)=\{1,2,\dots,k_\u(\pt)\}$ for all $\u\in\pt$, so that for every $\boldsigma=(\sigma_\u)_{\u\in\pt}\in\perm_\pt$ and every $\u\in\pt$, the permutation $\sigma_\u$ is indeed an injective mapping $\pos_\u(\pt)\to\{1,2,\dots,|\pos_\u(\pt)|\}$, that is an element of $\G_{\pos_\u(\pt)}$.

In particular, for $\pt\in\trees$ and $\boldsigma\in\perm_\pt$, and for every $\u=(u_1,\dots,u_h)\in \pt$ with ancestral line $\u_i=(u_1,\dots,u_i)$, ${0\leq i\leq h}$, we can consider as in Section~\ref{subsec:moving-subtrees},
\begin{align*}
\boldsigma\cdot\u
=\bigl(\sigma_\emptyset(u_1),\sigma_{\u_1}(u_2),\dots,\sigma_{\u_{h-1}}(u_h)\bigr).
\end{align*}
This is just a sub-case of the definition~\eqref{eq:def-action-g}, and we can directly import the relevant notation and use the properties proved in this section.

\begin{rem}\label{rem:from-g-to-sigma}
	For $\pt\in\trees$ and $\boldsigma\in\perm_\pt$, observe that $\boldsigma\cdot\pt$ is also a plane tree.
	More generally, given any  set $\bbX$, if $(\pt,\x)$ is an $\bbX$-decorated plane tree, then so is $(\boldsigma\cdot\pt,\boldsigma\pf\x)$ for every $\boldsigma\in\perm\pt$.
	Observe also that the ``inverse'' $\boldsigma^{-1}\in\G(\boldsigma\cdot\pt)$ given by Lemma~\ref{lem:inverse-g} is actually in $\perm_{\boldsigma\cdot\pt}$ for every $\boldsigma\in\perm_\pt$, $\pt\in\trees$.
\end{rem}

\mypar{Choosing the permutations}
Let us consider the set $\bbXseq$ of sequences made of $N_\wtheta$ pairwise distinct elements of $\{1,2,\dots\}$.
Given any  $x=(x_1,\dots,x_{N_\wtheta})\in\bbXseq$ and some $0\leq k\leq N_\wtheta$, we let
\begin{align}\label{eq:choice-sigma-of-g}
g_{k,x}(\ell)&=x_\ell,\qquad 1\leq \ell\leq k,
\end{align}
which defines an injective mapping $g_{k,x}\colon\{1,2,\dots,k\}\to\{x_1,\dots,x_k\}$ since by definition of $\bbXseq$, the elements $(x_1,\dots,x_{N_\wtheta})$ are pairwise distinct.
This means that $g_{k,x}\in\G_{\{1,2,\dots,k\}}$.
Then, we set
\begin{align}\label{eq:choice-sigma-of-x}
\sigma_{k,x}&= p_{\{x_1,\dots,x_k\}}\circ g_{k,x},
\end{align}
where $p_{\{x_1,\dots,x_k\}}\colon \{x_1,\dots,x_k\}\to\{1,2,\dots,k\}$ is the injective mapping which we defined in \eqref{eq:def-p_S}.
By composition of injective mappings, this defines is an injective mapping $\sigma_{k,x}\colon\{1,\dots,k\}\to\{1,\dots,k\}$, that is a permutation $\sigma_{k,x}\in\perm_k$.

\mypar{A better coupling strategy}
The construction begins as in the ``naive coupling''.
Consider again mutually independent copies
\begin{align*}
X_\u=(X_{1,\u},\dots,X_{N_\wtheta,\u}),\qquad\u\in\U,
\end{align*}
of the random sequence $X=(X_{1},\dots,X_{N_\wtheta})$ given by Proposition~\ref{prop:growing-subsets}.
Independently, let $(\T_n)_{n\geq1}$ be the coupling given by Corollary~\ref{cor:growing-underlying-plane-trees}.
As in the ``naive coupling'', for $n\geq1$, we let $\boldS_n=(S_{k_\u(\T_n),\u})_{\u\in\T_n}$, where $S_{k,\u}=\{X_{1,\u},\dots,X_{k,\u}\}$ for every $\u\in\T_n$ and every $0\leq k\leq N_\wtheta$.
Now comes the difference with the ``naive coupling''.
For every $n\geq1$, we let:
\begin{align*}
\boldsigma_n=(\sigma_{k_\u(\T_n),X_\u})_{\u\in\T_n},
\end{align*}
where the notation $\sigma_{k,x}$ is defined in \eqref{eq:choice-sigma-of-x}.
This gives an element $\boldsigma_n\in\perm_{\T_n}$.

\begin{lemma}\label{lem:identity-distrib-after-shuffling}
	In the above setting, the random pair $(\boldsigma_n\cdot\T_n,(\boldsigma_n)\pf\boldS_n)$ has the same joint distribution as $(\T_n,\boldS_n)$, for every $n\geq1$.
\end{lemma}

The proof of Lemma~\ref{lem:identity-distrib-after-shuffling} relies on some material collected in Appendix~\ref{app:appendix-shuffling}, more precisely Corollary~\ref{cor:shuffling-sg-with-iid}, which proves a result similar to Lemma~\ref{lem:identity-distrib-after-shuffling}, but for a more general class of ``decoration-dependent'' shuffling operations.

\begin{proof}[Proof of Lemma~\ref{lem:identity-distrib-after-shuffling}]
	For $\pt\in\trees$, for $\x=(x_\u)_{\u\in\pt}\in\bbXseq^\pt$, and for $\u\in\pt$, we let $\psi(\pt,\x,\u)=\sigma_{k_{\u}(\pt),x_\u}$, where the notation $\sigma_{k,x}$ is introduced in~\eqref{eq:choice-sigma-of-x}.
	Then for every $\boldpi\in\perm_\pt$, we have
	\begin{align*}
	\psi(\boldpi\cdot\pt,\boldpi\pf\x,\boldpi\cdot\u)
		=\bigl.\sigma_{k,x}\bigr|^{%
				k=k_{\boldpi\cdot\u}(\boldpi\cdot\pt)
			}_{%
				x=x_{\boldpi^{-1}\cdot(\boldpi\cdot\u)}
			}
		=\bigl.\sigma_{k,x}\bigr|^{%
				k=k_{\u}(\pt)
			}_{%
				x=x_{\u}
			}
		=\psi(\pt,\x,\u),
	\end{align*}
	where we used that $k_{\boldpi\cdot\u}(\boldpi\cdot\pt)=k_\u(\pt)$ by Lemma~\ref{lem:children-after-shuffling}.
	In the language of Appendix~\ref{app:appendix-shuffling}, this means that $\psi$ is a $\perm$-equivariant $\bbXseq$-shuffling rule.
	If we let $\nu$ be the distribution of the sequence $(X_1,\dots,X_{N_\wtheta})$ of Proposition~\ref{prop:growing-subsets}, then Corollary~\ref{cor:shuffling-sg-with-iid} gives the result, with these choices of $\psi$ and $\nu$. 
\end{proof}

\begin{cor}[Final coupling]\label{cor:final-coupling-subtrees}
	In the above setting, if one sets for $n\geq1$,
	\begin{align}\label{eq:expr-final-coupling}
	\subT'_n=\embed\bigl(\boldsigma_n\cdot\T_n,(\boldsigma_n)\pf\boldS_n\bigr),
	\end{align}
	then the sequence $(\subT'_n)_{n\geq1}$ is a coupling of the distributions $(\Sub{\wtheta}{n})_{n\geq 1}$.
\end{cor}

\begin{proof}
	Let $n\geq1$.
	By Lemma~\ref{lem:identity-distrib-after-shuffling}, the pairs $(\boldsigma_n\cdot\T_n,(\boldsigma_n)\pf\boldS_n)$ and $(\T_n,\boldS_n)$ have the same distribution, so that $\subT'_n=\embed(\boldsigma_n\cdot\T_n,(\boldsigma_n)\pf\boldS_n)$ has the same distribution as $\subT_n=\embed\bigl(\T_n,\boldS_n\bigr)$, which is $\Sub{\wtheta}{n}$ by Proposition~\ref{prop:naive-coupling}.
\end{proof}

\subsection{Proof of Theorem~\ref{thm:growing-subtrees}}

We first prove two deterministic lemmas.
Let $(x_\u)_{\u\in\U}$ be an arbitrary \textit{fixed} collection of elements of $\bbXseq$, where $x_\u=(x_{1,\u},x_{2,\u},\dots,x_{N_\wtheta,\u})$ for every $\u\in\U$.
For every plane tree $\pt\in\trees$, we set:
\begin{align}\label{eq:notation-before-lem-expr-shuffling}
\boldS_{\pt}=(S_{k_\u(\pt),\u})_{\u\in\pt},\qquad
\boldsigma_{\pt}=(\sigma_{k_\u(\pt),x_\u})_{\u\in\pt},\qquad
\boldg_{\pt}=(g_{k_\u(\pt),x_\u})_{\u\in\pt},
\end{align}
where the notation $g_{k,x}$ and $\sigma_{k,x}$ are defined in \eqref{eq:choice-sigma-of-x}, and where we write $S_{k,\u}=\{x_{1,\u},\dots,x_{k,\u}\}$ for $\u\in\pt$ and $0\leq k\leq N_\wtheta$.

\begin{lemma}\label{lem:expr-after-shuffling}
	Let $\pt\in\trees$.
	For every $\u\in\pt$, we have
	\begin{align*}
	\overline{p_{(\boldsigma_{\pt})\pf\boldS_{\pt}}}\cdot\bigl(\boldsigma_{\pt}\cdot\u\bigr)
		=\boldg_{\pt}\cdot\u.
	\end{align*}
\end{lemma}

\begin{proof}
	In order to lighten notation, let us write $S_\u=S_{k_\u(\pt),\u}$, and $\sigma_\u=\sigma_{k_\u(\pt),x_\u}$, as well as $g_\u=g_{k_\u(\pt),x_\u}$, for every $\u\in\pt$.
	Notice that by \eqref{eq:choice-sigma-of-x}, we have 
	\begin{align}\label{eq:proof-expr-after-shuffling-0}
	g_\u=p^{-1}_{S_\u}\circ\sigma_\u,\qquad\u\in\pt.
	\end{align}
	For every $\u'=(u'_1,\dots,u'_{h})\in\boldsigma_{\pt}\cdot\pt$ with ancestral line $(\u'_\ell)_{0\leq\ell\leq h}$, we have by definition:
	\begin{align}\label{eq:proof-expr-after-shuffling-1}
	\overline{p_{(\boldsigma_{\pt})\pf\boldS_{\pt}}}\cdot\u'
		=	\Bigl(
				p^{-1}_{S_{\boldsigma_{\pt}^{-1}\cdot\u'_0}}(u'_1),
				p^{-1}_{S_{\boldsigma_{\pt}^{-1}\cdot\u'_1}}(u'_2),
				\dots,
				p^{-1}_{S_{\boldsigma_{\pt}^{-1}\cdot\u'_{h-1}}}(u'_{h})
			\Bigr).
	\end{align}
	Let $\u=(u_1,\dots,u_h)\in\pt$, with ancestral line $(\u_\ell)_{0\leq\ell\leq h}$.
	We have by definition:
	\begin{align}\label{eq:proof-expr-after-shuffling-2}
	\boldsigma_{\pt}\cdot\u
		=\Bigl(
			\sigma_{\u_0}(u_1),
			\sigma_{\u_1}(u_2),
			\dots,
			\sigma_{\u_{h-1}}(u_h)
		\Bigr).
	\end{align}
	If we combine \eqref{eq:proof-expr-after-shuffling-1} and \eqref{eq:proof-expr-after-shuffling-2} with $\u'=\boldsigma_{\pt}\cdot\u$, whose ancestral line in $\boldsigma_{\pt}\cdot\pt$ is given by $\u'_\ell=\boldsigma_{\pt}\cdot\u_\ell$, $0\leq\ell\leq h$, we obtain:
	\begin{align*}
	\overline{p_{(\boldsigma_{\pt})\pf\boldS_{\pt}}}\cdot(\boldsigma_{\pt}\cdot\u)
		&=	\Bigl(
				p^{-1}_{S_{\u_0}}\bigl(\sigma_{\u_0}(u_1)\bigr),\,
				p^{-1}_{S_{\u_1}}\bigl(\sigma_{\u_0}(u_2)\bigr),
				\dots,\,
				p^{-1}_{S_{\u_{h-1}}}\bigl(\sigma_{\u_0}(u_{h})\bigr)
			\Bigr)\\
		&=	\Bigl(
				g_{\u_0}(u_1),\,
				g_{\u_1}(u_2),
				\dots,\,
				g_{\u_{h-1}}(u_{h})
			\Bigr)\\
		&=	\boldg_{\pt}\cdot\u,
	\end{align*}
	where the second equality uses \eqref{eq:proof-expr-after-shuffling-0}.
	This concludes the proof.
\end{proof}

\begin{lemma}\label{lem:increase-after-shuffling}
	If ${\pt}$ and ${\widetilde\pt}$ be two plane trees such that ${\pt}\subseteq{\widetilde\pt}$, then we have
	\begin{align*}
	\embed
		\Bigl(
			\boldsigma_{{\pt}}\cdot{\pt},
			(\boldsigma_{{\pt}})\pf\boldS_{{\pt}}
		\Bigr)
	\subseteq
	\embed
		\Bigl(
			\boldsigma_{{\widetilde\pt}}\cdot{\widetilde\pt},
			(\boldsigma_{{\widetilde\pt}})\pf\boldS_{{\widetilde\pt}}
		\Bigr).
	\end{align*}
\end{lemma}

\begin{proof}
	By \eqref{eq:notation-before-lem-expr-shuffling}, we have $\boldg_{\pt}=(g_{k_\u(\pt),x_\u})_{\u\in\pt}$ for every $\pt\in\trees$.
	Thanks to how we defined $g_{k,x}$ in \eqref{eq:choice-sigma-of-g}, we have for every $\pt\in\trees$ and every $\u\in\pt$ with ancestral line $(\u_\ell)_{0\leq\ell\leq h}$ the expression:
	\begin{align*}
	\boldg_{\pt}\cdot\u
		=\bigl(x_{\u_0}(u_1),x_{\u_1}(u_2),\dots,x_{\u_{h-1}}(u_h)\bigr),
	\end{align*}
	where for clarity we have written $x_\u(k)$ as a substitute for the notation $x_{k,\u}$.
	Note that the right-hand side does not depend on $\pt$.
	In particular, given two plane trees ${\pt}$ and ${\widetilde\pt}\in\trees$ such that ${\pt}\subseteq{\widetilde\pt}$, we have $\boldg_{{\pt}}\cdot\u=\boldg_{{\widetilde\pt}}\cdot\u$ for every $\u\in{\pt}$.
	Since by assumption we have $\smash{{\pt}\subseteq{\widetilde\pt}}$, we get:
	\begin{align*}
	\boldg_{{\pt}}\cdot{\pt}=\boldg_{{\widetilde\pt}}\cdot{\pt}\subseteq\boldg_{{\widetilde\pt}}\cdot{\widetilde\pt}.
	\end{align*}
	Using Lemma~\ref{lem:expr-after-shuffling} and the expression $\embed\colon(\pt,\boldS)\mapsto\overline{p_\boldS}\cdot\pt$ of Proposition~\ref{prop:bijection-push}, this means that
	\begin{align*}
	\embed
	\Bigl(
	\boldsigma_{{\pt}}\cdot{\pt},
	(\boldsigma_{{\pt}})\pf\boldS_{{\pt}}
	\Bigr)
	\subseteq
	\embed
	\Bigl(
	\boldsigma_{{\widetilde\pt}}\cdot{\widetilde\pt},
	(\boldsigma_{{\widetilde\pt}})\pf\boldS_{{\widetilde\pt}}
	\Bigr),
	\end{align*}
	which concludes the proof.
\end{proof}

We are now equipped to prove Theorem~\ref{thm:growing-subtrees}.
It is sufficient to verify that the coupling $(\subT'_n)_{n\geq1}$ of the distributions $(\Sub{\wtheta}{n})_{n\geq 1}$ given by Corollary~\ref{cor:final-coupling-subtrees} is indeed increasing.
The Markov property we refer to in Theorem~\ref{thm:growing-subtrees} can then be enforced by successively resampling $\subT'_{n+1}$ conditionally on $\subT'_n$ for $n=1,2,\dots$.

We have to verify that $\subT'_n\subset\subT'_{n+1}$ for all $n\geq1$, where $(\subT'_n)_{n\geq1}$ are given by the expression \eqref{eq:expr-final-coupling}, that is:
\begin{align*}
\subT'_n=\embed\bigl(\boldsigma_n\cdot\T_n,(\boldsigma_n)\pf\boldS_n\bigr),\qquad n\geq1.
\end{align*}
If we reason conditionally on the collection $(X_\u)_{\u\in\U}$, thus treating it as a deterministic collection $(x_\u)_{\u\in\U}$, then with the notation~\eqref{eq:notation-before-lem-expr-shuffling} we have $\boldS_n=\boldS_{\T_n}$, and $\boldsigma_n=\boldsigma_{\T_n}$. 
In particular, Lemma~\ref{lem:increase-after-shuffling} gives that $\subT'_n\subset\subT'_{n+1}$ for every $n\geq1$, as needed.
This concludes the proof of Theorem~\ref{thm:growing-subtrees}.\qed

%% file: parts/appendix-shuffling.tex
Given a probability space $(\bbX,\mathcal X,\nu)$ and some non-negative sequence $\w=(w_0,w_1,\dots)$ with $w_0w_1>0$, the probability distribution $\SimpGen{\w}{n}(\mathrm{d}\pt)\cdot\nu^{\otimes\pt}(\mathrm{d}\x)$ is morally left invariant by ``shuffling'' the subtrees of descendants at every vertex.
Such a symmetry has been used for instance by Addario-Berry and Albenque in \cite{Addario-BerryAlbenque21}, where the shuffling operation is performed uniformly at random.
The main goal of this section is to arrive at Corollary~\ref{cor:shuffling-sg-with-iid}, which states that the measure $\SimpGen{\w}{n}(\mathrm{d}\pt)\cdot\nu^{\otimes\pt}(\mathrm{d}\x)$ is left invariant by some more general ``shuffling operations'' which are allowed to depend on the tree and its decorations.

As in Section~\ref{subsec:coupling-using-shuffling}, for $\pt\in\trees$, we denote by $\perm_\pt$ the set of collections $\boldsigma=(\sigma_\u)$, where $\sigma_\u$ is a permutation belonging to $\perm_{k_\u(\pt)}$ for every $\u\in\pt$.
Also, as in Sections~\ref{subsec:moving-subtrees} and~\ref{subsec:coupling-using-shuffling}, for every $\u=(u_1,\dots,u_h)\in \pt$ with ancestral line $\u_\ell=(u_1,\dots,u_\ell)$, ${0\leq i\leq h}$, we let
\begin{align*}
\boldsigma\cdot\u
=\bigl(\sigma_\emptyset(u_1),\sigma_{\u_1}(u_2),\dots,\sigma_{\u_{h-1}}(u_h)\bigr),
\end{align*}
and $\boldsigma\cdot\pt=\{\boldsigma\cdot\u\colon\u\in\pt\}$.
We have proven in Section~\ref{sec:application-random-subtrees}, see Lemma~\ref{lem:inverse-g} and Remark~\ref{rem:from-g-to-sigma}, that there exists $\boldsigma^{-1}\in\perm_{\boldsigma\cdot\pt}$ such that $\boldsigma^{-1}\cdot(\boldsigma\cdot\u)=\u$ for all $\u\in \t$, and $\boldsigma\cdot(\boldsigma^{-1}\cdot\u')=\u'$ for all $\u'\in\boldsigma\cdot\t$.
Given a set $\bbX$, this allows to define a push-forward mapping $\boldsigma\pf\colon\bbX^\pt\to\bbX^{\boldsigma\cdot\pt}$ by setting $\boldsigma\pf\x=(x_{\boldsigma^{-1}\cdot\u'})_{\u'\in\boldsigma\cdot\pt}$.
Other related properties have been established in Section~\ref{subsec:moving-subtrees}, which we will refer to when necessary.

\mypar{Shuffling rules}
Let $\bbX$ be a fixed set.
We consider the set of \textit{pointed} $\bbX$-decorated plane trees:
\begin{align*}
\trees^{\bullet}[\bbX]
=\{(\pt,\x,\u)\colon \pt\in\trees,\,\x\in\bbX^T,\,\u\in\pt\}.
\end{align*}

\begin{defin}
	An $\bbX$-\textit{shuffling rule} is a mapping $\psi\colon\trees^{\bullet}[\bbX]\to\perm$ such that for every $(\pt,\x,\u)\in\trees^{\bullet}[\bbX]$, its image $\psi(\pt,\x,\u)$ is an element of $\perm_{k_\u(\pt)}$.
	Given an $\bbX$-shuffling rule, for every  $\pt\in\trees$ and every $\x\in\bbX^T$, we set:
	\begin{align*}
	\boldsigma^\psi_{\pt,\x}=\bigl(\psi(\pt,\x,\u)\bigr)_{\u\in\pt},
	\end{align*}
	which defines an element of $\perm_\pt$ 
\end{defin}

Without additional conditions on $\psi$, there is no guarantee that the mapping $\pt\mapsto (\boldsigma^\psi_{\pt,\x})\cdot\pt$ is even bijective.
For instance, we can choose $\psi$ so that the action of $\boldsigma^\psi_{\pt,\x}$ on the tree $\pt$ is to switch the positions of the leftmost, resp.~the biggest, subtree of descendants of the root.
We then cannot ``go back'' since the information of where the biggest subtree was is lost.
We therefore introduce a natural condition on $\bbX$-shuffling rules, as follows.

\begin{defin}
	Let $\psi$ be an $\bbX$-shuffling rule.
	We say that $\psi$ is \textit{$\perm$-equivariant}%
	\footnote{
		The motivation for this name is the equivalence \eqref{eq:conseq-perm-invariance-psi}, by analogy with the notion of \textit{equivariance} of some mapping with respect to some group action, which means commutation of that mapping with the group action.
	}
	if for every $(\pt,\x,\u)\in\trees^{\bullet}[\bbX]$ and every $\boldsigma\in\perm_T$ we have:
	\begin{align*}
	\psi(\pt,\x,\u)=\psi\bigl(\boldsigma\cdot\pt,\,\boldsigma\pf\x,\,\boldsigma\cdot\u\bigr).
	\end{align*}
\end{defin}
Observe that given $\psi$ an $\bbX$-shuffling rule, we have the equivalence:
\begin{align}\label{eq:conseq-perm-invariance-psi}
\text{$\psi$ is $\perm$-equivariant}
\qquad\iff\qquad
\forall\boldpi\in\perm_\pt,\quad
\boldsigma^\psi_{\boldpi\cdot\pt,\boldpi_*\x}=\boldpi \pf\bigl(\boldsigma^\psi_{\pt,\x}\bigr).
\end{align}
The following lemma elucidates the structure of the ``reverse'' shuffling operation.

\begin{lemma}[Unshuffling lemma]\label{lem:unshuffling-lemma}
	Let $\bbX$ be a set and let $\psi$ be an $\bbX$-shuffling rule.
	If $\psi$ is $\perm$-equivariant, then for all $\pt,\pt'\in\trees$, $\x\in\bbX^\pt$, $\x'\in\bbX^{\pt'}$, $\boldpi\in\perm_\pt$, we have the equivalence
	\begin{align}\label{eq:equivalence-unshuffling-lemma}
	\begin{cases}
	\pt'	= \bigl(\boldsigma^\psi_{\pt,\x}\bigr)\cdot\pt\\
	\x' = \bigl(\boldsigma^\psi_{\pt,\x}\bigr)\pf\x
	\end{cases}
	\quad\iff\quad
	\begin{cases}
	\pt	=	\bigl(\overline{\boldsigma}^\psi_{\pt',\x'}\bigr)\cdot\pt'\\
	\x = \bigl(\overline{\boldsigma}^\psi_{\pt',\x'}\bigr)\pf\x',
	\end{cases}
	\end{align}
	and if one of these equivalent conditions is satisfied then $\boldsigma^\psi_{\pt',\x'}=(\overline\boldsigma^\psi_{\pt,\x})^{-1}$.
\end{lemma}

\begin{proof}
	Suppose that the left-hand side of \eqref{eq:equivalence-unshuffling-lemma} holds.
	Since $\psi$  is $\perm$-equivariant, if we let $\boldpi=\boldsigma^\psi_{\pt,\x}$, then by \eqref{eq:conseq-perm-invariance-psi}, we have $\boldsigma^\psi_{\boldpi\cdot\pt,\boldpi_*\x}=\boldpi \pf\bigl(\boldsigma^\psi_{\pt,\x}\bigr)$, that is,
	\begin{align*}
	\boldsigma^\psi_{\pt',\x'}=\bigl(\boldsigma^\psi_{\pt,\x}\bigr) \pf\bigl(\boldsigma^\psi_{\pt,\x}\bigr).
	\end{align*}
	By Lemma~\ref{lem:commutation-bar-pushforward-g}, this yields that
	\begin{align*}
	\overline\boldsigma^\psi_{\pt',\x'}=\bigl(\boldsigma^\psi_{\pt,\x}\bigr) \pf\bigl(\overline\boldsigma^\psi_{\pt,\x}\bigr).
	\end{align*}
	By Lemma~\ref{lem:other-expression-inverse-g}, this means that $\overline\boldsigma^\psi_{\pt',\x'}=\bigl(\boldsigma^\psi_{\pt,\x}\bigr)^{-1}$, which gives $\pt=\bigl(\overline\boldsigma^\psi_{\pt',\x'}\bigr)\cdot\pt'$ and $\x=\bigl(\overline\boldsigma^\psi_{\pt',\x'}\bigr)\pf\x'$ using
	Lemmas~\ref{lem:inverse-g} and~\ref{lem:inverse-pushforward-g}.
	Hence the left-hand side implies the right-hand side in \eqref{eq:equivalence-unshuffling-lemma}.
	The converse implication is proven similarly.
	We have already justified above that $\boldsigma^\psi_{\pt',\x'}=(\overline\boldsigma^\psi_{\pt,\x})^{-1}$ when one of the two equivalent conditions is satisfied.
\end{proof}

For purely technical reasons, we will need to rephrase Lemma~\ref{lem:unshuffling-lemma} as follows.

\begin{cor}\label{cor:unshuffling-cor}
	Let $\bbX$ be a set and let $\psi$ be an $\bbX$-shuffling rule.
	If $\psi$ is $\perm$-equivariant, then for every $\pt,\pt'\in\trees$, $\x\in\bbX^\pt$, and $\boldpi\in\perm_{\pt'}$, we have
	\begin{align}\label{eq:equi-cor-unshuffling-lemma}
	\begin{cases}
	\boldsigma^\psi_{\pt,\x}\cdot\pt=\pt',\\
	\boldsigma^\psi_{\pt,\x}=\boldpi^{-1}
	\end{cases}
	\qquad\iff\qquad
	\begin{cases}
	\pt=\boldpi\cdot\pt',\\
	\boldsigma^\psi_{\pt',\boldpi^{-1}_*\x}=\overline\boldpi.
	\end{cases}
	\end{align}
\end{cor}

\begin{proof}
	Let $\pt,\pt'\in\trees$, $\x\in\bbX^\pt$, and $\boldpi\in\perm_{\pt'}$.
	Then, Lemma~\ref{lem:unshuffling-lemma} applied with $\x'=\pi^{-1}_*\x$ gives the equivalence:
	\begin{align*}
	\begin{cases}
	(\boldsigma^\psi_{\pt,\x})\cdot\pt=\pt',\\
	\bigl(\boldsigma^\psi_{\pt,\x}\bigr)\pf\x=\boldpi^{-1}_*\x\\
	\boldsigma^\psi_{\pt,\x}=\boldpi^{-1}
	\end{cases}
	\iff
	\begin{cases}
	\pt= (\overline\boldsigma^\psi_{\pt',\boldpi^{-1}_*\x})\cdot\pt',\\
	\x=\bigl(\overline\boldsigma^\psi_{\pt',\boldpi^{-1}_*\x}\bigr)\pf(\boldpi^{-1}_*\x)\\
	\boldsigma^\psi_{\pt',\boldpi^{-1}_*\x}=\overline\boldpi.
	\end{cases}
	\end{align*}
	But the left/right-hand sides of the last display are clearly equivalent to the left/right-hand sides of \eqref{eq:equi-cor-unshuffling-lemma} respectively.
	In particular, the last display expresses the same equivalence as \eqref{eq:equi-cor-unshuffling-lemma}.
\end{proof}

\mypar{Invariance under shuffling}
Given a measurable space $(\bbX,\Xcal)$, we turn $\trees{[\bbX]}$ into a measurable space by endowing it with the $\sigma$-algebra generated by the sets $\{\pt\}\times A$ for $\pt\in\trees$ and $A\in\Xcal^{\otimes \pt}$.
We similarly endow $\trees^\bullet{[\bbX]}$ with the $\sigma$-algebra generated by the sets $\{\pt\}\times A\times\{\u\}$ for $\pt\in\trees$, $A\in\Xcal^{\otimes \pt}$, and $\u\in\pt$.

\begin{defin}
	A collection $(\mfrak_\pt)_{\pt\in\trees}$ of finite measures on $\bbX^T$, $\pt\in\trees$, respectively will be called \textit{$\perm$-consistent} if the image measure with respect to the mapping $\x\mapsto\boldsigma\pf\x$ of the measure $\mfrak_\pt$ is the measure $\mfrak_{\boldsigma\cdot\pt}$, for every $\pt\in\trees$ and every $\boldsigma\in\perm_\pt$.
\end{defin}

This notion of $\perm$-consistency allows to generate invariant measures on the set of decorated trees, which are invariant with respect to our ``decoration-dependent'' shuffling operations.

\begin{prop}\label{prop:shuffling-general-m}
	Let $(\bbX,\mathcal X)$ be a measurable space.
	Let $\psi$ be a measurable $\bbX$-shuffling rule, and let $(\mfrak_\pt)_{\pt\in\trees}$ be a collection  of finite measures on $\bbX^T$ respectively.
	If $\psi$ is $\perm$-equivariant  and $(\mfrak_\pt)_{\pt\in\trees}$ is $\perm$-consistent, then the measure $\mfrak$ on $\trees{[\bbX]}$ given by
	\begin{align}\label{eq:form-m-prop-shuffling-inv}
	\mfrak(\diff\pt,\diff\x)=\sum_{\pt_0\in\trees}\delta_{\pt_0}(\diff\pt)\,\mfrak_{\pt_0}(\diff\x)
	\end{align}
	is invariant%
	\footnote{
		We say that a measure $m$ is invariant under $f$ if we have $f_\#\,m=m$.
	}
	under the mapping $F^\psi\colon(\pt,\x)\mapsto \bigl((\boldsigma^\psi_{\pt,\x})\cdot\pt,(\boldsigma^\psi_{\pt,\x})\pf\x\bigr)$.
\end{prop}

\begin{proof}
	Let $\pt'\in\trees$ and let $\phi$ be a positive and measurable function on $\bbX^\pt$, and let $\Phi\colon\trees{[\bbX]}\to\R$ be the positive and measurable function $(\pt,\x)\mapsto\indic{\{\pt=\pt'\}}\phi(\x)$.
	Such functions suffice to characterize measures on $\trees{[\bbX]}$, so that we only need to check that $\int \Phi\circ F^\psi\diff \mfrak=\int \Phi\diff \mfrak$.
	Notice that
	\begin{align}\label{eq:proof-shuffling-general-m-1}
	\int_{\trees{[\bbX]}} \Phi\circ F^\psi\diff \mfrak
	=\int_{\trees{[\bbX]}}\mfrak(\diff\pt,\diff\x)\,
	\indic{\left\{
		\boldsigma^\psi_{\pt,\x}\cdot\pt=\pt'
		\right\}}\cdot
	\phi\bigl(	(\boldsigma^\psi_{\pt,\x})_*\x\bigr)
	=\sum_{\boldpi\in\perm_{\pt'}}I(\boldpi),
	\end{align}
	where for $\boldpi\in\perm_{\pt'}$, we denote
	\begin{align*}
	I(\boldpi)=
	\int_{\trees{[\bbX]}}\mfrak(\diff\pt,\diff\x)\,
	\indic{\left\{
		\boldsigma^\psi_{\pt,\x}\cdot\pt=\pt'
		\right\}}\cdot
	\indic{\left\{
		\boldsigma^\psi_{\pt,\x}=\boldpi^{-1}
		\right\}}\cdot
	\phi\bigl(	\boldpi^{-1}_*\x\bigr).
	\end{align*}
	Let $\boldpi\in\perm_{\pt'}$.
	By Corollary~\ref{cor:unshuffling-cor}, we have
	\begin{align*}
	I(\boldpi)
		&=
			\int_{\trees{[\bbX]}}\mfrak(\diff\pt,\diff\x)\,
			\indic{\left\{
				\pt=\boldpi\cdot\pt'			
				\right\}}\cdot
			\indic{\left\{
				\boldsigma^\psi_{\pt',\smash{\boldpi^{-1}_*\x}}=\overline\boldpi
				\right\}}\cdot
			\phi\bigl(	\boldpi^{-1}_*\x\bigr)\\
		&=	\int_{\bbX^{\boldpi\cdot\pt'}}\mfrak_{\boldpi\cdot\pt'}(\diff\x)\,
			\indic{\left\{
				\boldsigma^\psi_{\pt',\boldpi^{-1}_*\x}=\overline\boldpi
				\right\}}\cdot
			\phi\bigl(	\boldpi^{-1}_*\x\bigr),
	\end{align*}
	where the last equality uses that $\mfrak$ takes the form \eqref{eq:form-m-prop-shuffling-inv}.
	Since by $\perm$-consistency, the measure $\mfrak_{\pt'}$  is the push-forward with respect to $\x\mapsto\boldpi^{-1}_*\x$ of the measure $\mfrak_{\boldpi\cdot\pt'}$, we obtain
	\begin{align*}
	I(\boldpi)
	=
	\int_{\bbX^{\pt'}}\mfrak_{\pt'}(\diff\x')\,
	\indic{\left\{
		\boldsigma^\psi_{\pt',\x'}=\overline\boldpi
		\right\}}\cdot
	\phi\bigl(	\x'\bigr).
	\end{align*}
	By summing over all $\boldpi\in\perm_{\pt'}$ and using \eqref{eq:proof-shuffling-general-m-1}, we get
	\begin{align*}
	\int_{\trees{[\bbX]}} \Phi\circ F^\psi\diff \mfrak
	&= \sum_{\boldpi\in\perm_{\pt'}}\,
	\int_{\bbX^{\pt'}}\mfrak_{\pt'}(\diff\x')\,
	\indic{\{
		\boldsigma^\psi_{\pt',\x'}=\overline\boldpi
		\}}\cdot
	\phi\bigl(	\x'\bigr)\\
	&= \int_{\bbX^{\pt'}}\mfrak_{\pt'}(\diff\x')\,
	\phi\bigl(	\x'\bigr)
	\end{align*}
	Since $\mfrak$ takes the form \eqref{eq:form-m-prop-shuffling-inv}, the right-hand side equals $\int \Phi\diff \mfrak$.
	Hence we have $\int \Phi\circ F^\psi\diff \mfrak=\int \Phi\diff \mfrak$, as needed to conclude.
\end{proof}

\begin{lemma}\label{lem:iid-is-consistent}
	Let $(\bbX,\mathcal X,\nu)$ be a probability space.
	The collection $(\nu^{\otimes\pt})_{\pt\in\trees}$ of probability measures is $\perm$-consistent.
\end{lemma}

\begin{proof}
	Let $\pt\in\trees$ and $\boldsigma\in\perm_\pt$.
	For every collection $(\phi_{\u'})_{\u'\in\boldsigma\cdot\pt}$ of bounded measurable functions on $\bbX$, we have
	\begin{align*}
	\int\nu^{\otimes\pt}(\diff\x)
	\prod_{\u'\in\boldsigma\cdot\pt}
	\phi_{\u'}(x_{\boldsigma^{-1}\cdot\u'})
	&=	\prod_{\u'\in\boldsigma\cdot\pt}
	\int g_{\u'}\diff\nu \\
	&=	\int\nu^{\otimes(\boldsigma\cdot\pt)}(\diff\x')
	\prod_{\u'\in\boldsigma\cdot\pt}\phi_{\u'}(x'_{\u'}),
	\end{align*}
	which gives that the image measure of $\nu^{\otimes\pt}$ with respect to the mapping $\x\mapsto\boldsigma\pf\x$ is the measure $\nu^{\otimes(\boldsigma\cdot\pt)}$.
	Hence $(\nu^{\otimes\pt})_{\pt\in\trees}$ is $\perm$-consistent.
\end{proof}

\begin{cor}\label{cor:shuffling-sg-with-iid}
	Let $(\bbX,\mathcal X,\nu)$ be a probability space, and let $\psi$ be a measurable $\bbX$-shuffling rule.
	If $\psi$ is $\perm$-equivariant, then for every non-negative sequence $\w=(w_0,w_1,\dots)$ with $w_0w_1>0$, and for every $n\geq1$, the probability measure
	\begin{align*}
	\SimpGen{\w}{n}(\mathrm{d}\pt)\cdot\nu^{\otimes\pt}(\mathrm{d}\x)
	\end{align*}
	is invariant under the mapping $F^\psi\colon(\pt,\x)\mapsto \bigl((\boldsigma^\psi_{\pt,\x})\cdot\pt,(\boldsigma^\psi_{\pt,\x})\pf\x\bigr)$.
\end{cor}

\begin{proof}
	Fix $n\geq1$.
	For $\pt\in\trees$, we let $\mfrak_\pt=\SimpGen{\w}{n}(\pt)\cdot\nu^{\otimes\pt}(\mathrm{d}\x)$.
	By Proposition~\ref{prop:shuffling-general-m}, it suffices to verify that the collection $(\mfrak_\pt)_{\pt\in\trees}$ is $\perm$-consistent.
	Let $\pt\in\trees$ and $\boldsigma\in\perm_\pt$.
	We deduce from Lemma~\ref{lem:children-after-shuffling} and from the definition~\eqref{eq:def-SG} of the distribution $\SimpGen{\w}{n}(\pt)$ that we have $\SimpGen{\w}{n}(\pt)=\SimpGen{\w}{n}(\boldsigma\cdot\pt)$.
	Combining this with Lemma~\ref{lem:iid-is-consistent}, we get that the image under $\x\mapsto\boldsigma\pf\x$ of $\mfrak_\pt$ is $\mfrak_{\boldsigma\cdot\pt}$, and therefore the collection $(\mfrak_\pt)_{\pt\in\trees}$ is $\perm$-consistent, as needed.
\end{proof}